\documentclass[hidelinks,onefignum,onetabnum]{siamart251216}

\usepackage{cite}
\usepackage{afterpage}
\usepackage{jabbrv}
\DefineSpuriousJournalWord{on}
\DefineSpuriousJournalWord{Its}
\usepackage{amsmath,amsfonts,amscd,amssymb}
\usepackage{mathtools}
\usepackage{mathrsfs}
\usepackage{graphicx}
\usepackage{epstopdf}
\usepackage{algorithmic}
\usepackage{enumitem}
\usepackage{booktabs}
\usepackage[normalem]{ulem}
\ifpdf
  \DeclareGraphicsExtensions{.eps,.pdf,.png,.jpg}
\else
  \DeclareGraphicsExtensions{.eps}
\fi

\def\thick{0.8}
\usepackage{xcolor}

\newcommand*{\dd}{{\,\mathrm{d}}}
\newcommand{\koop}{\mathcal{K}}

\newcommand{\Lv}{\mathbf{L}}

\newcommand{\Wv}{\mathbf{W}}

\newcommand{\Psiv}{\mathbf{\Psi}}

\DeclareMathOperator*{\argmin}{argmin}

\usepackage{comment}

\newsiamremark{remark}{Remark}
\newsiamremark{hypothesis}{Hypothesis}
\crefname{hypothesis}{Hypothesis}{Hypotheses}
\newsiamthm{claim}{Claim}
\newsiamremark{fact}{Fact}
\crefname{fact}{Fact}{Facts}

\headers{Trustworthy Koopman Operator Learning}{G. Conradie, N. Boull\'e, J-C. Loiseau,  S. L. Brunton, M. J. Colbrook,}

\title{Trustworthy Koopman Operator Learning:\\Invariance Diagnostics and Error Bounds}

\author{Gustav Conradie\thanks{Centre for Mathematical Sciences, University of Cambridge, UK 
  (\email{gjc51@cam.ac.uk}, \email{mjc249@cam.ac.uk})}
\and Nicolas Boull\'e\thanks{Department of Mathematics, Imperial College London, UK (\email{n.boulle@imperial.ac.uk})}
  \and Jean-Christophe Loiseau\thanks{Laboratoire DynFluid, Arts et Métiers Institute of Technology, France (\email{jean-christophe.loiseau@ensam.eu})}
  \and ~~~~Steven L. Brunton\thanks{Department of Mechanical Engineering, University of Washington, USA
  (\email{sbrunton@uw.edu})}
    \and Matthew J. Colbrook\footnotemark[1]}

\usepackage{amsopn}

\begin{document}

\maketitle

\begin{abstract}
Koopman operator theory provides a global linear representation of nonlinear dynamics and underpins many data-driven methods. In practice, however, finite-dimensional feature spaces induced by a user-chosen dictionary are rarely invariant, so closure failures and projection errors lead to spurious eigenvalues, misleading Koopman modes, and overconfident forecasts. This paper addresses a central validation problem in data-driven Koopman methods: how to quantify invariance and projection errors for an arbitrary feature space using only snapshot data, and how to use these diagnostics to produce actionable guarantees and guide dictionary refinement? A unified a posteriori methodology is developed for certifying when a Koopman approximation is trustworthy and improving it when it is not. Koopman invariance is quantified using principal angles between a subspace and its Koopman image, yielding principal observables and a principal angle decomposition (PAD), a dynamics-informed alternative to SVD truncation with significantly improved performance. Multi-step error bounds are derived for Koopman and Perron--Frobenius mode decompositions, including RKHS-based pointwise guarantees, and are complemented by Gaussian process expected error surrogates. The resulting toolbox enables validated spectral analysis, certified forecasting, and principled dictionary and kernel learning, demonstrated on chaotic and high-dimensional benchmarks and real-world datasets, including cavity flow and the Pluto--Charon system.
\end{abstract}



\setcounter{tocdepth}{2}
\tableofcontents
\linespread{0.97}
\section{Introduction}
Over the past decade, data-driven modeling of dynamical systems has become increasingly important, particularly when governing equations are unknown, incomplete, or prohibitively expensive to resolve at application-relevant fidelity~\cite{Brunton2022book,karniadakis2021physics,mezic2013analysis}. A common objective is to extract \emph{structure}: low-complexity representations that are interpretable, transferable across regimes, and amenable to analysis and control. Recent advances in sparse and symbolic regression \cite{Bongard2007pnas,Schmidt2009science,cranmer2023interpretable,brunton2016discovering} exemplify this emphasis on parsimony and interpretability. Another foundational objective is \emph{coordinate discovery}: identifying transformations under which the dynamics become simpler, or ideally admit an accurate linear representation.

Koopman operator theory provides a principled route to such linearization~\cite{mezic2005spectral,budivsic2012applied,brunton2021modern}. Rather than evolving the state forward by a nonlinear map, Koopman theory studies the evolution of \emph{observables} (functions of the state). The resulting Koopman operator is always linear (and typically infinite-dimensional), and its spectral structure underpins the Koopman mode decomposition (KMD) used for model reduction, forecasting, and control. Standard approximation methods, such as Dynamic Mode Decomposition (DMD)~\cite{schmid2010dynamic,rowley2009spectral,kutz2016dynamic,colbrook2023multiverse} and Extended Dynamic Mode Decomposition (EDMD)~\cite{williams2015data,williams2015kernel}, can be viewed as Galerkin projections of the infinite-dimensional Koopman operator onto a finite-dimensional subspace of observables. These data-driven algorithms have proved highly successful, largely due to their simple formulation in terms of numerical linear algebra~\cite{tu2014dynamic} and their flexibility and extensibility~\cite{kutz2016dynamic,schmid2022dynamic}. Notable successes span a wide range of applications, including robot control~\cite{haggerty2023control}, climate dynamics~\cite{froyland2021spectral}, training neural networks~\cite{orvieto2023resurrecting}, epidemiological modeling~\cite{proctor2015discovering}, and neuroscience~\cite{brunton2016extracting}.

\begin{figure}[t]
    \centering
    \includegraphics[width=\textwidth]{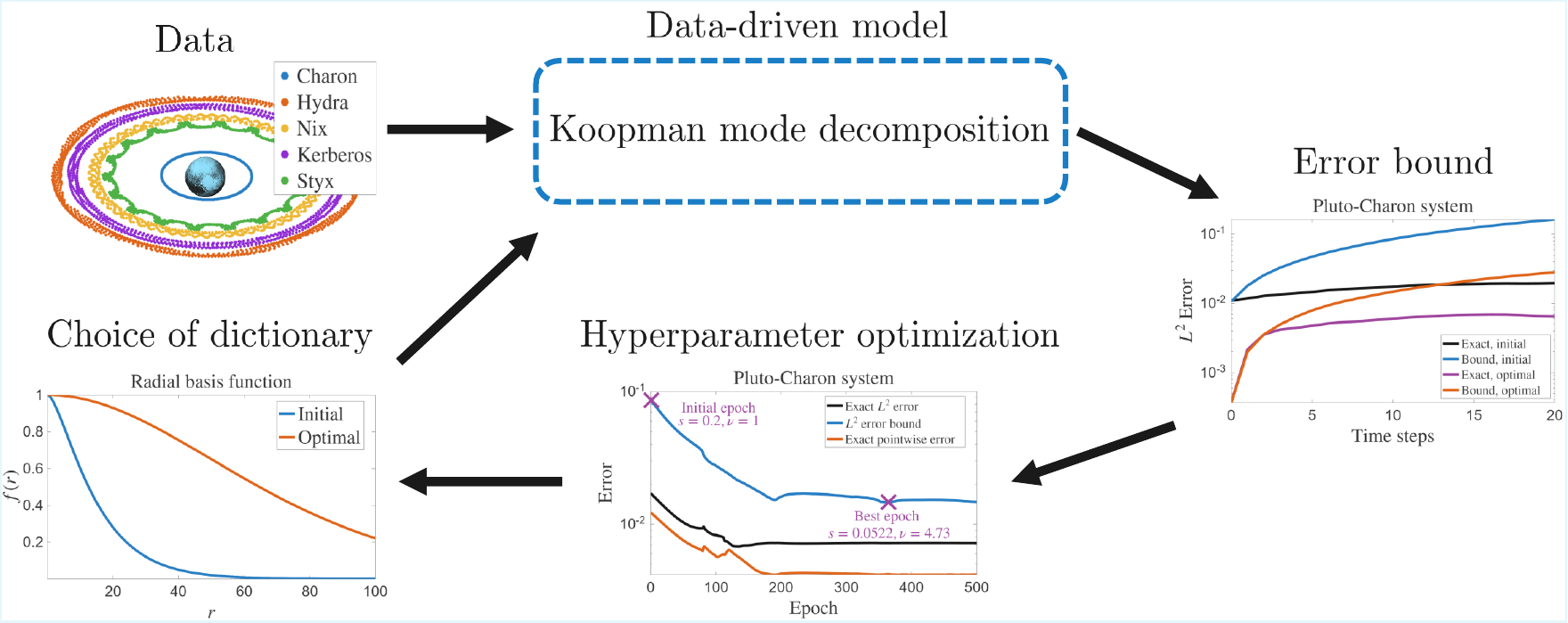}    
    \caption{Workflow for computing error bounds on Koopman operators and learning a dictionary.}
    \label{fig:flowchart_paper}
\end{figure}

However, these algorithms often suffer from spurious dynamics and artifacts. Many analyses of DMD implicitly assume that the chosen dictionary spans a finite-dimensional, non-trivial (i.e., not merely constant) invariant subspace. In practice, such invariance is exceedingly rare. One must therefore work with \textit{approximately} invariant subspaces, which causes closure errors~\cite{Bagheri2014pof,brunton2016koopman,dawson2016characterizing,hemati2017biasing,brunton2021modern,colbrook2024rigorous}. In turn, this leads to spurious eigenvalues (spectral pollution), misleading modes, and overly optimistic long-horizon predictions. More generally, the errors in DMD and the associated KMD can be grouped into three categories:
\begin{itemize}[leftmargin=*]
	\item \textit{Estimation error}, due to approximating the projected Koopman operator from a finite set of potentially noisy trajectory data.
	\item \textit{Numerical error} (e.g., roundoff, stability, or further compression), incurred when computing properties of the (finite-dimensional) projected Koopman operator.
    \item \textit{Projection error} (or approximation error), arising from truncating the infinite-dimensional Koopman operator to a finite-dimensional space of observables.
\end{itemize}
While substantial progress has been made in controlling estimation and numerical errors in the presence of noisy and finite data (see \cref{sec:existing_work}), bounds on projection errors remain scarce. Thus, a central practical question remains:
\begin{quote}
\emph{Given snapshot data and a chosen dictionary, how can one certify whether the associated Koopman approximation is trustworthy, and how should the dictionary be refined when it is not?}
\end{quote}

This paper develops a methodology for answering this question using \emph{a posteriori} certificates that quantify projection error and subspace invariance directly from the same snapshot data used by DMD/EDMD. The starting point is recent convergent, pollution-free approaches for approximating spectral properties of an infinite-dimensional Koopman operator and its adjoint (the Perron--Frobenius operator)~\cite{colbrook2024rigorous,boulle2025convergent}. We repurpose and extend those ideas into a broader methodology comprising: (i) computable bounds that control the projection error and propagate to certified statements about Koopman and Perron--Frobenius mode decompositions (PFMD) and forecasts across different function spaces, including pointwise bounds; (ii) a geometry-based diagnostic of invariance via principal angles, leading to a \emph{principal angle decomposition} (PAD) that plays a dynamics-informed role analogous to the SVD (e.g., for dimensionality reduction); and (iii) complementary \emph{expected error} surrogates, based on Gaussian process modeling, that are useful when fully rigorous bounds are conservative. Numerical experiments demonstrate the efficacy of the proposed methods and illustrate key theoretical features. Applications include several high-dimensional and chaotic systems, as well as real-world examples. The result is a practical and theoretically grounded toolbox that enables validated Koopman spectral analysis, certified forecasting, and principled dictionary selection (illustrated in \cref{fig:flowchart_paper}) for data-driven modeling of complex dynamical systems.

\begin{figure}[t]
    \centering
    \includegraphics[width=1\linewidth]{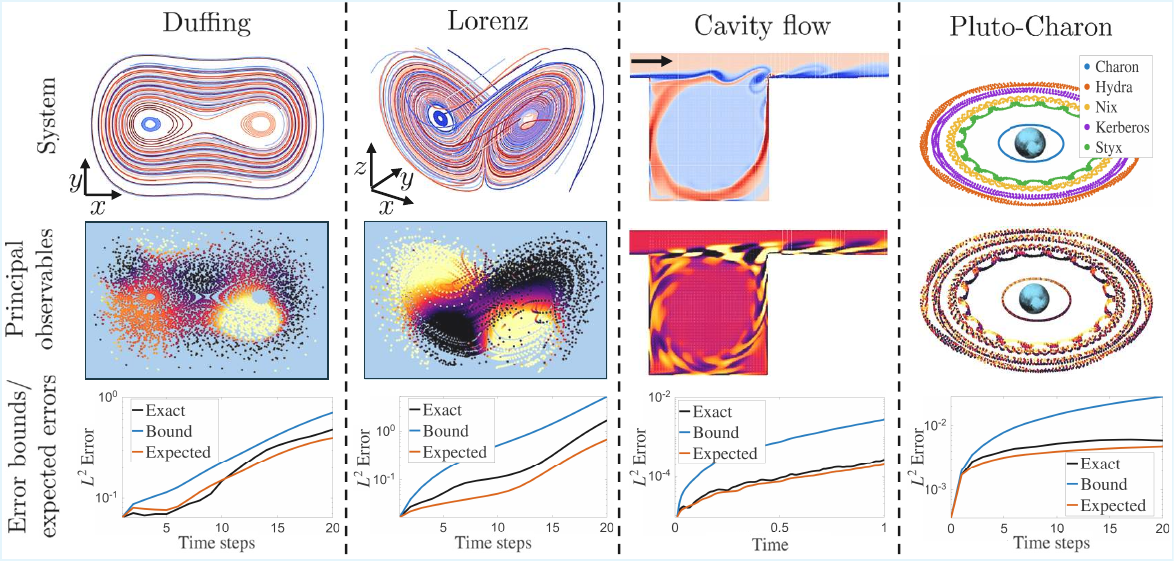}    
    \caption{A summary of the numerical examples of this paper and some of the results of the experiments applied to them. Principal observables are discussed in \cref{sec:Koopman_invariance_measure}, and error bounds and expected errors are discussed in \cref{sec_kmd_eb,sec:expected_errors}, respectively.}
    \label{fig:placeholder}
\end{figure}

\linespread{0.97}
\begin{table}[t]
    \caption{The key algorithms of this paper.}
    \begin{center}
        \begin{tabular}{ lll }
            \toprule[\thick pt]
            Computed quantity  & Pseudocode       & Theoretical results         \\
            \midrule[\thick pt]
            Principal angles and observables                    & \cref{alg:resDMD_angle}        & \cref{sec:residual_connect}       \\
            Principal angle decomposition &
            \cref{alg:pad} 
            & \cref{sec:pad} \\
            Error bounds on KMD/PFMD & \cref{alg:resDMD_KMD,alg:specrkhs_pfmd} & \cref{sec:comp_erbs} \\
            Expected errors of KMD/PFMD & \cref{alg:expected_error_all_l2,alg:expected_error_all_kernel} & \cref{sec:ee_gp_comp}\\
            \bottomrule[\thick pt]
        \end{tabular}
        \label{algorithms-table}
    \end{center}
\end{table}
\linespread{0.97}

The paper is organized as follows. \Cref{sec:data_systems} introduces prerequisites and describes related work, while \cref{sec:Koopman_invariance_measure} discusses subspace invariance and principal angle decompositions. \Cref{sec_kmd_eb} develops rigorous upper bounds for the Koopman and Perron--Frobenius mode decompositions, and \cref{sec:expected_errors} complements this with expected error computations. \Cref{sec:num_examples} presents pointwise error bounds and dictionary refinement applied to data from the Pluto--Charon system. \Cref{algorithms-table} summarizes the proposed algorithms, and \cref{fig:placeholder} illustrates the experiments considered throughout the paper. Code and datasets for reproducing the numerical results are available at \texttt{\url{https://github.com/GustavConradie1/TrustKoopman}}.

\section{Data-driven dynamical systems} \label{sec:data_systems}

We consider discrete-time dynamical systems with state space $\mathcal{X}$ (e.g., the combined position-momentum space of a particle system) and evolution
\begin{equation}\setlength\abovedisplayskip{6pt}\setlength\belowdisplayskip{6pt}
\label{eq:DynamicalSystem} 
x_{n+1} = F(x_n), \quad n= 0,1,2,\ldots.
\end{equation}
Here, $F:\mathcal{X}\rightarrow\mathcal{X}$ is an unknown function that generates highly nonlinear or chaotic dynamics and $x_n\in\mathcal{X}$ denotes the state of the system at step $n$. We assume that snapshot data of the form $\{x^{(m)},y^{(m)}=F(x^{(m)})\}_{m=1}^M$ are available from experiments or numerical simulations.
A central challenge is to reconstruct quantitative information about \cref{eq:DynamicalSystem} from such data.
Traditional geometric approaches, which generally rely on local linearization \cite{strogatz2018nonlinear}, often struggle in high-dimensional or data-driven settings \cite{budivsic2012applied,colbrook2024rigorous,kutz2016dynamic}.

\subsection{The Koopman operator}
\label{sec:prelims}

Koopman operator theory provides global linearizations of dynamical systems~\cite{mezic2005spectral,budivsic2012applied,brunton2021modern}.
The key idea is to lift the dynamics to an infinite-dimensional space by studying the evolution of observables of the states, i.e., functions $g:\mathcal{X}\to\mathbb{C}$, rather than the states themselves.
The Koopman operator associated with \cref{eq:DynamicalSystem} is the \textit{linear} operator $\mathcal{K}$ satisfying
\[\setlength\abovedisplayskip{6pt}\setlength\belowdisplayskip{6pt}
[\koop g](x)=g(F(x)),\quad x \in\mathcal{X}.\]
The linearity of $\koop$ enables the use of tools from spectral theory to study the dynamical system.
Koopman operators are typically defined on a Hilbert space $\mathcal{H}$ of observables. Two common choices are: 

\paragraph{Square-integrable functions} The choice of space $L^2(\mathcal{X},\omega)$ for a (possibly non-invariant) measure $\omega$ is classical in ergodic theory. The inner product is 
\[\setlength\abovedisplayskip{6pt}\setlength\belowdisplayskip{6pt}
\langle f,g\rangle_{L^2}=\int_{\mathcal{X}}{f(x)}\overline{g(x)}\dd \omega(x),\quad f,g\in L^2(\mathcal{X},\omega).\]
This is the setting in which Koopman and von Neumann first introduced $\koop$ \cite{koopman1931hamiltonian,koopman1932dynamical}.

\paragraph{Reproducing kernel Hilbert spaces (RKHS)} These spaces are often used for high-dimensional systems \cite{williams2015kernel}. An RKHS on $\mathcal{X}$ is a Hilbert space $\mathcal{H}$ of functions $g:\mathcal{X}\to\mathbb{C}$ such that, for every $x\in\mathcal{X}$, the point-evaluation functional $E_x:g\mapsto g(x)$ is bounded. By the Riesz representation theorem, there exists a kernel function $\mathfrak{K}_x\in\mathcal{H}$ at $x$ satisfying  $g(x)=\langle g,\mathfrak{K}_x\rangle_{\mathfrak{K}}$ for all $g\in\mathcal{H}$. This induces a conjugate symmetric, positive-definite kernel function $\mathfrak{K}:\mathcal{X}\times\mathcal{X}\rightarrow\mathbb{C}$ such that 
\[\setlength\abovedisplayskip{6pt}\setlength\belowdisplayskip{6pt}\mathfrak{K}(x,y)=\langle\mathfrak{K}_x,\mathfrak{K}_y\rangle_{\mathfrak{K}}=\mathfrak{K}_x(y),\quad x,y\in\mathcal{X},\]
which uniquely determines the RKHS \cite{aronszajn_theory_1950}.

A major advantage of RKHSs is that, unlike the $L^2$ norm, convergence in the RKHS norm implies pointwise convergence. As we shall see shortly, working on RKHSs allows us to avoid the large-data limit and the sampling assumptions required by $L^2$-based methods. Moreover, they handle high-dimensional systems more efficiently and provide flexibility in the choice of kernel \cite{boulle2025convergent}. However, the Koopman operator on $L^2$ spaces has closer connections to ergodic theory and statistical properties of the dynamics, such as autocorrelations \cite{mezic2005spectral}. $L^2$ spaces are also the natural setting for the case of measure-preserving dynamics \cite{colbrook2022mpedmd}. Whenever it is clear from context, we drop the notational subscript $L^2$ or $\mathfrak{K}$. Unless otherwise specified, we use the $2$-norm for vectors and matrices.

\subsection{Controlling projection errors} \label{sec_application_error}

Approximations of $\koop$ seldom provide certified error bounds compared to the true operator. We therefore focus on rigorous error bounds for Galerkin approximations of the Koopman operator, and show how these bounds yield quantitative measures of subspace invariance in \cref{sec:Koopman_invariance_measure} and forecast guarantees in \cref{sec_kmd_eb}.

\subsubsection{Galerkin methods}

\label{sec:galerkin}
The most prevalent method for approximating the Koopman operator in the $L^2$ setting is EDMD \cite{williams2015data}. 
Given a finite dictionary of observables $\{\psi_1,\ldots,\psi_{N}\}\subseteq L^2(\mathcal{X},\omega)$, EDMD constructs a matrix $\mathbf{K}\in\mathbb{C}^{N\times N}$ from snapshot data  $\{x^{(m)},y^{(m)}\}_{m=1}^M$ that approximates the projection of $\mathcal{K}$ onto the finite-dimensional subspace $\mathcal{V}_{{N}}=\mathrm{span}\{\psi_1,\ldots,\psi_{N}\}$. We define the vector-valued observable 
$\smash{\Psi(x)=\begin{bmatrix}\psi_1(x) & \cdots& \psi_{{N}}(x) \end{bmatrix}\in\mathbb{C}^{1\times {N}}}$
so that any observable $g\in \mathcal{V}_{N}$ can be expressed as $\smash{g(x)=\sum_{j=1}^{N}\psi_j(x)g_j=\Psi(x)\,\mathbf{g}}$ for some $\mathbf{g}\in\mathbb{C}^{N}$. We seek $\mathbf{K}\in\mathbb{C}^{N\times N}$ such that $\mathcal{K}[\Psi \mathbf{g}]\approx  \Psi[\mathbf{K} \mathbf{g}]
$ for $\mathbf{g}\in\mathbb{C}^{N}$. To find this matrix, we define
\begin{equation}\setlength\abovedisplayskip{6pt}\setlength\belowdisplayskip{6pt}
\begin{split}
\mathbf{\Psi}_X\coloneqq\begin{pmatrix}
\Psi(x^{(1)})\\
\vdots\\
\Psi(x^{(M)})
\end{pmatrix}\in\mathbb{C}^{M\times N},\quad
\mathbf{\Psi}_Y\coloneqq\begin{pmatrix}
\Psi(y^{(1)})\\
\vdots \\
\Psi(y^{(M)})
\end{pmatrix}\in\mathbb{C}^{M\times N}.
\label{PSI_defs}
\end{split}
\end{equation}
We assign a weight $w_m>0$ to each $x^{(m)}$, so that $\{(x^{(m)},w_m)\}_{m=1}^M$ can be viewed as a quadrature rule for integration with respect to $\omega$. Let $\mathbf{W}=\mathrm{diag}(w_1,\ldots,w_M)$ and
\[\setlength\abovedisplayskip{6pt}\setlength\belowdisplayskip{6pt}
\mathbf{G}{=}\mathbf{\Psi}_X^*\mathbf{W}\mathbf{\Psi}_X{=}\sum_{m=1}^{M} w_m \Psi(x^{(m)})^*\Psi(x^{(m)}),\,
\mathbf{A}{=}\mathbf{\Psi}_X^*\mathbf{W}\mathbf{\Psi}_Y{=}\sum_{m=1}^{M} w_m \Psi(x^{(m)})^*\Psi(y^{(m)}).
\]
If this quadrature approximation converges, then
\begin{equation}\setlength\abovedisplayskip{6pt}\setlength\belowdisplayskip{6pt}
\label{quad_convergence}
\lim_{M\rightarrow\infty}\mathbf{G}_{jk} = \langle \psi_k,\psi_j \rangle_{L^2}\quad \text{ and }\quad \lim_{M\rightarrow\infty}\mathbf{A}_{jk} = \langle \mathcal{K}\psi_k,\psi_j \rangle_{L^2}.
\end{equation}
There are typically three scenarios where \cref{quad_convergence} holds: 1) random sampling where the $\{{x}^{(m)}\}_{m=1}^M$ are drawn independently \cite{klus2015numerical,korda2018convergence}, 2) ergodic sampling where ${x}^{(m)}=F^{m-1}(x^{(1)})$ \cite{arbabi2017ergodic}, and 3) high-order quadrature \cite{colbrook2024rigorous}.
$\mathbf{G}$ and $\mathbf{A}$ respectively fulfill similar roles to the mass and stiffness matrices from finite element methods \cite{ern2004theory}. Assuming $\lim_{M\rightarrow\infty}\mathbf{G}$ is invertible (i.e., the dictionary is linearly independent), the convergence in \cref{quad_convergence} means that the matrix
\begin{equation}\label{eqn:K_l2_defn}\setlength\abovedisplayskip{6pt}\setlength\belowdisplayskip{6pt}
\mathbf{K}=\mathbf{G}^{-1}\mathbf{A}=(\mathbf{\Psi}_X^*\mathbf{W}\mathbf{\Psi}_X)^{-1}(\mathbf{\Psi}_X^*\mathbf{W}\mathbf{\Psi}_Y)=(\sqrt{\mathbf{W}}\mathbf{\Psi}_X)^{\dagger}\sqrt{\mathbf{W}}\mathbf{\Psi}_Y
\end{equation}
approaches a representation of the truncated operator $\mathcal{P}_{\mathcal{V}_{N}}\mathcal{K}\mathcal{P}_{\mathcal{V}_{N}}^*$ as $M\to\infty$. Here, $\mathcal{P}_{\mathcal{V}_{N}}$ denotes the orthogonal projection onto the subspace $\mathcal{V}_{N}$ and $\mathbf{M}^{\dagger}\in\mathbb{C}^{m\times n}$ denotes the pseudoinverse of $\mathbf{M}\in\mathbb{C}^{n\times m}$.

\paragraph{Koopman mode decomposition}
Given an observable $g\in \mathcal{V}_N$, we may expand it as
$g=\Psi\mathbf{g} = \Psi \mathbf{V}\left(\mathbf{V}^{-1}\mathbf{g}\right)$, where $\mathbf{V}$ is the matrix of eigenvectors of $\mathbf{K}$ with eigenvalues $\{\lambda_j\}_{j=1}^N$ (we assume that $\mathbf{K}$ is diagonalizable). This expansion is known as the Koopman mode decomposition. If $g\in L^2(\mathcal{X},\omega)\setminus \mathcal{V}_N$, we typically take the least squares approximation to $g$ in $\mathcal{V}_N$. Since $\mathbf{K}\mathbf{V}=\mathbf{V}\mathbf{\Lambda}$, where $\mathbf{\Lambda}=\mathrm{diag}(\lambda_1,\dots,\lambda_N)$, this provides a simple expression for the evolution of the observable under the action of the Koopman operator. Provided that $\koop$ is bounded, the KMD converges as the size of the dictionary $N\rightarrow\infty$~\cite{korda2018convergence}. In particular, the large-data limit of $\mathbf{K}$ converges to $\koop$ in the strong operator topology, i.e., for all $g\in L^2(\mathcal{X},\omega)$ and $\mathbf{g}\in\mathbb{C}^{N}$ such that $\lim_{N\rightarrow\infty}\lim_{M\rightarrow\infty}\Psi\mathbf{g}=g$,
\begin{equation}\label{eqn:strong_op_conv}\setlength\abovedisplayskip{6pt}\setlength\belowdisplayskip{6pt}
\lim_{N\rightarrow\infty}\lim_{M\rightarrow\infty}\Psi\mathbf{K}\mathbf{g}=\koop g.
\end{equation}
In this paper, we bound the projection errors corresponding to finite $N$.

\paragraph{Kernelized EDMD}
When the state space is high-dimensional, EDMD suffers from the curse of dimensionality~\cite{bishop_pattern_2016}. The kernelized EDMD (kEDMD) algorithm~\cite{williams2015kernel} aims to mitigate this by exploiting the kernel trick~\cite{scholkopf_kernel_2000}, allowing efficient low-rank representations.
As noted in~\cite{boulle2025convergent}, kEDMD can be expressed as a Galerkin method by considering the action of $\koop^*$ on the kernel functions $\mathfrak{K}_x$. For all $x\in\mathcal{X}$ and $f\in\mathcal{H}$,
\[\setlength\abovedisplayskip{6pt}\setlength\belowdisplayskip{6pt}
    \langle f,\koop^*\mathfrak{K}_x\rangle_{\mathfrak{K}}=\langle\koop f,\mathfrak{K}_x\rangle_{\mathfrak{K}}=(\koop f)(x)=f(F(x))=\langle f,\mathfrak{K}_{F(x)}\rangle_{\mathfrak{K}}.
\]
Hence, $\koop^*\mathfrak{K}_x=\mathfrak{K}_{F(x)}$.
Setting $N=M$, one can construct a Galerkin approximation of $\koop^*$ from pairs of snapshot data $\{(x^{(n)},y^{(n)})\}_{n=1}^N$ using the dictionary of data-centred kernel function observables $\{\mathfrak{K}_{x^{(n)}}\}_{n=1}^N$. We define the matrices $\mathbf{G}^{\mathfrak{K}}$ and $\mathbf{A}^{\mathfrak{K}}$ by
\begin{equation}\setlength\abovedisplayskip{6pt}\setlength\belowdisplayskip{6pt}\label{eqn:GA_ker_defn}
    \mathbf{G}^{\mathfrak{K}}_{jk}=\mathfrak{K}(x^{(k)},x^{(j)})=\langle\mathfrak{K}_{x^{(k)}},\mathfrak{K}_{x^{(j)}}\rangle_{\mathfrak{K}},\; \mathbf{A}^{\mathfrak{K}}_{jk}=\mathfrak{K}(y^{(k)},x^{(j)})=\langle\koop^*\mathfrak{K}_{x^{(k)}},\mathfrak{K}_{x^{(j)}}\rangle_{\mathfrak{K}}.
\end{equation}
The superscript $\mathfrak{K}$ distinguishes these from the matrices in \cref{quad_convergence}, but will be dropped when the context is clear. Then, \begin{equation}\label{eqn:K_ker_defn}\setlength\abovedisplayskip{6pt}\setlength\belowdisplayskip{6pt}\mathbf{K}^{\mathfrak{K}}=(\mathbf{G}^{\mathfrak{K}})^{-1}\mathbf{A}^{\mathfrak{K}}\end{equation} is a Galerkin approximation to $\koop^*$ on the subspace $\mathcal{V}_N=\mathrm{span}\{\mathfrak{K}_{x^{(1)}},\dots,\mathfrak{K}_{x^{(N)}}\}$. Unlike the $L^2$-Galerkin approximation used in EDMD, the RKHS formulation does not require taking the large-data limit $M\to\infty$ to form the Galerkin matrices, and does not require that the data is sampled according to a convergent quadrature \cite{boulle2025convergent}.
It yields convergence to $\koop^*$ in the strong operator topology as $N\to\infty$. The kEDMD matrix \cite{williams2015kernel} corresponds to $(\mathbf{K}^{\mathfrak{K}})^\top$, and it is common to project onto a dominant eigenspace of $\mathbf{G}^{\mathfrak{K}}$. As in the $L^2$ case, in this paper we study the projection errors corresponding to finite $N$.

Approximations of $\koop^*$ can be used to study the forward dynamics of the system. For example, suppose that $\mathcal{X}\subset\mathbb{R}^d$ and set $g_i(x)=[x]_i$, the $i$th component of $x\in\mathcal{X}$, for $1\leq i\leq d$. Then
\begin{equation}
\label{eqn:pfmd_predictions}
\setlength\abovedisplayskip{6pt}\setlength\belowdisplayskip{6pt}
    [x_n]_i=(\koop^ng_i)(x_0)=\langle \koop^ng_i,\mathfrak{K}_{x_0}\rangle_{\mathfrak{K}}=\langle g_i,(\koop^*)^n\mathfrak{K}_{x_0}\rangle_{\mathfrak{K}}.
\end{equation}
Thus, by applying $\koop^*$ to the single observable $\mathfrak{K}_{x_0}$, which can be approximated using the basis $\mathfrak{K}_{x^{(j)}}$ for $j=1,\dots,N$, we can predict all $d$ observables $g_i$. Decomposing $\mathbf{K}^{\mathfrak{K}}$ in terms of its eigenvalues yields the analogue of the KMD, the Perron--Frobenius mode decomposition (PFMD).

\subsubsection{Computing spectra}\label{sec:spectral_convergence}

The convergence in \cref{quad_convergence} implies that the EDMD eigenvalues approach the spectrum of $\mathcal{P}_{\mathcal{V}_{N}}\mathcal{K}\mathcal{P}_{\mathcal{V}_{N}}^*$ as $M\rightarrow\infty$. However, as $N\to\infty$, despite the strong operator topology convergence \cref{eqn:strong_op_conv}, the spectrum of $\mathcal{P}_{\mathcal{V}_{N}}\mathcal{K}\mathcal{P}_{\mathcal{V}_{N}}^*$ need not converge to that of $\koop$~\cite{bottcher1983finite}, leading to spectral pollution (persistent spurious eigenvalues) or spectral invisibility (missing spectral regions) \cite{colbrook2025introductory}. The same is true in the RKHS setting. It is therefore essential to verify candidate eigenpairs, which can be done using residual dynamic mode decomposition (ResDMD)~\cite{colbrook2024rigorous} in the $L^2$ setting or SpecRKHS \cite{boulle2025convergent} in the RKHS setting. 

Given $g\in\mathcal{V}_N$ and $\lambda\in\mathbb{C}$, the goal of ResDMD is to compute the infinite-dimensional residual $\|\koop g-\lambda g\|_{L^2}$.
To achieve this, define $\Lv=\Psiv_Y^*\Wv\Psiv_Y$, which, assuming the quadrature rule converges, approximates $\mathcal{K}^*\mathcal{K}$ in the sense that
\begin{equation}\setlength\abovedisplayskip{6pt}\setlength\belowdisplayskip{6pt}
    \label{quad_convergence2}
    \lim_{M\rightarrow\infty}\Lv_{jk} = \langle \mathcal{K}\psi_k,\mathcal{K}\psi_j \rangle_{L^2}.
\end{equation}
Hence,
$
    \|\koop g-\lambda g\|^2_{L^2}=\lim_{M\rightarrow\infty}\mathbf{g}^*(\mathbf{L}-\lambda \mathbf{A}^*-\overline{\lambda} \mathbf{A}+|\lambda|^2\mathbf{G})\mathbf{g}= \lim_{M\rightarrow\infty}\mathrm{res}(\lambda,g)^2.
$
The residual $\mathrm{res}(\lambda,g)$ measures the quality of $(\lambda,g)$. If $\|g\|_{L^2}=1$ and $\|(\mathcal{K}-\lambda I)g\|_{L^2}\leq\epsilon$, then $g$ is an $\epsilon$-pseudoeigenfunction. Such observables are dynamically relevant since
\[\setlength\abovedisplayskip{6pt}\setlength\belowdisplayskip{6pt}
\|\mathcal{K}^ng-\lambda^n g\|_{L^2}= \mathcal{O}(n\epsilon),\quad \text{for all }n\in\mathbb{N},
\]
so $\lambda$ represents an approximate coherent oscillation with corresponding growth or decay of $g$ over time.
With this residual in hand, we can build convergent algorithms for computing various spectral properties of $\koop$ \cite{colbrook2024rigorous}. Analogous to ResDMD, convergent spectral approximations of $\koop^*$ on an RKHS require introducing an additional matrix~\cite{colbrook_another_2024,boulle2025convergent}:
\begin{equation}\setlength\abovedisplayskip{6pt}\setlength\belowdisplayskip{6pt}\label{eqn:L_ker_defn}
    \mathbf{R}^{\mathfrak{K}}_{jk}=\mathfrak{K}(y^{(k)},y^{(j)})=\langle\mathfrak{K}_{y^{(k)}},\mathfrak{K}_{y^{(j)}}\rangle_{\mathfrak{K}}=\langle\koop^*\mathfrak{K}_{x^{(k)}},\koop^*\mathfrak{K}_{x^{(j)}}\rangle_{\mathfrak{K}},\quad 1\leq j,k\leq N.
\end{equation}
Just as ResDMD uses the matrix $\mathbf{L}$, SpecRKHS~\cite{boulle2025convergent} uses $\mathbf{R}^{\mathfrak{K}}$ to compute spectral properties of $\koop^*$ on an RKHS. Pseudoeigenfunctions of $\koop^*$ can be used to approximate $\mathfrak{K}_{x_0}$ in \cref{eqn:pfmd_predictions} to provide an analogous interpretation to the $L^2$ case.

\subsection{Related methods}\label{sec:existing_work}

Existing work has illuminated several aspects of data-driven Koopman analysis, but the corresponding validation theory remains distributed across a few complementary directions. Broadly, the literature focuses on three issues: projection error from the choice of approximation space, estimation and numerical error from finite and noisy data, and diagnostics for how nearly invariant a learned feature space is under the true dynamics. These developments are valuable, but they are often tied to specific dictionaries, kernels, or sampling settings. In this work, we develop general, computable, \emph{a posteriori} tools for assessing the reliability of a user-chosen Koopman approximation directly from snapshot data. We organize the discussion below around these three themes.

\paragraph{Koopman projection error bounds}

Several works derive error bounds and convergence rates for specific choices of dictionary or kernel. For example, \cite{zhang2023quantitative,kurdila2018koopman,yadav2025approximation} estimate projection errors for classical finite element spaces, Bernstein polynomials, and wavelets, respectively, while \cite{kohne_linfty-error_2024} treats the kernelized setting and proves error bounds for kEDMD with Wendland kernels. By contrast, our framework allows the user to choose an arbitrary dictionary and then \textit{verify} accuracy a posteriori using computable error bounds. Moreover, by computing principal observables of the Koopman operator, we identify which observables are predicted more or less reliably and obtain tailored, sharper bounds, rather than estimates that depend only on class membership. All our error bounds are fully explicit and computable.

\paragraph{Estimation and numerical error}

In EDMD, estimation error arises from the use of a quadrature rule, which converges in the large-data limit~\cite{colbrook2024rigorous}, so classical convergence-rate results apply. By contrast, \cite{boulle2025convergent} shows that kEDMD does not require a large-data limit, and analyses the case of noisy data. \cite{dawson2016characterizing, hemati2017biasing} both propose adjustments to DMD to reduce bias in the presence of noise. Kernel-based methods and convergence rates for Hilbert--Schmidt Koopman operators are analyzed in~\cite{klus2020kernel,philipp_error_2024}, while \cite{nuske2023finite} derives probabilistic bounds for truncated Koopman generators of SDEs and nonlinear control systems under i.i.d. and ergodic sampling; in these works, projection error is not considered. Related bounds for parabolic PDEs (again after truncation) are given in~\cite{lu2020prediction}. In~\cite{colbrook_beyond_2024}, ResDMD is extended to stochastic systems, with a joint analysis of projection error, estimation error, and their variances. \cite{drmavc2023data,Drmac2020} provide a numerical linear algebra perspective on numerical errors in Koopman approximations.

\paragraph{Subspace invariance}

To measure subspace invariance, we compute principal angles between the chosen dictionary and its image under the Koopman operator. In statistics, principal angles between subspaces are closely related to canonical correlation analysis (CCA)~\cite{hardoon2004canonical}. CCA has been used for stochastic Koopman operators in~\cite{wu2020variational} to develop a variational framework for Markov processes (VAMP), and in~\cite{mardt2018vampnets} this approach was applied to protein folding with neural-network parametrized observables. Principal angles are also connected to coherent sets over finite time horizons~\cite{froyland2013analytic}; see~\cite{banisch2017understanding} for diffusion-map techniques and~\cite{klus2019kernel} for links with kernel CCA. The largest principal angle was theoretically related to the invariance proximity measure in~\cite{haseli2023invariance}.

\paragraph{Other data-driven techniques}

Error bounds are a central component of reduced-order modeling. Reduced basis methods~\cite{hesthaven2022reduced}, widely used for parametrized PDEs, are typically equipped with a posteriori estimators and rigorous error bounds~\cite{rozza2008reduced,hesthaven2016certified,quarteroni2015reduced}; see~\cite{benner2015survey} for a survey of projection-based model reduction. Another major approach is proper orthogonal decomposition (POD)~\cite{berkooz1993proper}, for which error bounds have been developed~\cite{chaturantabut2010nonlinear,kunisch2002galerkin,kunisch2001galerkin}, as well as for variants such as balanced POD~\cite{moore1981principal,willcox2002balanced,rowley2005model} and Petrov--Galerkin POD~\cite{carlberg2017galerkin,otto2022optimizing,otto2023model}. Bagging and ensemble strategies have also been applied to data-driven identification methods such as DMD and SINDy; while these can yield empirical error estimates, they rely on subsampling to improve robustness rather than providing intrinsic, certified bounds~\cite{fasel2022ensemble,sashidhar2022bagging}.

\section{Principal angles and Koopman invariance}
\label{sec:Koopman_invariance_measure}

To control projection errors, we seek to quantify how invariant a subspace $\mathcal{V}$ of observables is under the Koopman operator $\mathcal{K}$. Applications include: 1) measuring the distance between $\mathrm{span}\{g\}$ and $\mathrm{span}\{\mathcal{K}g\}$ for an approximate eigenfunction $g$, 2) mitigating closure errors in (k)EDMD by selecting a dictionary that is nearly invariant under $\mathcal{K}$, and 3) model reduction.
In this section, invariance is quantified using principal angles between subspaces of observables. We compute these quantities for the infinite-dimensional Koopman and Perron--Frobenius operators.

\paragraph{Principal angles and vectors} Let $S_1$ and $S_2$ be finite-dimensional subspaces of a Hilbert space $\mathcal{H}$ with
$q= \mathrm{dim}(S_2)\leq\mathrm{dim}(S_1)=p$. The principal angles $\{\theta_j\}_{j=1}^q\subset[0,\pi/2]$ and principal vectors $\{(u_j,v_j)\}_{j=1}^q\subset S_1\times S_2$ are defined by
\[\setlength\abovedisplayskip{6pt}\setlength\belowdisplayskip{6pt}
\theta_1=\min\left\{\arccos(|\langle u,v\rangle|):(u,v)\in S_1\times S_2, \|u\|=\|v\|=1\right\},
\]
with $|\langle u_1,v_1\rangle|=\cos(\theta_1)$, and for $2\leq j\leq q$,
\begin{align*}\setlength\abovedisplayskip{6pt}\setlength\belowdisplayskip{6pt}
\theta_j=\min\left\{\arccos(|\langle u,v\rangle|):(u,v)\in S_1\times S_2,u\perp u_k,v\perp v_k, \, 1\leq k\leq j-1, \right.\\
\left.\|u\|=\|v\|=1\right\},\quad 
|\langle u_j,v_j\rangle|=\cos(\theta_j),\quad \|u_j\|=\|v_j\|=1.
\end{align*}
The principal vectors need not be uniquely defined (even up to phase), but the principal angles always are; without loss of generality, we select $\langle u_j,v_j\rangle\in\mathbb{R}_+$. Principal angles and vectors can be obtained through the singular value decomposition of a covariance matrix, implemented in the \texttt{subspacea}~\cite{knyazevcodel} MATLAB package.

\subsection{Computing the principal angles and vectors}\label{sec:princ_ang_vec_comp}

To measure Koopman subspace invariance, we seek to compute the subspace angles between $S_1=\mathcal{V}\subset L^2(\mathcal{X},\omega)$ and $S_2=\mathcal{K}\mathcal{V}$. The subspace $\mathcal{V}\subset\mathcal{V}_N$ is finite-dimensional and spanned by the columns of $\Psi \mathbf{B}$, where $\mathbf{B}\in\mathbb{C}^{N\times n}$ is a subspace selection matrix with $n\le N$ columns. For instance, $\mathcal{V}$ may coincide with the full space $\mathcal{V}_N$ introduced in \cref{sec:galerkin} or with the span of a candidate eigenvector. To compute the principal angles and vectors using \texttt{subspacea}, we require representations of $\mathcal{V}$ and $\mathcal{K}\mathcal{V}$ in $L^2$-orthonormal bases. Note that $\mathcal{K}\mathcal{V}$ is defined using the infinite-dimensional Koopman operator $\mathcal{K}$, not a finite-dimensional approximation.

Since we are only interested in the subspace spanned by the columns of $\Psi\mathbf{B}$, we compress the matrices defined in \cref{quad_convergence,quad_convergence2}, and define
\[\setlength\abovedisplayskip{6pt}\setlength\belowdisplayskip{6pt}
\mathbf{G}_\mathcal{V}=\mathbf{B}^*\mathbf{G}\mathbf{B}, \quad \mathbf{A}_\mathcal{V}=\mathbf{B}^*\mathbf{A}\mathbf{B},\quad \mathbf{L}_\mathcal{V}=\mathbf{B}^*\mathbf{L}\mathbf{B}.
\]
This is equivalent to re-running EDMD with the dictionary $
g_1(x) = \Psi(x)\mathbf{B}(:,1)$, $\dots$, $g_{n}(x) = \Psi(x)\mathbf{B}(:,n)
$; we do not assume that the $g_i$'s are linearly independent. We next define the self-adjoint positive semi-definite matrix
\[\setlength\abovedisplayskip{6pt}\setlength\belowdisplayskip{6pt}
\mathbf{J}_\mathcal{V}=\begin{pmatrix}
\mathbf{G}_\mathcal{V} & \mathbf{A}_\mathcal{V}\\
\mathbf{A}_\mathcal{V}^* & \mathbf{L}_\mathcal{V}
\end{pmatrix}\in\mathbb{C}^{2n\times 2n},
\]
which is the mass matrix of the observables $\{f_1,\ldots,f_{2n}\}\coloneqq\{g_1,\ldots,g_n,\mathcal{K}g_1,\ldots,\mathcal{K}g_n\}$ with respect to the $L^2$-inner product. We compute the eigendecomposition
$\mathbf{J}_\mathcal{V}=\mathbf{U}\mathbf{D}\mathbf{U}^*$, where $\mathbf{U}$ is a unitary matrix whose columns, $
\xi_k=\sum_{j=1}^{2n} \mathbf{U}_{jk}f_j$, are the eigenvectors of $\mathbf{J}_\mathcal{V}$ and $\mathbf{D}$ is a diagonal matrix. Then,
\[
\setlength\abovedisplayskip{6pt}\setlength\belowdisplayskip{6pt}
\langle \xi_i,\xi_j\rangle =\sum_{k,l=1}^{2n}{\mathbf{U}_{ki}}\overline{\mathbf{U}_{lj}}(\mathbf{J}_{\mathcal{V}})_{lk}=[\mathbf{U}^*\mathbf{J}_{\mathcal{V}}\mathbf{U}]_{ji}=\mathbf{D}_{ji}.
\]
Hence, the set of vectors $\{\mathbf{D}_{jj}^{-1/2}\xi_j:\mathbf{D}_{j,j}>0\}$
is orthonormal in $L^2(\mathcal{X},\omega)$. In practice, there are errors arising from the finite amount of snapshot data and numerical round-off. For a cut-off tolerance $\epsilon_c\geq 0$, we define
\[\setlength\abovedisplayskip{6pt}\setlength\belowdisplayskip{6pt}
\mathcal{S}=\{\mathbf{D}_{jj}^{-1/2}\xi_j:\mathbf{D}_{jj}>\epsilon_c\}, \quad \mathcal{I}=\{j:\mathbf{D}_{jj}>\epsilon_c,\,j=1,\ldots,2n\},
\]
and set
\[\setlength\abovedisplayskip{6pt}\setlength\belowdisplayskip{6pt}
\mathbf{C}_1=\sqrt{\mathbf{D}(\mathcal{I},\mathcal{I})}\left(\mathbf{U}(1:n,\mathcal{I})\right)^*,\quad \mathbf{C}_2=\sqrt{\mathbf{D}(\mathcal{I},\mathcal{I})}\left(\mathbf{U}(n+1:2n,\mathcal{I})\right)^*.
\]
The choice of cut-off parameter can be guided by the condition in \cref{thm:largedataconv} below and the numerical precision in use. The spans of the columns of $\mathbf{C}_1$ and $\mathbf{C}_2$ represent $\mathcal{V}$ and $\mathcal{K}\mathcal{V}$, respectively, with respect to the orthonormal basis $\mathcal{S}$.

With $\mathbf{C}_1$ and $\mathbf{C}_2$ in hand, we may compute the principal angles and vectors using \texttt{subspacea}. After obtaining principal angles $\{\theta_j\}_{j=1}^q$ and vectors $\mathbf{U}_1$, $\mathbf{U}_2$, we convert back to the original dictionary via 
\begin{equation}\setlength\abovedisplayskip{6pt}\setlength\belowdisplayskip{6pt}\label{dict_conversion_formula}
\widehat{\textbf{U}}=\mathbf{U}(:,\mathcal{I})(\mathbf{D}(\mathcal{I},\mathcal{I}))^{-1/2}\textbf{U}_1,\quad \widehat{\textbf{V}}=\mathbf{U}(:,\mathcal{I})(\mathbf{D}(\mathcal{I},\mathcal{I}))^{-1/2}\textbf{U}_2.
\end{equation}
Then, the principal observables for $j=1,\dots,q$ are given by
\begin{equation}\setlength\abovedisplayskip{6pt}\setlength\belowdisplayskip{6pt}\label{principal_obs_defn}
u_j(x)=[\Psi(x) \mathbf{B}, [\mathcal{K}\Psi](x) \mathbf{B} ]\widehat{\textbf{U}}(:,j),\quad v_j(x)=[\Psi(x) \mathbf{B}, [\mathcal{K}\Psi](x) \mathbf{B} ]\widehat{\textbf{V}}(:,j).
\end{equation}
\cref{alg:resDMD_angle} summarizes our procedure.

\begin{algorithm}[t]
\textbf{Input:} Matrices $\mathbf{G}$, $\mathbf{A}$ and $\mathbf{L}$ in \cref{quad_convergence,quad_convergence2} for dictionary $\{\psi_j\}_{j=1}^{N}$,  subspace selection matrix $\mathbf{B}\in\mathbb{C}^{N\times n}$, cut-off tolerance $\epsilon_c\geq 0$. For RKHS, use instead the matrices $\mathbf{G}$, $\mathbf{A}$ and $\mathbf{R}$ (instead of $\mathbf{L}$) from \cref{eqn:GA_ker_defn,eqn:L_ker_defn}.\\
\begin{algorithmic}[1]
\STATE Compress the matrices to form
$
\mathbf{G}_\mathcal{V}=\mathbf{B}^*\mathbf{G}\mathbf{B}$, $\mathbf{A}_\mathcal{V}=\mathbf{B}^*\mathbf{A}\mathbf{B}$, and  $\mathbf{L}_\mathcal{V}=\mathbf{B}^*\mathbf{L}\mathbf{B}$.
\STATE Compute the matrix
\[\setlength\abovedisplayskip{6pt}\setlength\belowdisplayskip{6pt}
\mathbf{J}_\mathcal{V}=\begin{pmatrix}
\mathbf{G}_\mathcal{V} & \mathbf{A}_\mathcal{V}\\
\mathbf{A}_\mathcal{V}^* & \mathbf{L}_\mathcal{V}
\end{pmatrix}\in\mathbb{C}^{2n\times 2n},
\]
and its eigendecomposition $\mathbf{J}_\mathcal{V}=\mathbf{U}\mathbf{D}\mathbf{U}^*.$
\STATE Set
$
\mathcal{I}=\{j:\mathbf{D}_{jj}>\epsilon_c,\,j=1,\ldots,2n\},
$
and compute
\[\setlength\abovedisplayskip{2pt}\setlength\belowdisplayskip{0pt}
\mathbf{C}_1=\sqrt{\mathbf{D}(\mathcal{I},\mathcal{I})}\left(\mathbf{U}(1:n,\mathcal{I})\right)^*,\quad \mathbf{C}_2=\sqrt{\mathbf{D}(\mathcal{I},\mathcal{I})}\left(\mathbf{U}(n+1:2n,\mathcal{I})\right)^*.
\]
\STATE Compute the principal angles and observables
\[\setlength\abovedisplayskip{6pt}\setlength\belowdisplayskip{6pt}
[\{\theta_j\}_{j=1}^q,\textbf{U}_1,\textbf{U}_2] = \texttt{subspacea}(\mathbf{C}_1,\mathbf{C}_2), \quad q\leq n,
\]
and convert back to the original dictionary via \cref{dict_conversion_formula,principal_obs_defn}.
\end{algorithmic} 
\textbf{Output:} Principal angles $\{\theta_j\}_{j=1}^q$ and principal observables $\{(u_j,v_j)\}_{j=1}^q$.
\caption{Principal angles and observables between $\mathcal{V}$ and $\mathcal{K}\mathcal{V}$.}\label{alg:resDMD_angle}
\end{algorithm}

In the RKHS setting, we want to compute the subspace angles between $S_1=\mathcal{V}\subset\mathcal{H}$ and $S_2=\koop^*\mathcal{V}$. To do so, we simply run \cref{alg:resDMD_angle} using the matrices from \cref{eqn:GA_ker_defn,eqn:L_ker_defn} and the dictionary of observables $\{\mathfrak{K}_{x^{(n)}}\}_{n=1}^N$. For example, if $\{x^{(n)}\}_{n=1}^N$ represents a single trajectory, then $\mathfrak{K}_{x^{(n+1)}}=\koop^*\mathfrak{K}_{x^{(n)}}$ for all $n$ so
\[\setlength\abovedisplayskip{6pt}\setlength\belowdisplayskip{6pt}
\mathcal{V}=\mathrm{span}\{\mathfrak{K}_{x^{(1)}},\dots,(\koop^*)^{N-1}\mathfrak{K}_{x^{(1)}}\},\text{ and } \koop^*\mathcal{V}=\mathrm{span}\{\koop^*\mathfrak{K}_{x^{(1)}},\dots,(\koop^*)^N\mathfrak{K}_{x^{(1)}}\}.
\]
Hence, in this case, all principal angles except the last are zero. The same is true for $L^2$ if we use a Krylov subspace $\{g,\koop g,\dots,\koop^{N-1}\}$ as the dictionary.

\subsection{Convergence and connection to residuals}\label{sec:residual_connect}

The following theorem establishes the convergence of \cref{alg:resDMD_angle} in the large-data limit $M\to\infty$ in the $L^2$ setting. This result also extends to the convergence of subspaces spanned by principal vectors associated with equal principal angles; we omit the details for brevity.

\linespread{1}
\begin{theorem}[Large data convergence]\label{thm:largedataconv}
Let $\theta_j(M,\epsilon_c;\mathcal{V},\mathcal{K}\mathcal{V})$ denote the output of \cref{alg:resDMD_angle} for subspaces $\mathcal{V}$ and $\koop\mathcal{V}$ using $M$ snapshots and cut-off parameter $\epsilon_c$, and let $\lambda_1$ be the smallest positive eigenvalue of $\lim_{M\rightarrow\infty}\mathbf{J}_\mathcal{V}$. If $\epsilon_c<\lambda_1$, then
\[\setlength\abovedisplayskip{6pt}\setlength\belowdisplayskip{6pt}
\lim_{M\rightarrow\infty}\theta_j(M,\epsilon_c;\mathcal{V},\mathcal{K}\mathcal{V})=\theta_j(\mathcal{V},\mathcal{K}\mathcal{V}).
\]
\end{theorem}
\begin{proof}
Without loss of generality, we may assume that the eigenvalues of $\mathbf{J}_\mathcal{V}$ are listed in decreasing order. If $\epsilon_c<\lambda_1$, then $\mathcal{I}$ is eventually constant as $M\rightarrow\infty$. It follows that the matrices $\mathbf{C}_j$ in \cref{alg:resDMD_angle} have limits as $M\rightarrow\infty$. In particular, the quantities used in \texttt{subspacea} converge as $M\rightarrow\infty$ to the corresponding quantities as if \cref{alg:resDMD_angle} were executed with no quadrature error and $\epsilon_c=0$.
\end{proof}\linespread{0.97}

In the RKHS case, no analogue of \cref{thm:largedataconv} is required since no large-data limit is needed. It remains advisable, however, to include the cut-off parameter $\epsilon_c$ to mitigate data noise and numerical round-off.

The principal angles connect with the residuals considered in \cref{sec:spectral_convergence} and to error bounds in the KMD. If $\mathcal{V}=\mathrm{span}\{g\}$ is one-dimensional, then $\sin(\theta(\mathcal{V},\koop\mathcal{V}))$ is related to relative residuals:
\[\setlength\abovedisplayskip{6pt}\setlength\belowdisplayskip{6pt}
\sin(\theta(\mathcal{V},\mathcal{K}\mathcal{V}))=\inf_{v\in \mathcal{K}\mathcal{V}}\left\|v-\frac{g}{\|g\|}\right\|=\inf_{\lambda\in \mathbb{C}\backslash\{0\}}\left\|\frac{1}{\lambda\|g\|}\mathcal{K}g-\frac{g}{\|g\|}\right\|=\inf_{\lambda\in\mathbb{C}\backslash\{0\}} \frac{\left\|\mathcal{K}g-\lambda g\right\|}{|\lambda|\|g\|}.
\]
We may also use principal angles along with \cref{lemma:projection_inv} below to bound one-step errors in the KMD.
\begin{lemma}\label{lemma:projection_inv}
For $j=1,\dots,q$, $\mathcal{P}_{\mathcal{V}}^*\mathcal{P}_{\mathcal{V}}v_j=(\cos\theta_j)u_j$ and $\mathcal{P}_{\koop\mathcal{V}}^*\mathcal{P}_{\koop\mathcal{V}}u_j=(\cos\theta_j)v_j$. In addition, $\|\mathcal{P}_{\mathcal{V}}^*\mathcal{P}_{\mathcal{V}}v_j-v_j\|=\sin\theta_j$ and $\|\mathcal{P}_{\koop\mathcal{V}}^*\mathcal{P}_{\koop\mathcal{V}}u_j-u_j\|=\sin\theta_j$.
\end{lemma}

\begin{proof}
By symmetry, it suffices to prove the results involving $\mathcal{P}_{\mathcal{V}}^*\mathcal{P}_{\mathcal{V}}v_j$. Fix $j$, and let $\lambda = \|\mathcal{P}_{\mathcal{V}}^*\mathcal{P}_{\mathcal{V}}v_j\|$; note that $\langle\mathcal{P}_{\mathcal{V}}^*\mathcal{P}_{\mathcal{V}}v_j,v_j\rangle\geq 0$. By the definition of principal angles, for all $u\in\mathcal{V}$ with $\|u\|=1$, $\cos\theta_j\geq |\langle u,v_j\rangle|$, and so
\[\setlength\abovedisplayskip{6pt}\setlength\belowdisplayskip{6pt}
\|\mathcal{P}_{\mathcal{V}}^*\mathcal{P}_{\mathcal{V}}v_j-v_j\|^2=\lambda^2-2\langle\mathcal{P}_{\mathcal{V}}^*\mathcal{P}_{\mathcal{V}}v_j,v_j\rangle+1\geq \lambda^2-2\lambda\cos\theta_j+1.
\]
But $\langle u_j,v_j\rangle=\cos\theta_j$ so if $\mathcal{P}_{\mathcal{V}}^*\mathcal{P}_{\mathcal{V}}v_j=\lambda u_j$ we obtain exactly the right hand side, and so by uniqueness $\mathcal{P}_{\mathcal{V}}v_j=\lambda u_j$. It is clear that the right hand side is minimized when $\lambda=\cos\theta_j$ with minimum value $1-(\cos\theta_j)^2=(\sin\theta_j)^2$, establishing the result.
\end{proof}\linespread{0.97}

Then, given $g\in\mathcal{V}$, $\koop g=\sum_{i=1}^qg_iv_i$ for $g_i=\langle \koop g,v_i\rangle=\mathbf{g}^*\mathbf{L}\mathbf{v}_i$ (recall that the coefficients of ${v}_i$ are with respect to $\koop\mathcal{V}$ and not $\mathcal{V}$). Hence,
\[\setlength\abovedisplayskip{6pt}\setlength\belowdisplayskip{6pt}
\|\mathcal{P}_{\mathcal{V}}^*\mathcal{P}_{\mathcal{V}}\koop g-\koop g\|\leq\sum_{i=1}^q|g_i|\|\mathcal{P}_{\mathcal{V}}^*\mathcal{P}_{\mathcal{V}}v_i-v_i\|=\sum_{i=1}^q|\mathbf{g}^*\mathbf{L}\mathbf{v_i}|\sin\theta_i
\]
upper bounds the one-step error when applying the Koopman operator. We extend these ideas to provide tighter bounds on multi-step errors later in \cref{sec_kmd_eb}.

\subsection{Principal Angle Decomposition}\label{sec:pad}

To reduce computational costs and enforce model simplicity, approximations of the Koopman operator $\mathbf{K}\in\mathbb{C}^{N\times N}$ are often truncated. This is frequently done using a truncated singular value decomposition \cite{schmid2010dynamic,williams2015data,williams2015kernel}. This corresponds to projecting the Koopman operator onto the span of the $r<N$ eigenvectors of $\mathbf{G}$ with the largest eigenvalues. It produces an $r\times r$ rather than $N\times N$ Koopman matrix. Our algorithms lead to an alternative decomposition of the Koopman matrix that connects directly to properties of the infinite-dimensional Koopman operator, rather than just its finite-dimensional approximations. In \cref{sec:angles_numex} we shall see that this significantly improves performance.

Suppose we run EDMD with a dictionary of $N$ functions spanning a subspace $\mathcal{V}_N$. Then, applying \cref{alg:resDMD_angle} with $\mathbf{B}=\mathbf{I}_N$ we obtain principal angles $\{\theta_j\}_{j=1}^q$ and corresponding matrix of principal observables $\mathbf{U}_1$ (we will not use the coefficients associated with $v_j$ here). The columns of $\mathbf{U}_1$ are ordered in increasing order of angle $\theta_j$. We define $\mathbf{U}_1'=\mathbf{U}_1(1:q,:)+\mathbf{K}\mathbf{U}_1(q+1:2q,:)$ so that our representation of the principal observables is purely in terms of $\mathcal{V}_N$. Motivated by \cref{sec:residual_connect}, we expect that the modes corresponding to smaller angles are more accurately predicted by the approximate Koopman operator. Hence, given $r<N$ we define $\mathbf{U}=\mathbf{U}_1'(:,1:r)$, the matrix of the $r$ principal observables with the smallest angles. By restricting the action of the Koopman operator to these observables, we reduce the risk of spurious dynamics. In particular, we define the principal angle decomposition reduced matrices
\[\setlength\abovedisplayskip{6pt}\setlength\belowdisplayskip{6pt}
\mathbf{G}_{\mathrm{pad}}=\mathbf{U}^*\mathbf{G}\mathbf{U}=\mathbf{I}_r,\quad \mathbf{A}_{\mathrm{pad}}=\mathbf{U}^*\mathbf{A}\mathbf{U}=\mathbf{K}_{\mathrm{pad}},\quad \text{ and }\quad \mathbf{L}_{\mathrm{pad}}=\mathbf{U}^*\mathbf{L}\mathbf{U}.
\]
Then we can decompose an observable of interest in terms of the $r$ chosen principal observables and forecast its evolution using $\mathbf{K}_{\mathrm{pad}}$. This provides an analogue of the KMD using principal observables instead of eigenfunctions.
This algorithm is summarized in \cref{alg:pad}. 

\begin{algorithm}[t]
\textbf{Input:} Matrices $\mathbf{G}$, $\mathbf{A}$, $\mathbf{K}$ and $\mathbf{L}$ in \cref{quad_convergence,eqn:K_l2_defn,quad_convergence2} for dictionary $\{\psi_j\}_{j=1}^{N}$, cut-off tolerance $\epsilon_c\geq 0$, number of post-compression modes $r>0$. For RKHS, use instead the matrices $\mathbf{G}$, $\mathbf{A}$, $\mathbf{K}$ and $\mathbf{R}$ (instead of $\mathbf{L}$) from \cref{eqn:GA_ker_defn,eqn:K_ker_defn,eqn:L_ker_defn}.\\
\begin{algorithmic}[1]
\STATE Compute principal angles $\smash{\{\theta_j\}_{j=1}^q}$ and observables $\smash{\mathbf{U}_1}$ via \cref{alg:resDMD_angle} ($\mathbf{B}=\mathbf{I}_N$).\!\!\!\!\!
\STATE Define $\mathbf{U}_1'=\mathbf{U}_1(1:q,:)+\mathbf{K}\mathbf{U}_1(q+1:2q,:)$, and truncate to $\mathbf{U}=\mathbf{U}_1'(:,1:r).$
\STATE Construct the matrices $ \mathbf{A}_{\mathrm{pad}}=\mathbf{U}^*\mathbf{A}\mathbf{U}=\mathbf{K}_{\mathrm{pad}}$ and $ \mathbf{L}_{\mathrm{pad}}=\mathbf{U}^*\mathbf{L}\mathbf{U}.$
\end{algorithmic}
\textbf{Output: } The compressed matrices $\mathbf{A}_{\mathrm{pad}}$, $\mathbf{K}_{\mathrm{pad}}$ and $\mathbf{L}_{\mathrm{pad}}$.
\caption{Principal Angle Decomposition (PAD).}\label{alg:pad}
\end{algorithm}

\subsection{Numerical examples}\label{sec:angles_numex} We consider several illustrative examples of the algorithms developed in this section and their applications.

\subsubsection{Duffing oscillator}
Consider the unforced Duffing oscillator
\begin{equation}\setlength\abovedisplayskip{6pt}\setlength\belowdisplayskip{6pt}\label{duffing_eqn}
    \ddot{u}+0.05\dot{u}-u+u^3=0,\quad \text{and} \quad u(0),\,\dot{u}(0)\in\mathbb{R},
\end{equation}
with state $\smash{x(t)=(u(t),\,\dot{u}(t)
)^\top}$ for $t\geq 0$ and state space $\smash{\mathcal{X}=\mathbb{R}^2}$. This system models damped oscillations featuring in a wide variety of physical processes. We define the evolution map $F$ by $y=F(x)$, where $y$ is the solution of \cref{duffing_eqn} at time $\Delta t=1/4$ with initial condition $x$. 
 We generate snapshot data using MATLAB’s \texttt{ode45}, consisting of $50$ trajectories of $5$ seconds each from uniformly random initial conditions in $[-2,2]^2$, yielding $M=1000$ samples.
To run EDMD, we must choose a dictionary; we take $N=100$ tensor-product Chebyshev polynomials on $[-2.5,2.5]$ of degree at most $9$. To analyze the choice of dictionary, we apply \cref{alg:resDMD_angle}. The resulting principal angles are shown in \cref{fig:duffing_angles}. 
\begin{figure}[t]
    \centering
    \includegraphics[width=0.3\linewidth]{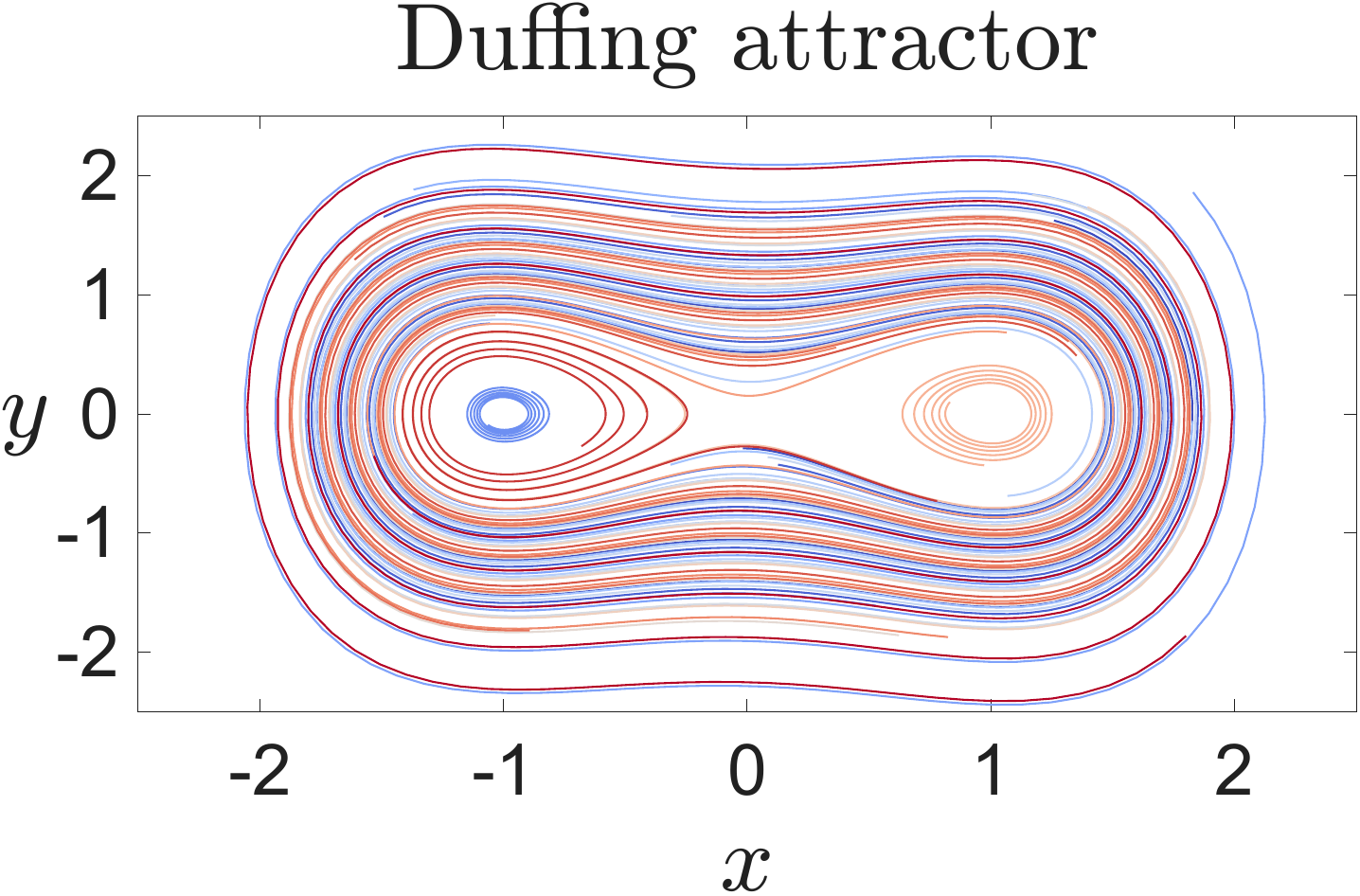}
    \includegraphics[height=0.2\linewidth]{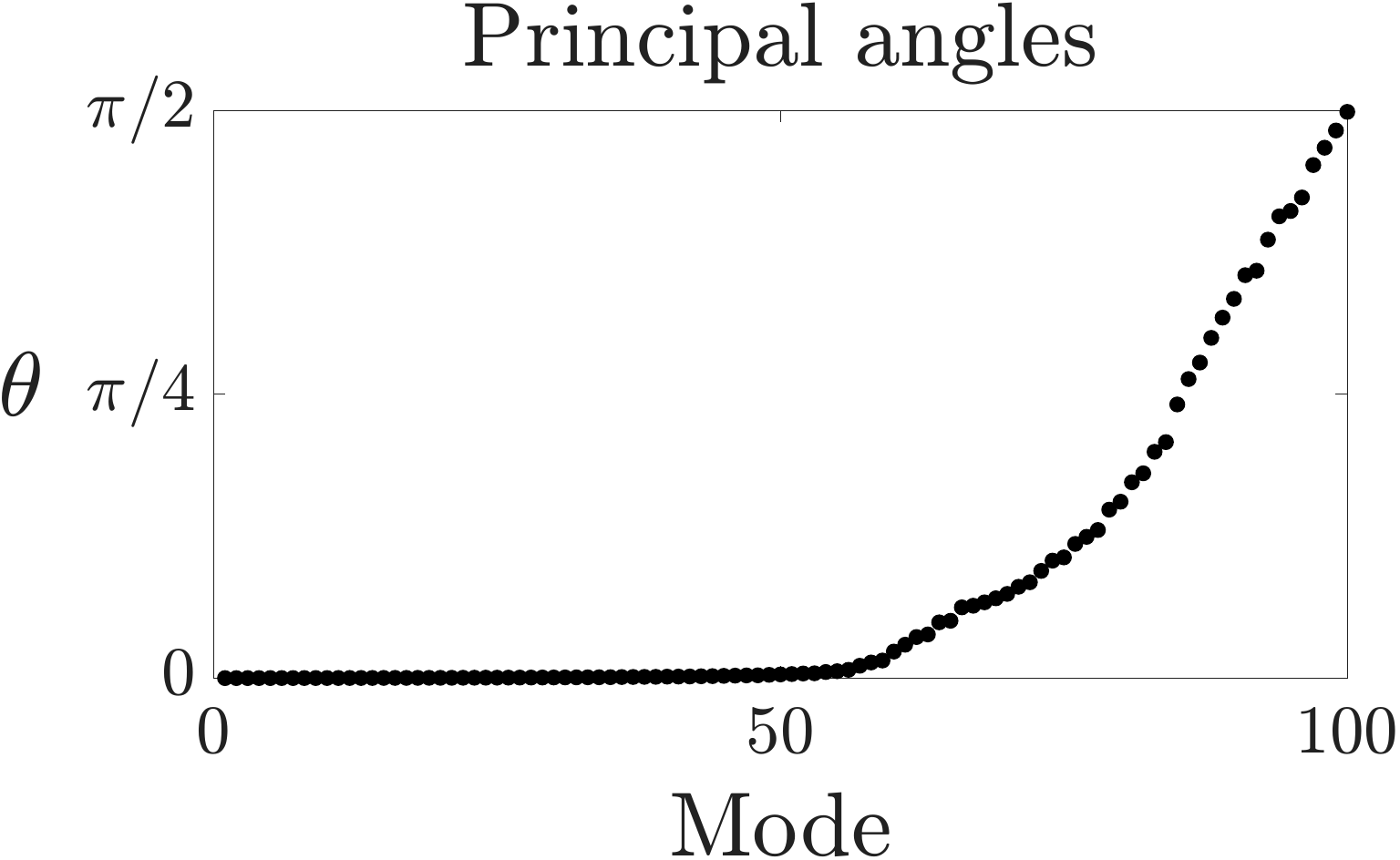}
    \includegraphics[height=0.2\linewidth]{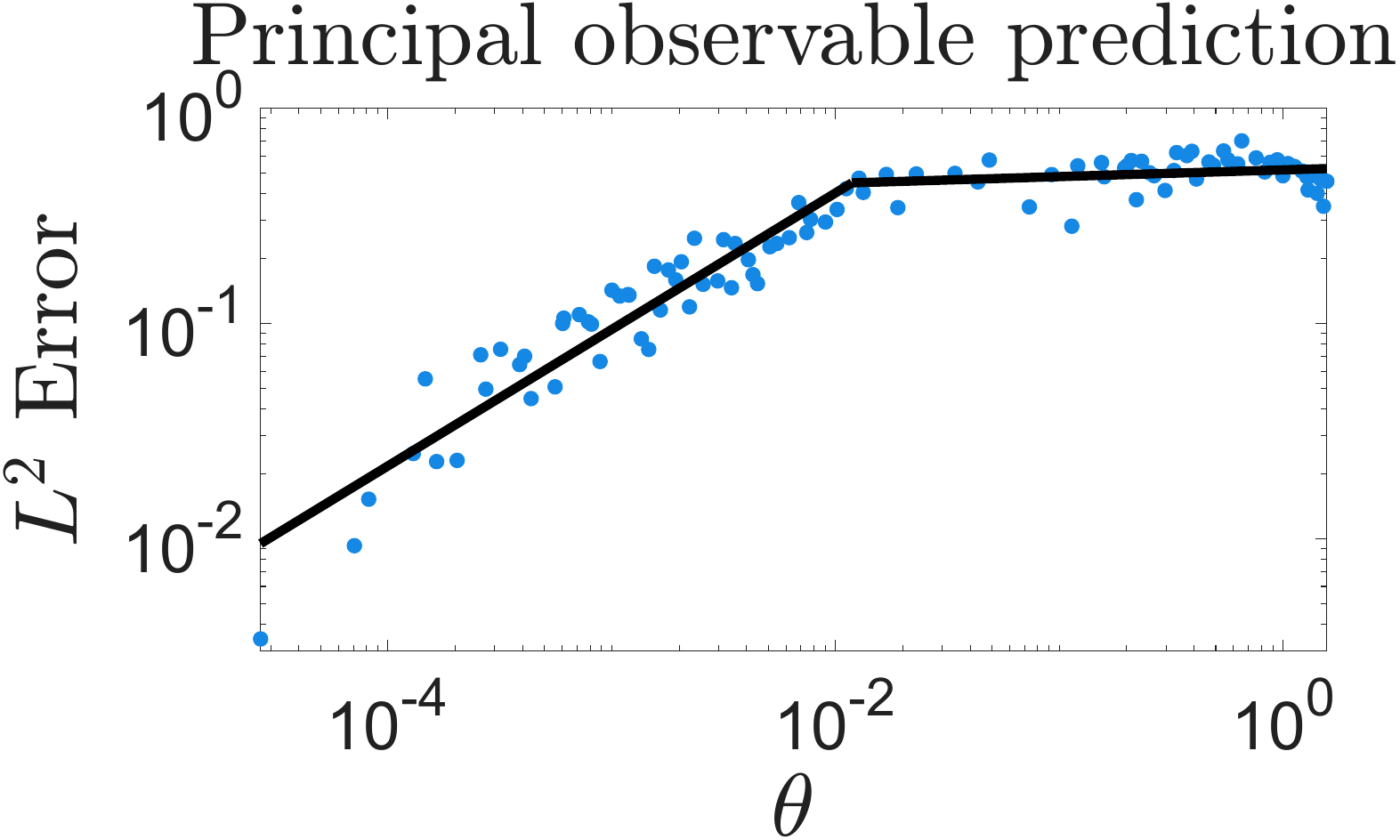}\vspace{0.2mm}
\vspace{-2mm}
    \caption{Left: The Duffing attractor. Middle: Principal angles on the Duffing oscillator for a finite subspace of Chebyshev polynomials. Right: One-step prediction errors for the corresponding principal observables. The solid line denotes a piecewise-linear best fit with a single breakpoint.}\vspace{-3mm}
    \label{fig:duffing_angles}
\end{figure}

The first  $60$ modes have small principal angles, with the angles increasing thereafter. To assess the impact of these angles on forecasting accuracy, for each principal observable
$u_j$ we compute $\|\koop u_j-\Psi\mathbf{K}\mathbf{u}_j\|$, which quantifies how well the projected operator approximates the true Koopman action. The results (right panel of \cref{fig:duffing_angles}) show a strong correlation between smaller principal angles and improved predictions. A piecewise-linear fit with a single breakpoint (using the method of~\cite{bogartz1968least}) indicates that beyond $\theta\approx0.01$ the error plateaus, motivating a distinction between “good” and “bad” observables.

\begin{figure}[t]
    \centering
    \includegraphics[height=0.3\linewidth]{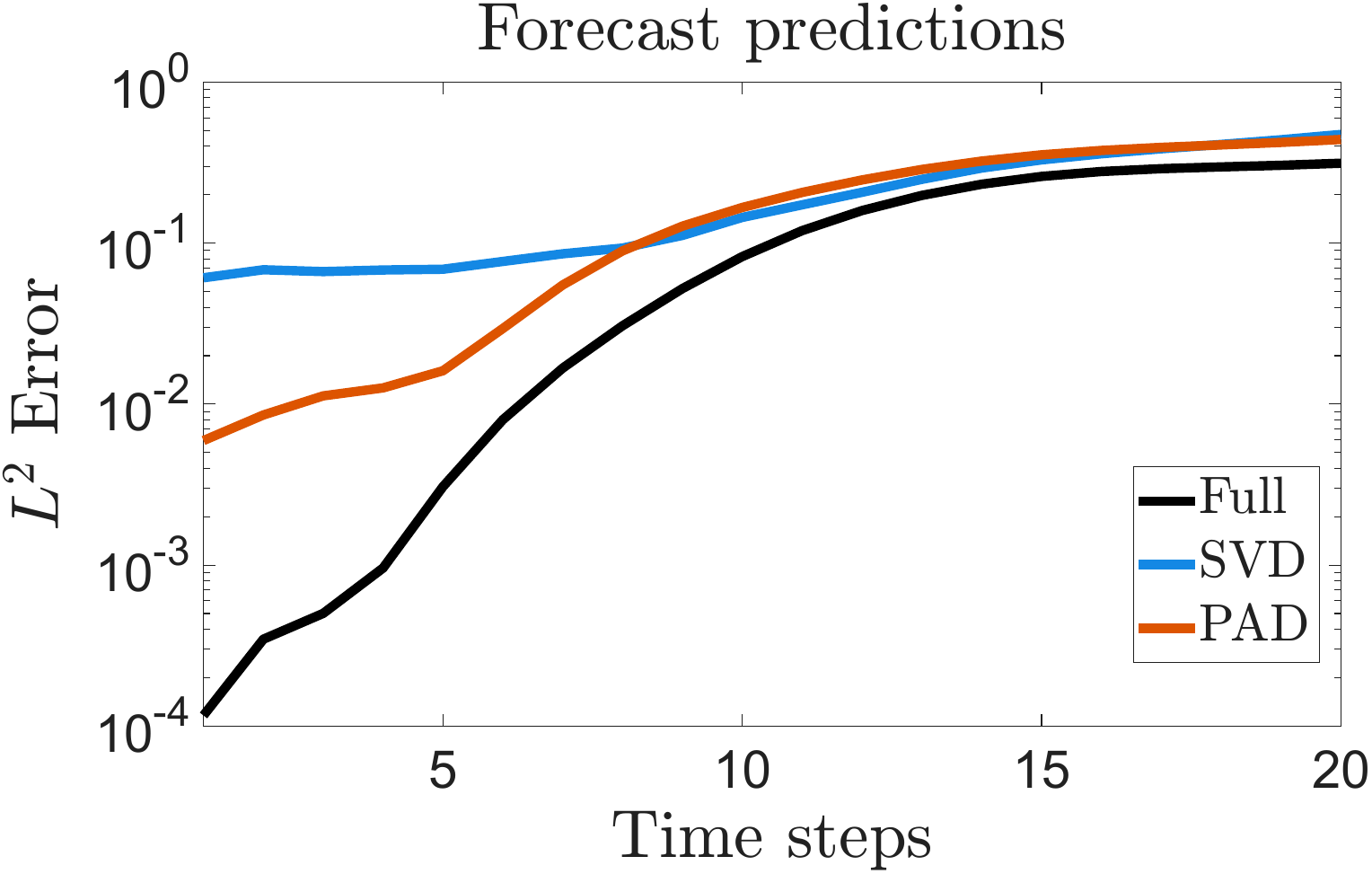}\hfill
        \includegraphics[height=0.3\linewidth]{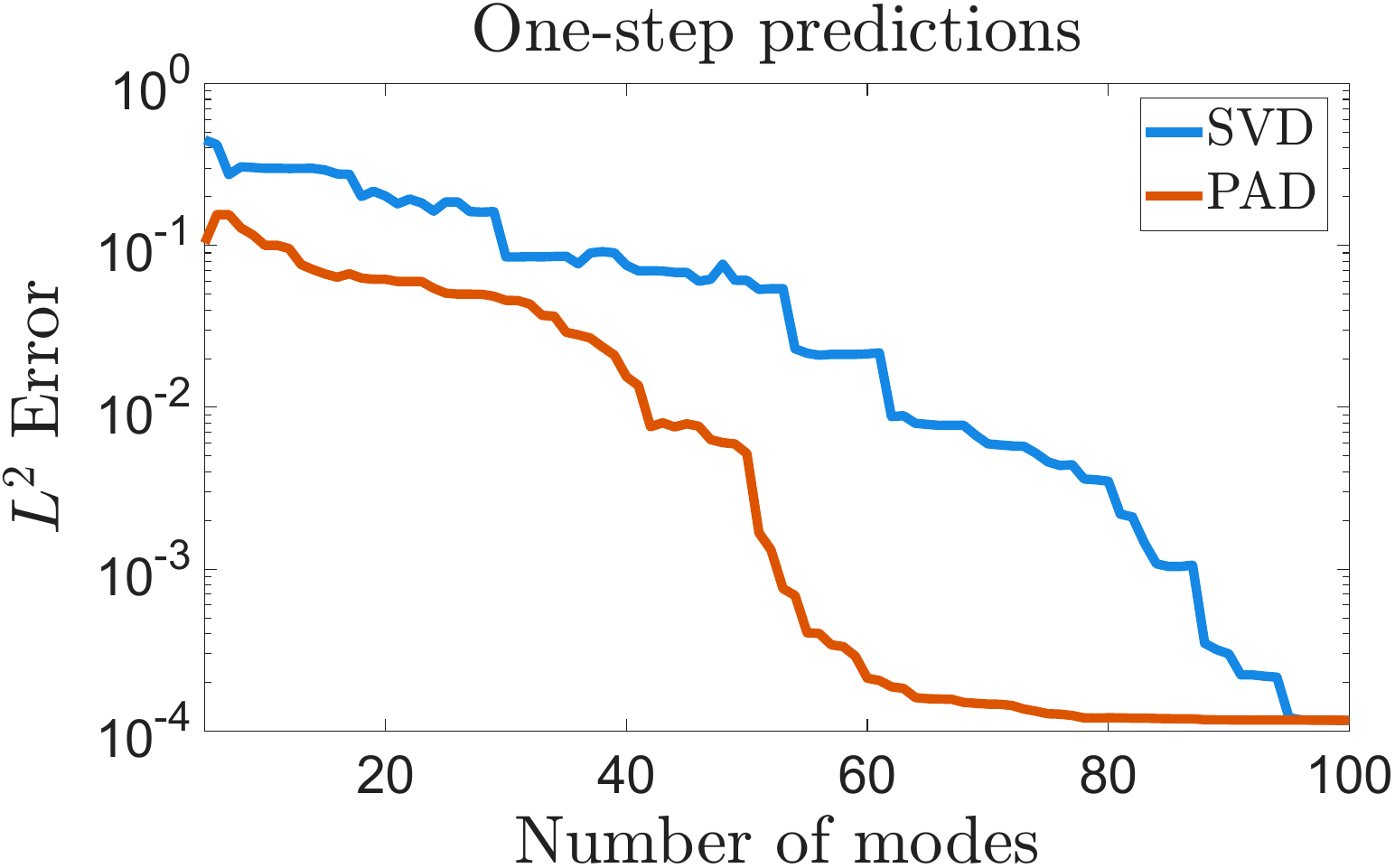}
    \caption{Left: Forecast errors for the full Koopman matrix on the Duffing oscillator versus the Koopman operator truncated using a SVD or PAD. Right: One-step prediction errors for Koopman operators truncated via PAD and SVD, shown as a function of the number of modes.}
    \label{fig:pad_svd_full_errors}
\end{figure}

Using this “good”/“bad” distinction, we set $r=\max\{s:\theta_s<0.01\}$ and apply \cref{alg:pad} to construct a reduced-order model $\mathbf{K}_{\mathrm{pad}}$. We compare its performance with that of an SVD-truncated Koopman operator using the same number of modes. \cref{fig:pad_svd_full_errors} (left) reports the prediction errors for the state observable $g(x)=x_1=u$ over $20$ time steps for the PAD truncation, the SVD truncation, and the full untruncated matrix. PAD outperforms SVD, particularly at short times, while their long-time performance is comparable. This is expected: the principal-angle construction aligns only the subspaces $\mathcal{V}$ and $\koop\mathcal{V}$.
To demonstrate that this behavior is not specific to a single choice of $r$, \cref{fig:pad_svd_full_errors} (right) shows the one-step prediction error as a function of the number of modes. PAD consistently achieves lower errors, and once the angle threshold $\theta=0.01$ is reached, additional modes yield little improvement, consistent with \cref{fig:duffing_angles}. By contrast, SVD truncation requires nearly the full-order model to reach comparable performance.

\subsubsection{Lorenz system}

As further illustration, we consider the Lorenz--63 system~\cite{lorenz1963deterministic}, a standard chaotic benchmark in data-driven dynamical systems:
\begin{equation}\setlength\abovedisplayskip{6pt}\setlength\belowdisplayskip{6pt}\label{lorenz_defn}\dot{u}_1=10(u_2-u_1),\quad \dot{u}_2=u_1(28-u_3)-u_2,\quad \dot{u}_3=u_1u_2-(8/3)u_3.
\end{equation}
For these parameters, almost all trajectories converge to the Lorenz attractor shown on the left of \cref{fig:pad_svd_no_modes}. We define the state as $x=(u_1,u_2,u_3)\in\mathbb{R}^3$ and the evolution map as the solution with timestep $\Delta t=0.1$. We restrict the dynamics to the Lorenz attractor; this yields a unitary Koopman operator with respect to the corresponding physical, or Sinai-Ruelle-Bowen, measure \cite{colbrook2025rigged,young2002srb}. For data, we generate a single long trajectory starting from a randomly selected initial condition using MATLAB's \texttt{ode45} function; we discard the initial transient before the trajectory enters the Lorenz attractor. We run EDMD using a dictionary of $100$ exponential radial basis functions (RBFs) $\psi_i(x)=e^{-s\|x-c_i\|}$ for randomly selected centres $\{c_i\}_{i=1}^N$ and scale factor $s$, and compare PAD and SVD truncations for predicting the observable $g(x)=x_1=u_1$; see \cref{fig:pad_svd_no_modes} for the results. As in the Duffing example, PAD outperforms SVD for short prediction horizons, achieving near-full accuracy with roughly half as many modes, whereas SVD requires nearly all modes to reach comparable performance.

\begin{figure}[t]
    \centering
      \includegraphics[height=0.2\linewidth]{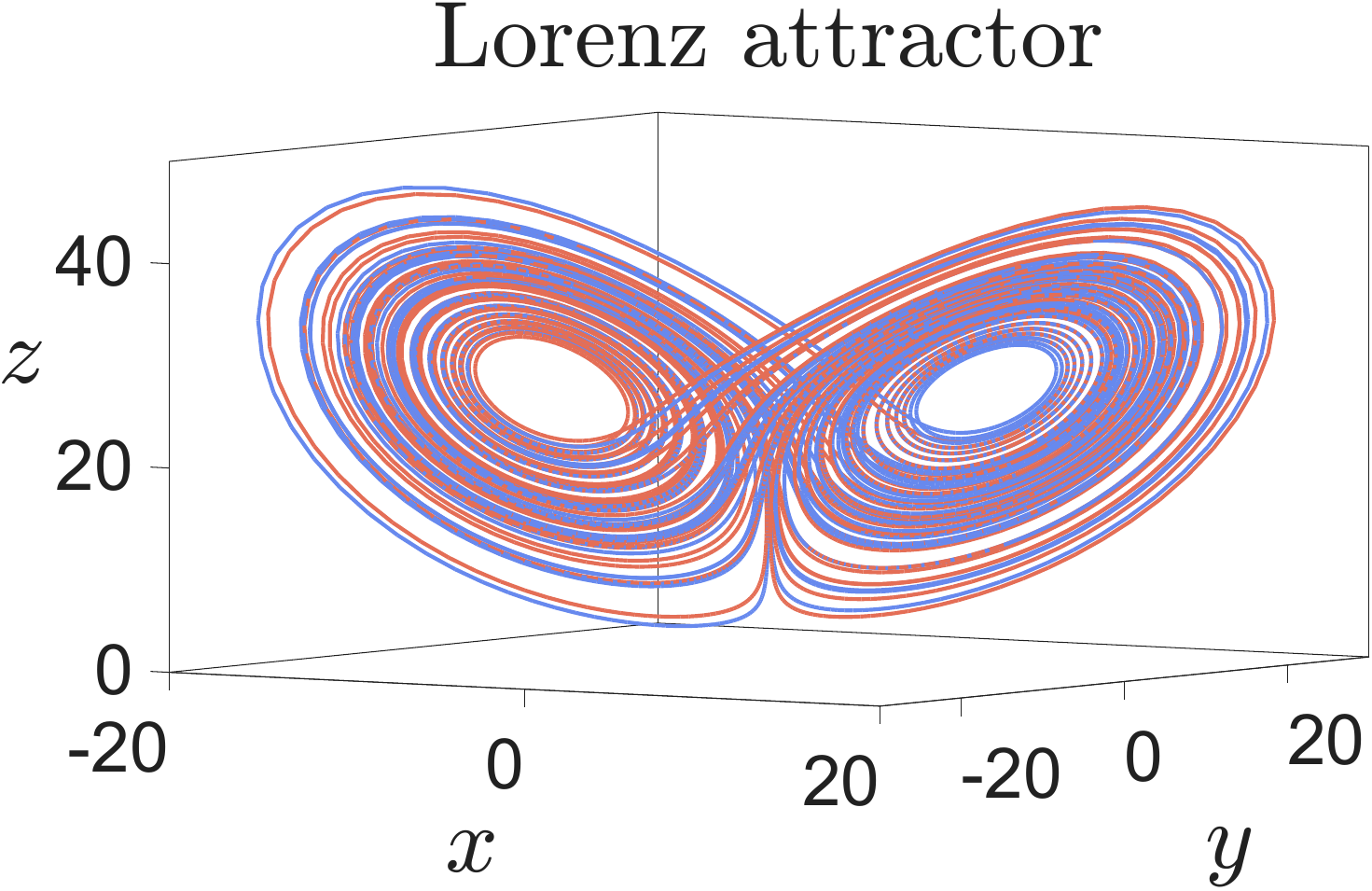}
      \includegraphics[height=0.2\linewidth]{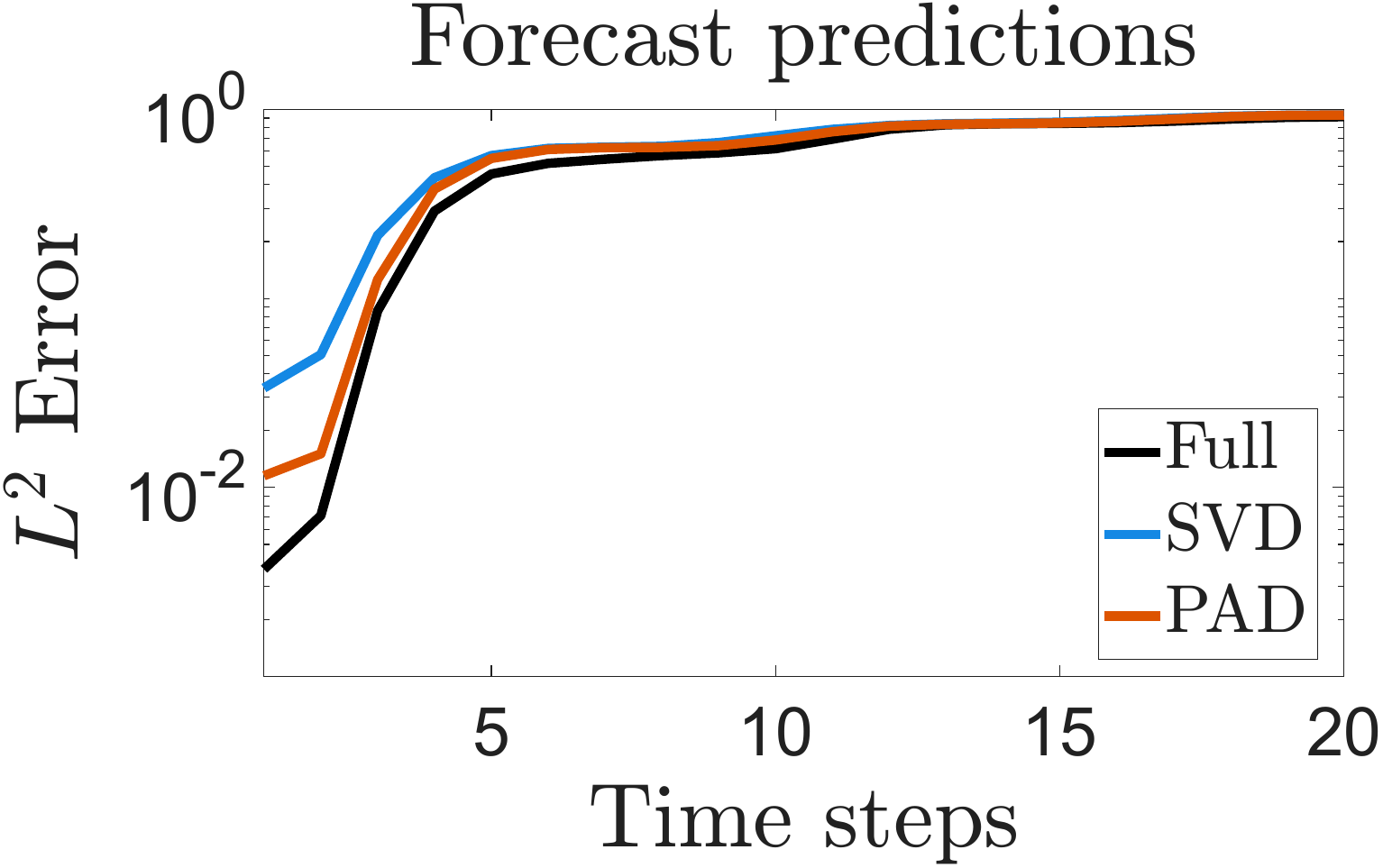}
        \includegraphics[height=0.2\linewidth]{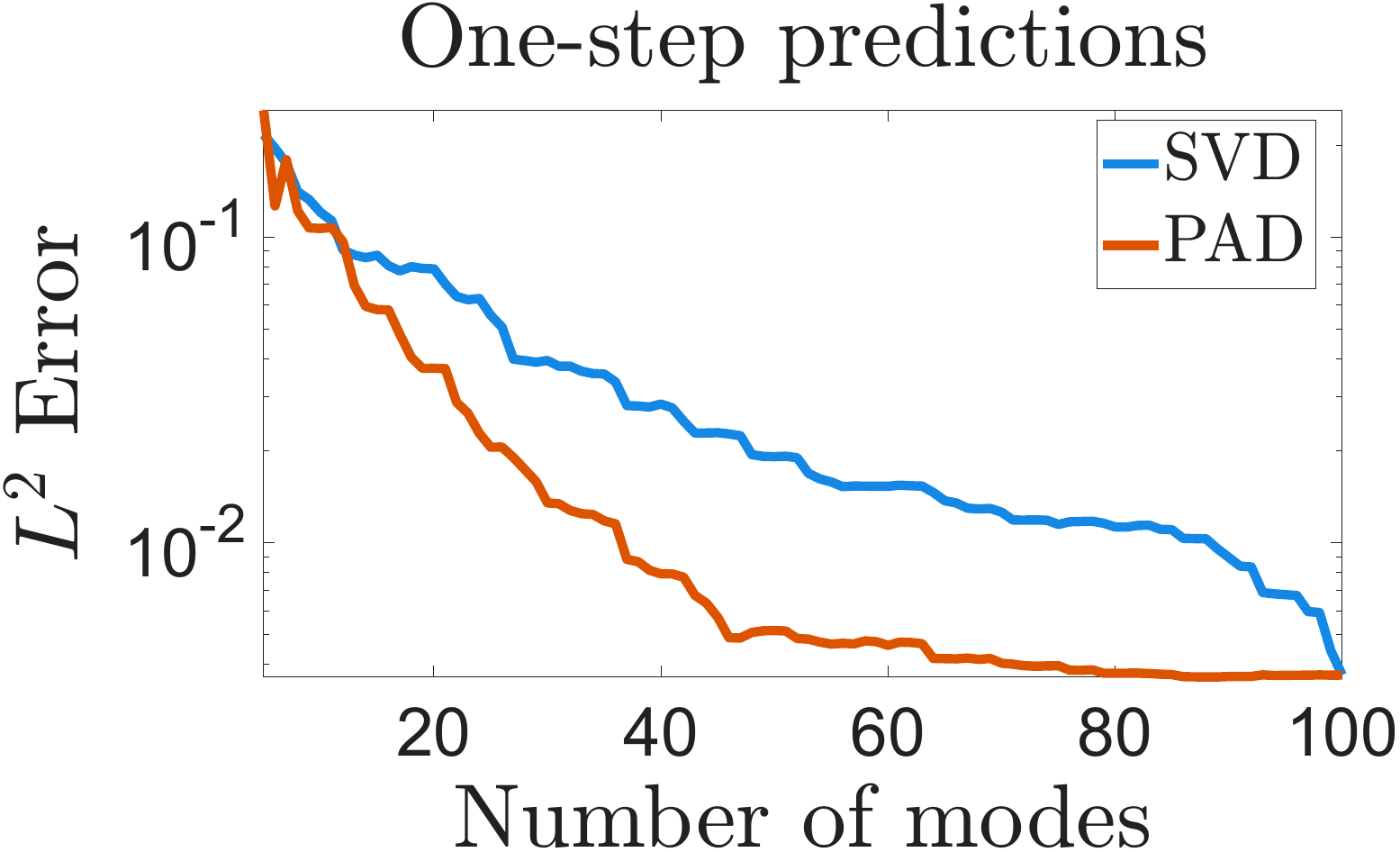}
    \caption{Left: Two trajectories on the Lorenz attractor. Middle: Exact prediction errors for the full Lorenz system and for truncated Koopman operators using PAD and SVD. Right: The dependence of one-step prediction error on the number of modes for both PAD and SVD.}
    \label{fig:pad_svd_no_modes}
\end{figure}

\section{Error bounds for Koopman mode decompositions}
\label{sec_kmd_eb}

The KMD in \cref{sec_application_error} expands an observable using approximate eigenfunctions. This expansion is inexact due to the finite basis and because $\mathcal{V}_N$ need not be invariant. In \cref{sec:Koopman_invariance_measure}, we introduced a measure of Koopman invariance for a subspace. Using similar ideas, we now bound the KMD error under repeated application of the Koopman operator. The analysis differs in two respects: first, the spaces are one-dimensional, corresponding to the observable being predicted, and second, the relevant subspace evolves with each application of the Koopman operator. The theoretical results in \cref{sec:decomp_kmd} apply to a general Hilbert space. \Cref{sec:comp_erbs} provides computable error bounds for $\koop$ on $L^2$ spaces, and \cref{sec-ptwise_erbs} extends the analysis to $\koop^*$ on an RKHS, where pointwise error bounds are also obtained.

\subsection{Decomposing the Koopman mode decomposition}\label{sec:decomp_kmd}

Let $\mathcal{H}$ be a Hilbert space with norm $\|\cdot \|$ and $\mathcal{V}_N\subset\mathcal{H}$ be a subspace spanned by a chosen dictionary. If $g\in\mathcal{V}_N$, there exists $\mathbf{g}\in\mathbb{C}^N$ such that $g=\Psi\mathbf{g}$, and the projection error in the KMD is given by $\|\mathcal{K}^{n}g-\Psi\mathbf{K}^n\mathbf{g}\|$. When $g\notin\mathcal{V}_N$, with a slight abuse of notation we let $\Psi\mathbf{g}$ be the least squares approximation to $g$ in $\mathcal{V}_N$, and control the error by
\[\setlength\abovedisplayskip{6pt}\setlength\belowdisplayskip{6pt}
\|\koop^ng-\Psi\mathbf{K}^n\mathbf{g}\|\leq\|\koop^n\|\!\!\!\!\!\!\!\underbrace{\|g-\Psi\mathbf{g}\|}_{\text{initialization error}}\!\!\!\!\!\!\!+\underbrace{\|\koop^n(\Psi\mathbf{g})-\Psi\mathbf{K}^n\mathbf{g}\|}_{\text{projection error}}.
\]
The initialization error term is straightforward to approximate numerically, so for the remainder of this section we assume that $g\in\mathcal{V}_N$ and so $g=\Psi\mathbf{g}$. Moreover, in the $L^2$ case we can eliminate the initialization error by adding $g$ to the dictionary. The following theorem provides upper bounds on the projection error.
\begin{theorem}[KMD projection error]\label{kmd_norm_bound_thm}
For every $g\in \mathcal{V}_N$ and $n\in\mathbb{N}$, we have
\begin{align}
\setlength\abovedisplayskip{1pt}\setlength\belowdisplayskip{0pt}
\left\|\mathcal{K}^n\mathcal{P}_{\mathcal{V}_N}^*g{-}\mathcal{P}_{\mathcal{V}_N}^*(\mathcal{P}_{\mathcal{V}_N}\mathcal{K}\mathcal{P}_{\mathcal{V}_N}^*)^{n}g\right\| &\leq \sqrt{\sum_{j=0}^{n-1}    \left\|\mathcal{P}_{\mathcal{V}_N}\mathcal{K}\right\|^{2j}\left\|\mathcal{Q}_{\mathcal{V}_N}\mathcal{K}^{n-j}\mathcal{P}_{\mathcal{V}_N}^*g\right\|^2}, \label{bound_lemma}\\\left\|\mathcal{K}^n\mathcal{P}_{\mathcal{V}_N}^*g{-}\mathcal{P}_{\mathcal{V}_N}^*(\mathcal{P}_{\mathcal{V}_N}\mathcal{K}\mathcal{P}_{\mathcal{V}_N}^*)^{n}g\right\| &\leq\sum_{j=0}^{n-1}    \left\|(\mathcal{P}_{\mathcal{V}_N}^*\mathcal{P}_{\mathcal{V}_N}\mathcal{K})^j\right\|\left\|\mathcal{Q}_{\mathcal{V}_N}\mathcal{K}^{n-j}\mathcal{P}_{\mathcal{V}_N}^*g\right\|. \label{bound_lemma2}
\end{align}
\end{theorem}
To prove \cref{kmd_norm_bound_thm} we use the following lemma, which provides a helpful decomposition of $\mathcal{P}_{\mathcal{V}_N}\mathcal{K}^n\mathcal{P}_{\mathcal{V}_N}^*$.
\begin{lemma}\label{kmd_bound_thm}
Let $\mathcal{Q}_{\mathcal{V}_N}$ be the orthogonal projection onto $\mathcal{V}_N^\perp$. Then for all $n\in\mathbb{N}$,
\begin{equation}\setlength\abovedisplayskip{3pt}\setlength\belowdisplayskip{3pt}
\mathcal{P}_{\mathcal{V}_N}\mathcal{K}^n\mathcal{P}_{\mathcal{V}_N}^*-(\mathcal{P}_{\mathcal{V}_N}\mathcal{K}\mathcal{P}_{\mathcal{V}_N}^*)^{n}=\sum_{j=1}^{n-1} \mathcal{P}_{\mathcal{V}_N}(\mathcal{P}_{\mathcal{V}_N}^*\mathcal{P}_{\mathcal{V}_N}\mathcal{K})^{j}\mathcal{Q}_{\mathcal{V}_N}^*\mathcal{Q}_{\mathcal{V}_N}\mathcal{K}^{n-j}\mathcal{P}_{\mathcal{V}_N}^*.\label{orthog_lemma}
\end{equation}
\end{lemma}
\begin{proof}
The expression holds for $n=1$. By orthogonality,
\[\setlength\abovedisplayskip{6pt}\setlength\belowdisplayskip{6pt}
\mathcal{P}_{\mathcal{V}_N}\mathcal{K}^{n+1}\mathcal{P}_{\mathcal{V}_N}^*=\mathcal{P}_{\mathcal{V}_N}\mathcal{K}\mathcal{Q}_{\mathcal{V}_N}^*\mathcal{Q}_{\mathcal{V}_N}\mathcal{K}^n\mathcal{P}_{\mathcal{V}_N}^*+\mathcal{P}_{\mathcal{V}_N}\mathcal{K}\mathcal{P}_{\mathcal{V}_N}^*\mathcal{P}_{\mathcal{V}_N}\mathcal{K}^n\mathcal{P}_{\mathcal{V}_N}^*.
\]
Hence if \cref{orthog_lemma} holds for $n$, then\begin{equation*}\setlength\abovedisplayskip{6pt}\setlength\belowdisplayskip{6pt}\begin{split}
\mathcal{P}_{\mathcal{V}_N}\mathcal{K}^{n+1}\mathcal{P}_{\mathcal{V}_N}^*&=\mathcal{P}_{\mathcal{V}_N}\mathcal{K}\mathcal{Q}_{\mathcal{V}_N}^*\mathcal{Q}_{\mathcal{V}_N}\mathcal{K}^n\mathcal{P}_{\mathcal{V}_N}^*+\mathcal{P}_{\mathcal{V}_N}\mathcal{K}\mathcal{P}_{\mathcal{V}_N}^*(\mathcal{P}_{\mathcal{V}_N}\mathcal{K}\mathcal{P}_{\mathcal{V}_N}^*)^{n}\\
&\quad\quad+\mathcal{P}_{\mathcal{V}_N}\mathcal{K}\mathcal{P}_{\mathcal{V}_N}^*\sum_{j=1}^{n-1} \mathcal{P}_{\mathcal{V}_N}(\mathcal{P}_{\mathcal{V}_N}^*\mathcal{P}_{\mathcal{V}_N}\mathcal{K})^{j}\mathcal{Q}_{\mathcal{V}_N}^*\mathcal{Q}_{\mathcal{V}_N}\mathcal{K}^{n-j}\mathcal{P}_{\mathcal{V}_N}^*\\
&=(\mathcal{P}_{\mathcal{V}_N}\mathcal{K}\mathcal{P}_{\mathcal{V}_N}^*)^{n+1}+\sum_{j=1}^{n} \mathcal{P}_{\mathcal{V}_N}(\mathcal{P}_{\mathcal{V}_N}^*\mathcal{P}_{\mathcal{V}_N}\mathcal{K})^{j}\mathcal{Q}_{\mathcal{V}_N}^*\mathcal{Q}_{\mathcal{V}_N}\mathcal{K}^{n+1-j}\mathcal{P}_{\mathcal{V}_N}^*,
\end{split}\end{equation*}
so that \cref{orthog_lemma} holds for $n+1$ and hence for all $n$ by induction.
\end{proof}
We may now prove \cref{kmd_norm_bound_thm}.
\begin{proof}[Proof of \cref{kmd_norm_bound_thm}]
By \cref{kmd_bound_thm} (after premultiplying by $\mathcal{P}_{\mathcal{V}_N}^*$),
\begin{equation*}\setlength\abovedisplayskip{6pt}\setlength\belowdisplayskip{2pt}\begin{split}
\mathcal{K}^n\mathcal{P}_{\mathcal{V}_N}^*g-\mathcal{P}_{\mathcal{V}_N}^*(\mathcal{P}_{\mathcal{V}_N}\mathcal{K}\mathcal{P}_{\mathcal{V}_N}^*)^{n}g&=\mathcal{P}_{\mathcal{V}_N}^*(\mathcal{P}_{\mathcal{V}_N}\mathcal{K}^n\mathcal{P}_{\mathcal{V}_N}^*)g-\mathcal{P}_{\mathcal{V}_N}^*(\mathcal{P}_{\mathcal{V}_N}\mathcal{K}\mathcal{P}_{\mathcal{V}_N}^*)^{n}g\\
&\quad\quad\quad\quad\quad+\mathcal{Q}_{\mathcal{V}_N}^*\mathcal{Q}_{\mathcal{V}_N}\mathcal{K}^n\mathcal{P}_{\mathcal{V}_N}^*g\\
&=\sum_{j=0}^{n-1} (\mathcal{P}_{\mathcal{V}_N}^*\mathcal{P}_{\mathcal{V}_N}\mathcal{K})^{j}\mathcal{Q}_{\mathcal{V}_N}^*\mathcal{Q}_{\mathcal{V}_N}\mathcal{K}^{n-j}\mathcal{P}_{\mathcal{V}_N}^*g,\end{split}
\end{equation*}
and \cref{bound_lemma2} follows from the triangle inequality and operator norm submultiplicativity. For \cref{bound_lemma}, by orthogonality we have that
\begin{equation*}\setlength\abovedisplayskip{2pt}\setlength\belowdisplayskip{2pt}\begin{split}
&\tilde{E}_n(g,\mathcal{V}_N):=\Big\|\sum_{j=0}^{n-1} (\mathcal{P}_{\mathcal{V}_N}^*\mathcal{P}_{\mathcal{V}_N}\mathcal{K})^{j}\mathcal{Q}_{\mathcal{V}_N}^*\mathcal{Q}_{\mathcal{V}_N}\mathcal{K}^{n-j}\mathcal{P}_{\mathcal{V}_N}^*g\Big\|^2\\
\;&=\left\|\mathcal{Q}_{\mathcal{V}_N}^*\mathcal{Q}_{\mathcal{V}_N}\mathcal{K}^{n}\mathcal{P}_{\mathcal{V}_N}^*g\right\|^2+\Big\|\sum_{j=1}^{n-1} (\mathcal{P}_{\mathcal{V}_N}^*\mathcal{P}_{\mathcal{V}_N}\mathcal{K})^{j}\mathcal{Q}_{\mathcal{V}_N}^*\mathcal{Q}_{\mathcal{V}_N}\mathcal{K}^{n-j}\mathcal{P}_{\mathcal{V}_N}^*g\Big\|^2\\
\;&\leq \left\|\mathcal{Q}_{\mathcal{V}_N}\mathcal{K}^{n}\mathcal{P}_{\mathcal{V}_N}^*g\right\|^2 \!+\!\left\|\mathcal{P}_{\mathcal{V}_N}\mathcal{K}\right\|^2\Big\|\sum_{j=0}^{n-2} (\mathcal{P}_{\mathcal{V}_N}^*\mathcal{P}_{\mathcal{V}_N}\mathcal{K})^{j}\mathcal{Q}_{\mathcal{V}_N}^*\mathcal{Q}_{\mathcal{V}_N}\mathcal{K}^{n-1-j}\mathcal{P}_{\mathcal{V}_N}^*g\Big\|^2\!\!.
\end{split}
\end{equation*}
We can continue this process recursively to arrive at
\[\setlength\abovedisplayskip{2pt}\setlength\belowdisplayskip{2pt}
\tilde{E}_n(g,\mathcal{V}_N)\leq \sum_{j=0}^{n-1}  \left\|\mathcal{P}_{\mathcal{V}_N}\mathcal{K}\right\|^{2j}\left\|\mathcal{Q}_{\mathcal{V}_N}\mathcal{K}^{n-j}\mathcal{P}_{\mathcal{V}_N}^*g\right\|^2,
\]
which completes the proof. 
\end{proof}\linespread{0.97}

\cref{kmd_norm_bound_thm} yields two different error bounds. The bound in \cref{bound_lemma2} can be advantageous as often $\|\koop^j\|$ is much smaller than $\|\koop\|^j$; indeed, many dynamical systems have $\|\koop\|>1$ but spectral radius $\rho(\koop)=1$, and so by the Gelfand spectral radius theorem~\cite{gelfand1941normierte}, $\lim_{k\rightarrow\infty}\|\koop^k\|^{1/k}=1<\|\koop\|$ and $\|\koop^j\|$ must grow slower than $\|\koop\|^j$. However, the sum of terms is larger than the square root of the sum of the squared terms. Hence, it is recommended to use \cref{bound_lemma} when $\|\koop\|=1$ (e.g., when the dynamical system is measure-preserving in the $L^2$ case) or $\|\koop\|$ is close to $1$, and \cref{bound_lemma2} otherwise; for RKHSs typically $\|\koop\|>1$~\cite[Cor.~B.3]{kohne_linfty-error_2024} and the second formula is preferred.

\subsection{Computing error bounds on \texorpdfstring{$L^2$}{L2}}
\label{sec:comp_erbs}

To compute the upper bounds in \cref{kmd_norm_bound_thm}, we use the ResDMD or SpecRKHS matrices from \cref{sec:spectral_convergence}. We must compute the terms $\|\mathcal{Q}_{\mathcal{V}_N}\koop^k\mathcal{P}_{\mathcal{V}_N}^*g\|$ and $\|\mathcal{P}_{\mathcal{V}_N}\koop^k\|$. For any $k\in\mathbb{N}$, we have
\begin{equation*}\setlength\abovedisplayskip{6pt}\setlength\belowdisplayskip{6pt}
\begin{split}
\mathcal{Q}_{\mathcal{V}_N}\mathcal{K}^{k}\mathcal{P}_{\mathcal{V}_N}^*{=}\mathcal{Q}_{\mathcal{V}_N}\mathcal{K}\mathcal{P}_{\mathcal{V}_N}^*(\mathcal{P}_{\mathcal{V}_N}\mathcal{K}\mathcal{P}_{\mathcal{V}_N}^*)^{k-1}{+}\mathcal{Q}_{\mathcal{V}_N}\mathcal{K}[\mathcal{K}^{k-1}\mathcal{P}_{\mathcal{V}_N}^*{-}\mathcal{P}_{\mathcal{V}_N}^*(\mathcal{P}_{\mathcal{V}_N}\mathcal{K}\mathcal{P}_{\mathcal{V}_N}^*)^{k-1}].
\end{split}
\end{equation*}
It is possible to use this relation and \cref{bound_lemma} recursively to obtain an explicit formula bounding $\|\mathcal{Q}_{\mathcal{V}_N}\mathcal{K}^{n-j}\mathcal{P}_{\mathcal{V}_N}^*g\|$ by terms of the form $\|\mathcal{Q}_{\mathcal{V}_N}\mathcal{K}\mathcal{P}_{\mathcal{V}_N}^*v\|$ for $v\in\mathcal{V}_N$. As a first-order approximation, we keep only the first-order terms in powers of $\mathcal{Q}_{\mathcal{V}_N}^*$:
\begin{equation}\setlength\abovedisplayskip{2pt}\setlength\belowdisplayskip{6pt}\label{eqn_first_order_approx}
\begin{split}
&\left\|\mathcal{Q}_{\mathcal{V}_N}\mathcal{K}^{n-j}\mathcal{P}_{\mathcal{V}_N}^*g\right\|^2\approx \left\|\mathcal{Q}_{\mathcal{V}_N}\mathcal{K}\mathcal{P}_{\mathcal{V}_N}^*(\mathcal{P}_{\mathcal{V}_N}\mathcal{K}\mathcal{P}_{\mathcal{V}_N}^*)^{n-j-1}g\right\|^2\\
&\quad\quad=\left\|\mathcal{K}\mathcal{P}_{\mathcal{V}_N}^*(\mathcal{P}_{\mathcal{V}_N}\mathcal{K}\mathcal{P}_{\mathcal{V}_N}^*)^{n-j-1}g\right\|^2-\left\|(\mathcal{P}_{\mathcal{V}_N}\mathcal{K}\mathcal{P}_{\mathcal{V}_N}^*)(\mathcal{P}_{\mathcal{V}_N}\mathcal{K}\mathcal{P}_{\mathcal{V}_N}^*)^{n-j-1}g\right\|^2\\
&\quad\quad=\lim_{M\rightarrow\infty}(\mathbf{K}^{n-j-1}\mathbf{g})^*(\mathbf{L}-\mathbf{K}^*\mathbf{G}\mathbf{K})\mathbf{K}^{n-j-1}\mathbf{g},
\end{split}
\end{equation}
where the first equality follows from orthogonality and the second equality holds assuming a convergent quadrature rule. 

Next, we make the approximation $\|\mathcal{P}_{\mathcal{V}_N}\koop^j\|\approx\|(\mathcal{P}_{\mathcal{V}_N}\koop\mathcal{P}_{\mathcal{V}_N}^*)^j\mathcal{P}_{\mathcal{V}_N}\|,$ which converges in the large subspace limit as $N\rightarrow\infty$. Then,
\begin{equation}\setlength\abovedisplayskip{6pt}\setlength\belowdisplayskip{6pt}\label{eqn:operator_norm_est}
\begin{split}
&\|(\mathcal{P}_{\mathcal{V}_N}\koop\mathcal{P}_{\mathcal{V}_N}^*)^j\mathcal{P}_{\mathcal{V}_N}\|^2=\sup_{g\in\mathcal{V}_N}\cfrac{(\mathbf{K}^j\mathbf{g})^*\mathbf{G}(\mathbf{K}^j\mathbf{g})}{\mathbf{g}^*\mathbf{Gg}}\\
&\quad\quad=\sup_{h\in\mathcal{V}_N}\cfrac{(\mathbf{G}^{1/2}\mathbf{K}^j\mathbf{G}^{-1/2}\mathbf{h})^*(\mathbf{G}^{1/2}\mathbf{K}^j\mathbf{G}^{-1/2}\mathbf{h})}{\mathbf{h}^*\mathbf{h}}=\|\mathbf{G}^{1/2}\mathbf{K}^j\mathbf{G}^{-1/2}\|^2.
\end{split}
\end{equation}
Combining \cref{eqn_first_order_approx,eqn:operator_norm_est} with \cref{kmd_norm_bound_thm} yields the first-order upper bounds:
\begin{align}\setlength\abovedisplayskip{2pt}\setlength\belowdisplayskip{2pt}
\label{KMD_bound_first_order}
E_n^{(1)}(g,\mathcal{V}_N)&\coloneqq \sqrt{\sum_{j=0}^{n-1}  
\|\mathbf{G}^{1/2}\mathbf{K}\mathbf{G}^{-1/2}\|^{2j}
(\mathbf{K}^{n-j-1}\mathbf{g})^*(\mathbf{L}-\mathbf{K}^*\mathbf{G}\mathbf{K})\mathbf{K}^{n-j-1}\mathbf{g}},\\
\label{KMD_bound_first_order_2}
E_n^{(2)}(g,\mathcal{V}_N)&\coloneqq \sum_{j=0}^{n-1}  
\|\mathbf{G}^{1/2}\mathbf{K}^j\mathbf{G}^{-1/2}\|
\sqrt{(\mathbf{K}^{n-j-1}\mathbf{g})^*(\mathbf{L}-\mathbf{K}^*\mathbf{G}\mathbf{K})\mathbf{K}^{n-j-1}\mathbf{g}}.
\end{align}

\begin{algorithm}[t]
\textbf{Input:} Matrices $\mathbf{G}$, $\mathbf{A}$, $\mathbf{K}$ and $\mathbf{L}$ in \cref{quad_convergence,eqn:K_l2_defn,quad_convergence2} for dictionary $\{\psi_j\}_{j=1}^{N}$, observable $g=\Psi\mathbf{g}$, number of timesteps $n\in\mathbb{N}$.\\
\begin{algorithmic}[1]
\STATE For $j=0,\ldots,n-1$, compute
$
\|\mathbf{G}^{1/2}\mathbf{K}^j\mathbf{G}^{-1/2}\|$ and $ (\mathbf{K}^{j}\mathbf{g})^*(\mathbf{L}-\mathbf{K}^*\mathbf{G}\mathbf{K})\mathbf{K}^{j}\mathbf{g}.
$
\STATE Compute
$E_n^{(1)}$ and $E_n^{(2)}$ following \cref{KMD_bound_first_order,KMD_bound_first_order_2},
and take $E_n=\min\{E_n^{(1)},E_n^{(2)}\}$.
\end{algorithmic} \textbf{Output:} $E_n$, a first-order error bound for the KMD of $g$ after $n$ time-steps.
\caption{Computation of error bounds for the KMD on $L^2$.}\label{alg:resDMD_KMD}
\end{algorithm}

\cref{alg:resDMD_KMD} summarizes the method. As $N\rightarrow \infty$, $\mathcal{Q}_{\mathcal{V}_N}g\rightarrow 0$, so both the exact error bounds \cref{bound_lemma,bound_lemma2} and the corresponding first-order bounds \cref{KMD_bound_first_order,KMD_bound_first_order_2} tend to $0$ (after taking the large-data limit in the latter case). Note that the one-step prediction error is always exact (up to quadrature error). This follows from Pythagoras’ theorem for observables $g\in\mathcal{V}_N$:
\[\setlength\abovedisplayskip{2pt}\setlength\belowdisplayskip{2pt}
\|\koop g-\mathcal{P}_{\mathcal{V}_N}^*\mathcal{P}_{\mathcal{V}_N}\koop g\|^2=\|\koop g\|^2-\|\mathcal{P}_{\mathcal{V}_N}^*\mathcal{P}_{\mathcal{V}_N}\koop g\|^2=\mathbf{g}^*\mathbf{L}\mathbf{g}-(\mathbf{K}\mathbf{g})^*\mathbf{G}(\mathbf{K}\mathbf{g}).
\]
The right-hand side coincides exactly with the output of \cref{alg:resDMD_KMD} for $n=1$.

\subsection{Pointwise error bounds in RKHS}\label{sec-ptwise_erbs}

\begin{algorithm}[t]
\textbf{Input:} Matrices $\mathbf{G}$, $\mathbf{A}$, $\mathbf{K}$ and $\mathbf{R}$ from \cref{eqn:GA_ker_defn,eqn:K_ker_defn,eqn:L_ker_defn} for kernel function $\mathfrak{K}:\mathcal{X}\times\mathcal{X}\rightarrow\mathbb{C}$, starting point $x_0$, observable $g\in\mathcal{H}$, number of timesteps $n\in\mathbb{N}$.\\
\begin{algorithmic}[1]
\STATE Let $\mathbf{b}=\begin{bmatrix}
    \mathfrak{K}(x_0,x^{(1)})&\dots &\mathfrak{K}(x_0,x^{(N)})
\end{bmatrix}^\top$, $\mathbf{c}=\mathbf{G}^{-1}\mathbf{b}$ and  
$\delta=\sqrt{\mathfrak{K}(x_0,x_0)-\mathbf{c}^*\mathbf{b}}$.
\STATE For $j=0,\ldots,n-1$, compute
$
\|\mathbf{G}^{1/2}\mathbf{K}^j\mathbf{G}^{-1/2}\|$ and $ 
(\mathbf{K}^{j}\mathbf{c})^*(\mathbf{R}-\mathbf{K}^*\mathbf{G}\mathbf{K})(\mathbf{K}^{j}\mathbf{c}).
$
\STATE Estimate $\|g\|$ and compute $E_n$ according to \cref{eqn:init_proj_error}.
\end{algorithmic} \textbf{Output:} First-order error bounds for the PFMD of $g$ at $x_0$ after $n$ time-steps.
\caption{Computation of error bounds for the PFMD on an RKHS.}\label{alg:specrkhs_pfmd}
\end{algorithm}

The RKHS setting avoids large-data limits and allows direct pointwise error bounds. Let $\mathfrak{K}$ be the used kernel and $\{x^{(n)},y^{(n)}\}_{n=1}^N$ be the snapshot data. We then define the matrices $\mathbf{G}$, $\mathbf{A}$, $\mathbf{K}$ and $\mathbf{R}$ as in \cref{eqn:GA_ker_defn,eqn:K_ker_defn,eqn:L_ker_defn}. We want to bound the error in estimating $g(x_n)=\koop^ng(x_0)$. Recalling from \cref{eqn:pfmd_predictions} that $\koop^ng(x_0)=\langle g,(\koop^*)^n\mathfrak{K}_{x_0}\rangle$, we need to bound the error in the PFMD for $\mathfrak{K}_{x_0}$. First, we approximate $\mathfrak{K}_{x_0}$ in the subspace $\mathcal{V}_N=\mathrm{span}\{\mathfrak{K}_{x^{(1)}},\dots,\mathfrak{K}_{x^{(N)}}\}$, i.e., we find a vector of coefficients $\mathbf{c}\in\mathbb{C}^N$ that minimizes
\begin{equation}\setlength\abovedisplayskip{2pt}\setlength\belowdisplayskip{2pt}
\label{min_projection}
    \Big\|\mathfrak{K}_{x_0}-\sum_{n=1}^Nc_n\mathfrak{K}_{x^{(n)}}\Big\|^2.
\end{equation}
Letting $\mathbf{b}=\begin{bmatrix}
    \mathfrak{K}(x_0,x^{(1)})&\cdots &\mathfrak{K}(x_0,x^{(N)})
\end{bmatrix}^\top$, this is equivalent to minimizing $\mathbf{c}^*\mathbf{G}\mathbf{c}-\mathbf{b}^*\mathbf{c}-\mathbf{c}^*\mathbf{b}$, which has solution
$\mathbf{c}=\mathbf{G}^{-1}\mathbf{b}$. We let the residual error be $\delta\coloneqq\min_{\mathbf{c}'}\|\mathfrak{K}_{x_0}-\sum_{n=1}^Nc_n'\mathfrak{K}_{x^{(n)}}\|=\sqrt{\mathfrak{K}(x_0,x_0)-\mathbf{c}^*\mathbf{b}}$, and let $c=\sum_{n=1}^Nc_n\mathfrak{K}_{x^{(n)}}\in\mathcal{V}_N$. Our approximation to $g(x_n)$ is $
\Gamma(g,x_0,n)=\langle g,\mathcal{P}_{\mathcal{V}_N}^*(\mathcal{P}_{\mathcal{V}_N}\koop^*\mathcal{P}_{\mathcal{V}_N}^*)^nc\rangle$, with error bounded by
\begin{equation}\setlength\abovedisplayskip{6pt}\setlength\belowdisplayskip{6pt}\label{eqn:pfmd_pointwise_bound}
    \begin{split}
        &|g(x_n)-\Gamma(g,x_0,n)|=\left|\langle g,(\koop^*)^n\mathfrak{K}_{x_0}\rangle_{\mathfrak{K}}-\left\langle g,\mathcal{P}_{\mathcal{V}_N}^*(\mathcal{P}_{\mathcal{V}_N}\koop^*\mathcal{P}_{\mathcal{V}_N}^*)^nc\right\rangle\right|\\
        &\;\leq \|g\|\left\|(\koop^*)^n\mathfrak{K}_{x_0}-\mathcal{P}_{\mathcal{V}_N}^*(\mathcal{P}_{\mathcal{V}_N}\koop^*\mathcal{P}_{\mathcal{V}_N}^*)^nc
        \right\|\\
        &\;\leq \|g\|\left(\left\|(\koop^*)^n\mathfrak{K}_{x_0}-(\koop^*)^n\mathcal{P}_{\mathcal{V}_N}^*c\right\|+\left\|(\koop^*)^n\mathcal{P}_{\mathcal{V}_N}^*c-\mathcal{P}_{\mathcal{V}_N}^*(\mathcal{P}_{\mathcal{V}_N}\koop^*\mathcal{P}_{\mathcal{V}_N}^*)^n c
        \right\|\right)\\
        &\;\leq \|g\|\|(\koop^*)^n\|\delta+\|g\|\left\|(\koop^*)^n\mathcal{P}_{\mathcal{V}_N}^*c-\mathcal{P}_{\mathcal{V}_N}^*(\mathcal{P}_{\mathcal{V}_N}\koop^*\mathcal{P}_{\mathcal{V}_N}^*)^nc
        \right\|.
    \end{split}
\end{equation}
We can approximate $\|(\koop^*)^n\|\approx \|\mathbf{G}^{1/2}\mathbf{K}^n\mathbf{G}^{-1/2}\|$, and $\|g\|$ can be computed analytically, approximated by finding a least squares approximation to $g$ in $\mathcal{V}_N$ or dealt with using the techniques of \cref{sec:ee_gp_comp}. To bound the projection error, note that \cref{kmd_norm_bound_thm} generalizes immediately to the Perron--Frobenius operator, i.e.,
\[\setlength\abovedisplayskip{3pt}\setlength\belowdisplayskip{3pt}
\left\|(\mathcal{K}^*)^n\mathcal{P}_{\mathcal{V}_N}^*c-\mathcal{P}_{\mathcal{V}_N}^*(\mathcal{P}_{\mathcal{V}_N}\mathcal{K}^*\mathcal{P}_{\mathcal{V}_N}^*)^{n}c\right\|\leq\sum_{j=0}^{n-1}  \left\|(\mathcal{P}_{\mathcal{V}_N}^*\mathcal{P}_{\mathcal{V}_N}\mathcal{K}^*)^j\right\|\left\|\mathcal{Q}_{\mathcal{V}_N}(\mathcal{K}^*)^{n-j}\mathcal{P}_{\mathcal{V}_N}^*c\right\|,
\]
for all $n\in\mathbb{N}$. As before, a first-order approximation is
\[\setlength\abovedisplayskip{6pt}\setlength\belowdisplayskip{3pt}
\|\mathcal{Q}_{\mathcal{V}_N}(\koop^*)^{n-j}\mathcal{P}_{\mathcal{V}_N}^*c\|^2\approx (\mathbf{K}^{n-j-1}\mathbf{c})^*(\mathbf{R}-\mathbf{K}^*\mathbf{G}\mathbf{K})(\mathbf{K}^{n-j-1}\mathbf{c}).
\]
These arguments yield the first-order RKHS-norm error bound 
\begin{equation}\setlength\abovedisplayskip{2pt}\setlength\belowdisplayskip{2pt}\label{eqn:init_proj_error}
\begin{split}
\|(\koop^*)^n\mathfrak{K}_{x_0}-\mathcal{P}_{\mathcal{V}_N}^*(\mathcal{P}_{\mathcal{V}_N}\koop^*\mathcal{P}_{\mathcal{V}_N}^*)^nc\|\leq E_n\coloneqq\underbrace{\|(\koop^*)^n\|\delta}_{\text{initialization error}}\hspace{2.5cm}\\
+\underbrace{\sum_{j=0}^{n-1}  
\|\mathbf{G}^{1/2}\mathbf{K}^j\mathbf{G}^{-1/2}\|
\sqrt{(\mathbf{K}^{n-j-1}\mathbf{c})^*(\mathbf{R}-\mathbf{K}^*\mathbf{G}\mathbf{K})(\mathbf{K}^{n-j-1}\mathbf{c})}}_{\text{projection error}},
\end{split}
\end{equation}
and the corresponding pointwise error bound 
\begin{equation}\setlength\abovedisplayskip{6pt}\setlength\belowdisplayskip{6pt}\label{rkhs_error_bound}
        |g(x_n)-\Gamma(g,x_0,n)|\lesssim \|g\|E_n.
\end{equation}
The algorithm is summarized in \cref{alg:specrkhs_pfmd}. As before, the projection error is exact at $n=1$, but the initialization error is only exact for $n=0$.

\subsection{Numerical examples}\label{sec:erb_numex}

This section presents numerical examples illustrating our error bounds.

\subsubsection{Optimality of the bounds in general}

We first provide an example of a system for which our error bounds, and their first-order numerical approximation, are sharp; hence, without further assumptions one cannot improve upon our results. Let $\mathcal{X}$, $F$ and $\omega$ be such that $\koop$ has (countably infinite) Lebesgue spectrum, i.e., there exists an orthonormal basis of $L^2(\mathcal{X},\omega)$ $
\{1,g_{i,j}:i\in\mathbb{N},j\in\mathbb{Z}\}$ with $\koop g_{i,j}=g_{i,j-1}$ for $i\in\mathbb{N}$, $j\in\mathbb{Z}$. Examples of such systems are a two-sided Bernoulli shift or Arnold's famous catmap \cite{arnold1968ergodic}. Let $\mathcal{V}_N\subset L^2(\mathcal{X},\omega)$ be a sequence of increasing subspaces such that for all $N$, $\{1,g_{1,-N},\dots,g_{1,N}\}\subset\mathcal{V}_N$ but $\{g_{1,-(N+1)},g_{1,(N+1)}\}\notin \mathcal{V}_N$. For all $N,k\in\mathbb{N}$, there exists $g\in\mathcal{V}_N$ such that \cref{bound_lemma,bound_lemma2,KMD_bound_first_order,KMD_bound_first_order_2} are sharp at $n=k$; in particular, we take $g=g_{1,-(N+1)+k}$. In fact, the full-order bound is exact for all $n\leq k$ and the first-order bound is exact at all times. 

The relevant errors and bounds are shown in \cref{fig:cat_map} for $k=4$ and $N=100$. The system is such that once an observable leaves $\mathcal{V}_N$ it never returns; this explains why the exact and first-order errors remain constant beyond $k=4$. By contrast, the full-order term keeps accumulating errors as more and more terms leave the subspace.

\begin{figure}[t]
    \centering
    \includegraphics[height=0.3\linewidth]{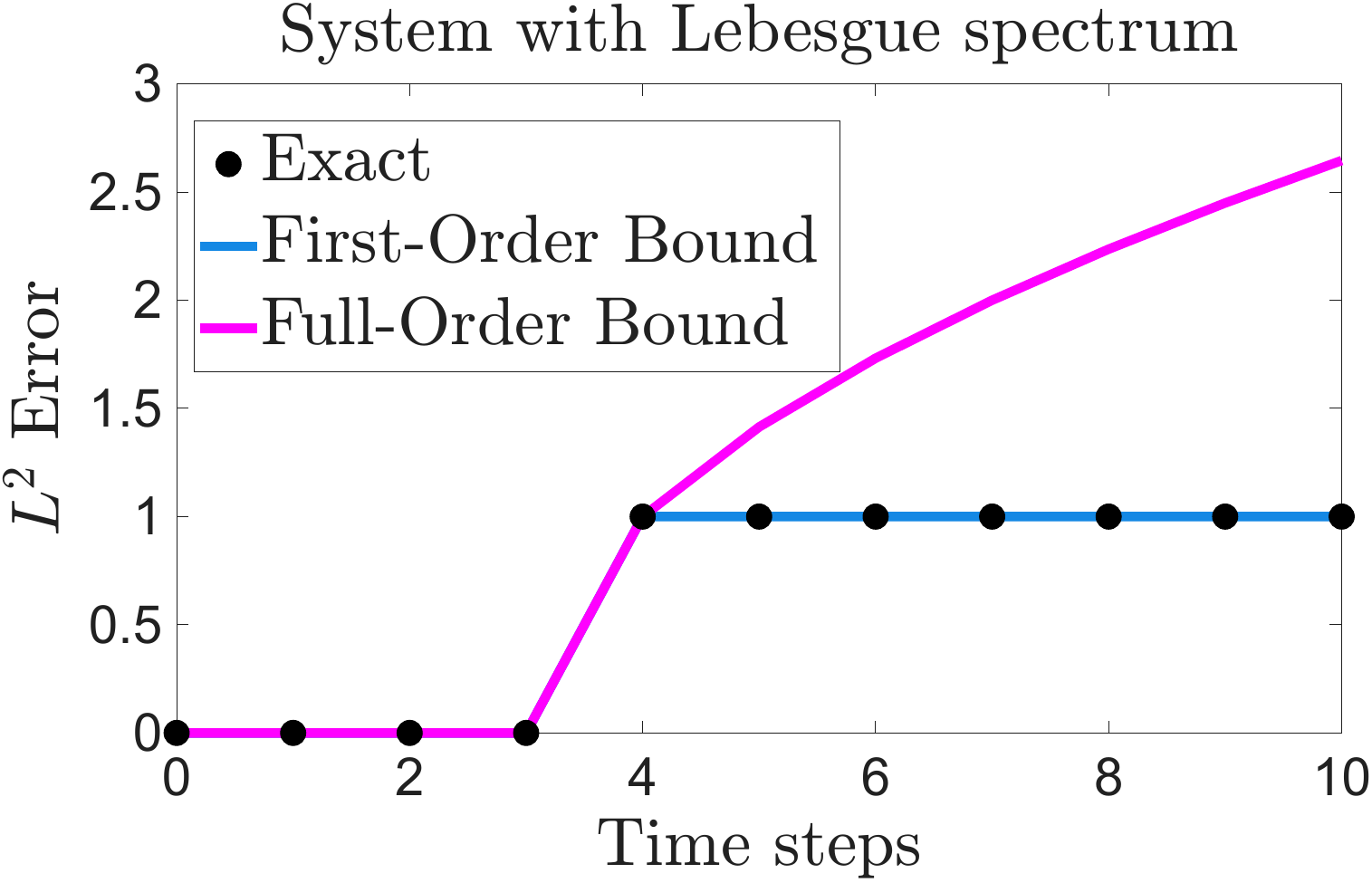}
        \caption{Sharpness of bounds. The exact errors and error bound formula to first-order (minimum of \cref{KMD_bound_first_order,KMD_bound_first_order_2}) and full-order (minimum of \cref{bound_lemma,bound_lemma2}) for a system with Lebesgue spectrum.}
    \label{fig:cat_map}
\end{figure}

\subsubsection{Duffing and Lorenz systems}
We now return to the Duffing and Lorenz systems defined in \cref{duffing_eqn,lorenz_defn}. We first consider the $L^2$ setting, using a dictionary of $N=100$ Chebyshev polynomials and $N=500$ exponential RBFs respectively, and apply \cref{alg:resDMD_KMD}. In \cref{sec:pad}, we showed that using a PAD instead of an SVD improves prediction accuracy; here we demonstrate that the error bounds also improve correspondingly. For each system, we take $r=N/2$ and consider the state observable $g(x)=x_1$. The exact errors and error bounds are shown in \cref{fig:erb_duffing_lorenz} for both the SVD- and PAD-truncated systems. The substantial reduction in exact errors achieved by the PAD at short times is mirrored by the error bounds. The bounds are not always strictly above the exact error, reflecting the use of first-order approximations and of quadrature error. For the Lorenz oscillator the bounds are essentially tight in both cases; errors in unitary Koopman operators are easier to predict.

\begin{figure}[t]
    \centering
        \includegraphics[height=0.3\linewidth]{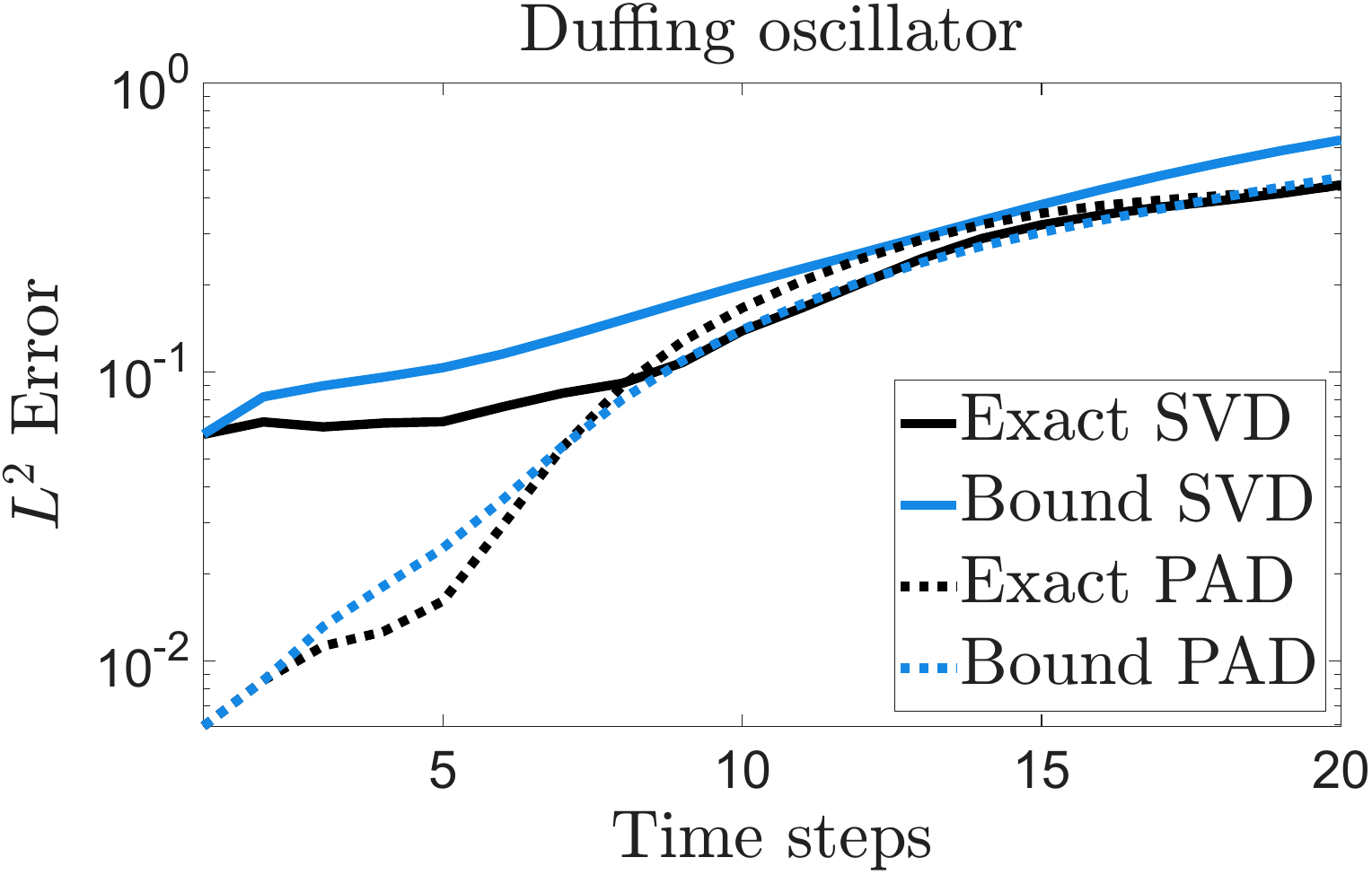}\hfill
    \includegraphics[height=0.3\linewidth]{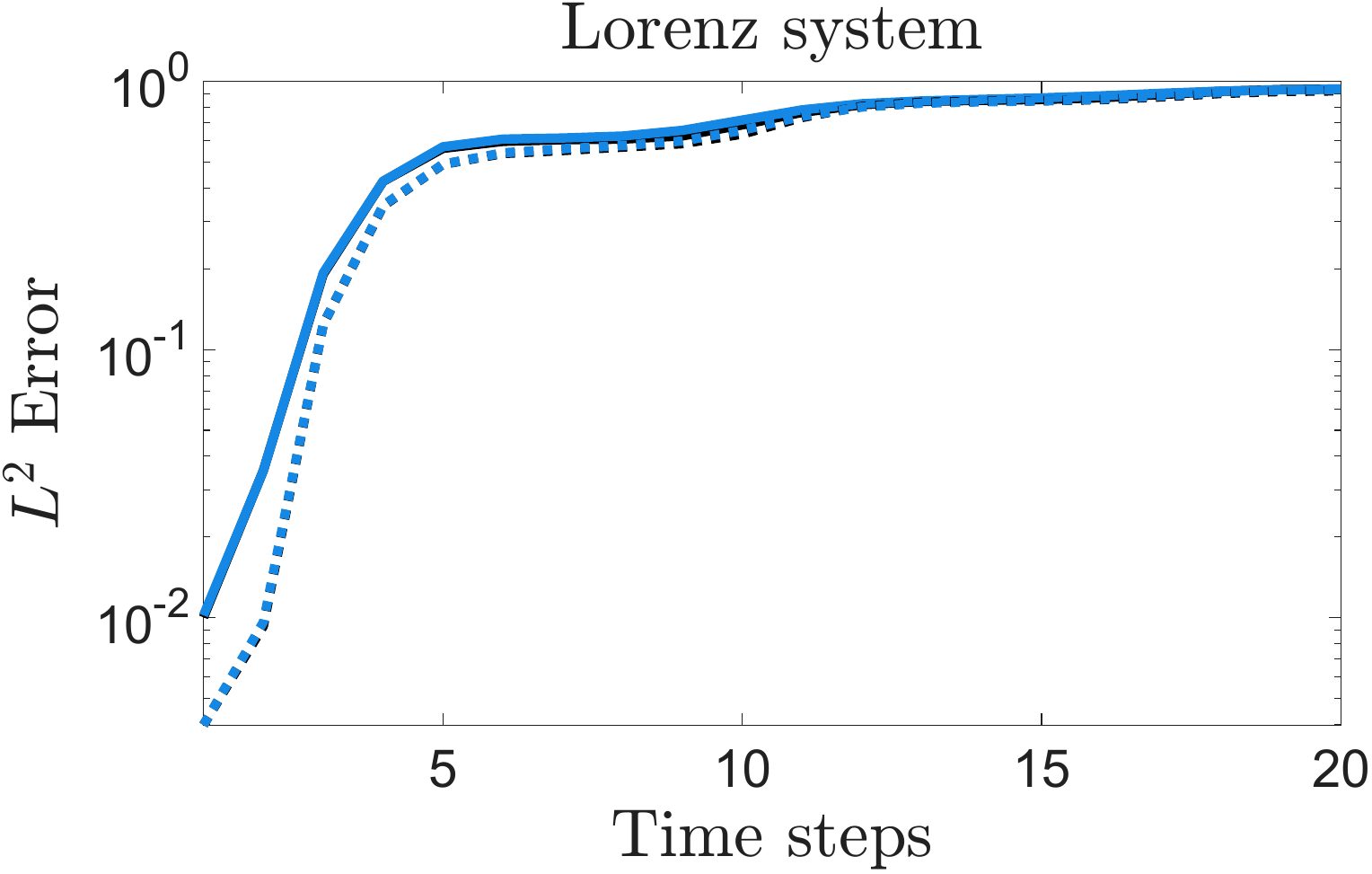}
    \caption{Exact $L^2$ errors and error bounds for the Duffing oscillator (left) and Lorenz system (right) using a truncated SVD and PAD with the same number of modes. The PAD method outperforms the SVD for small numbers of time steps, in terms of both exact errors and error bounds.}
    \label{fig:erb_duffing_lorenz}
\end{figure}

In the RKHS setting, we first apply \cref{alg:specrkhs_pfmd} to the Duffing oscillator with $g=\mathfrak{K}_{x_0}$, where $x_0\in[-2,2]^2$ is sampled uniformly at random, using the Mat\'ern kernel
\[\setlength\abovedisplayskip{6pt}\setlength\belowdisplayskip{6pt}
\mathfrak{K}(x,y)=\begin{cases}
    2,&\text{if }x=y,\\
    (\sigma\|x-y\|)^2K_2(\sigma\|x-y\|),&\text{otherwise},
\end{cases}
\]
where $K_2$ denotes the modified Bessel function of the second kind of order $2$, and $\sigma$ is chosen to ensure that the Gram matrix $\mathbf{G}$ is well conditioned; in particular, we take $\sigma$ to be inversely proportional to the averaged standard deviation of each component of the data. This kernel induces an RKHS equivalent to the Sobolev space $H^3(\mathbb{R}^2)$~\cite{kohne_linfty-error_2024,wendland_scattered_2004}. We use the same snapshot data as the $L^2$ case, and for order reduction, we use an SVD with $r=N/5$; since we are using a delay embedding method, most principal angles are $0$, so PAD is not as informative here. In \cref{fig:erb_svd_pad_l2}, we display the exact RKHS error and upper bound \cref{eqn:init_proj_error} for the Duffing oscillator and Lorenz system. Here, the initialization error dominates at short times but is eventually overtaken by the projection error. Both errors vanish as $N\to\infty$. Similar behavior is observed for the Lorenz system, again using a Mat\'ern kernel. Compared to the $L^2$ case, the error bounds are less tight; this shall be addressed in the following sections.

\begin{figure}[t]
    \centering
        \includegraphics[height=0.3\linewidth]{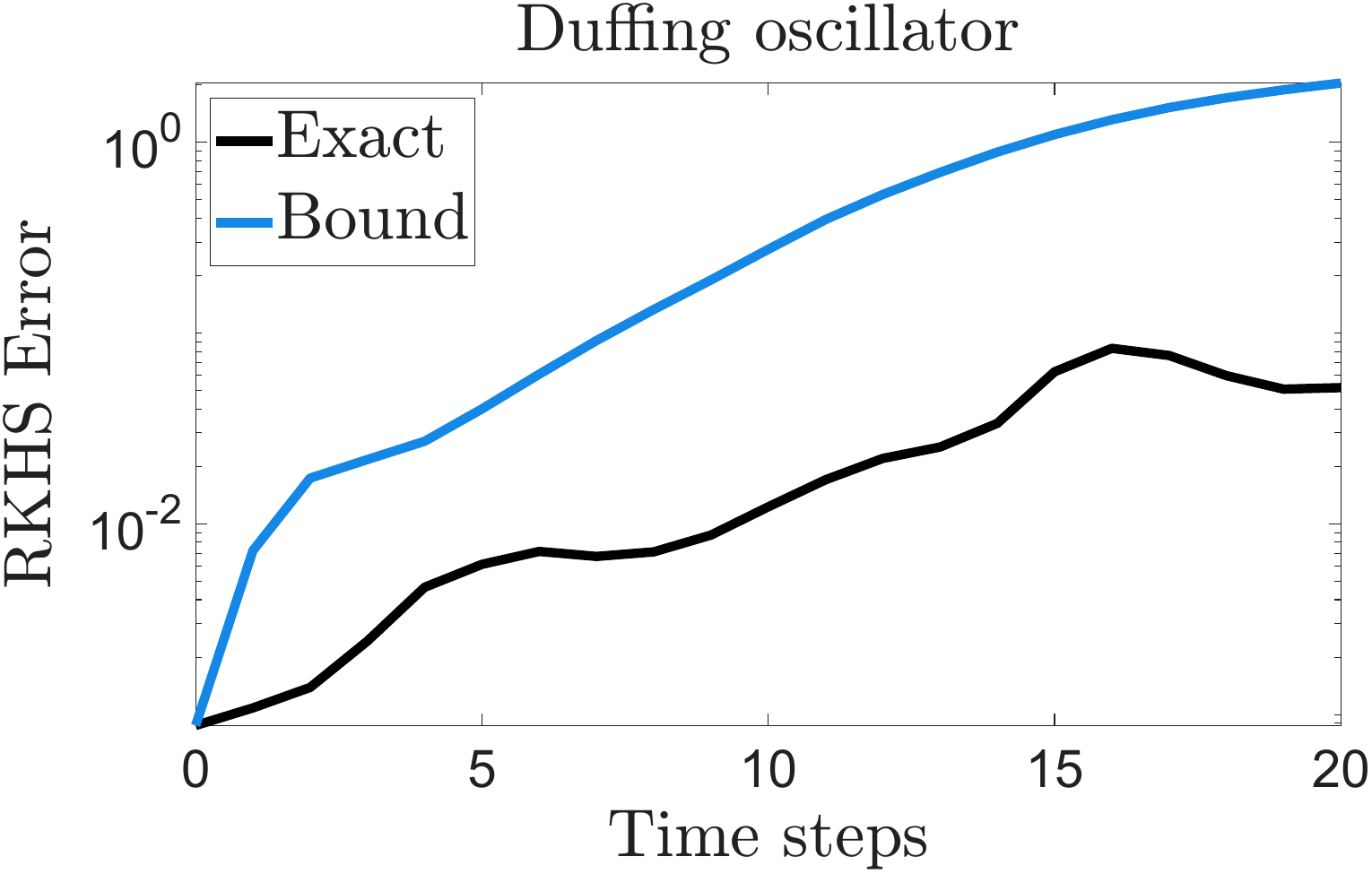}
        \hfill\includegraphics[height=0.3\linewidth]{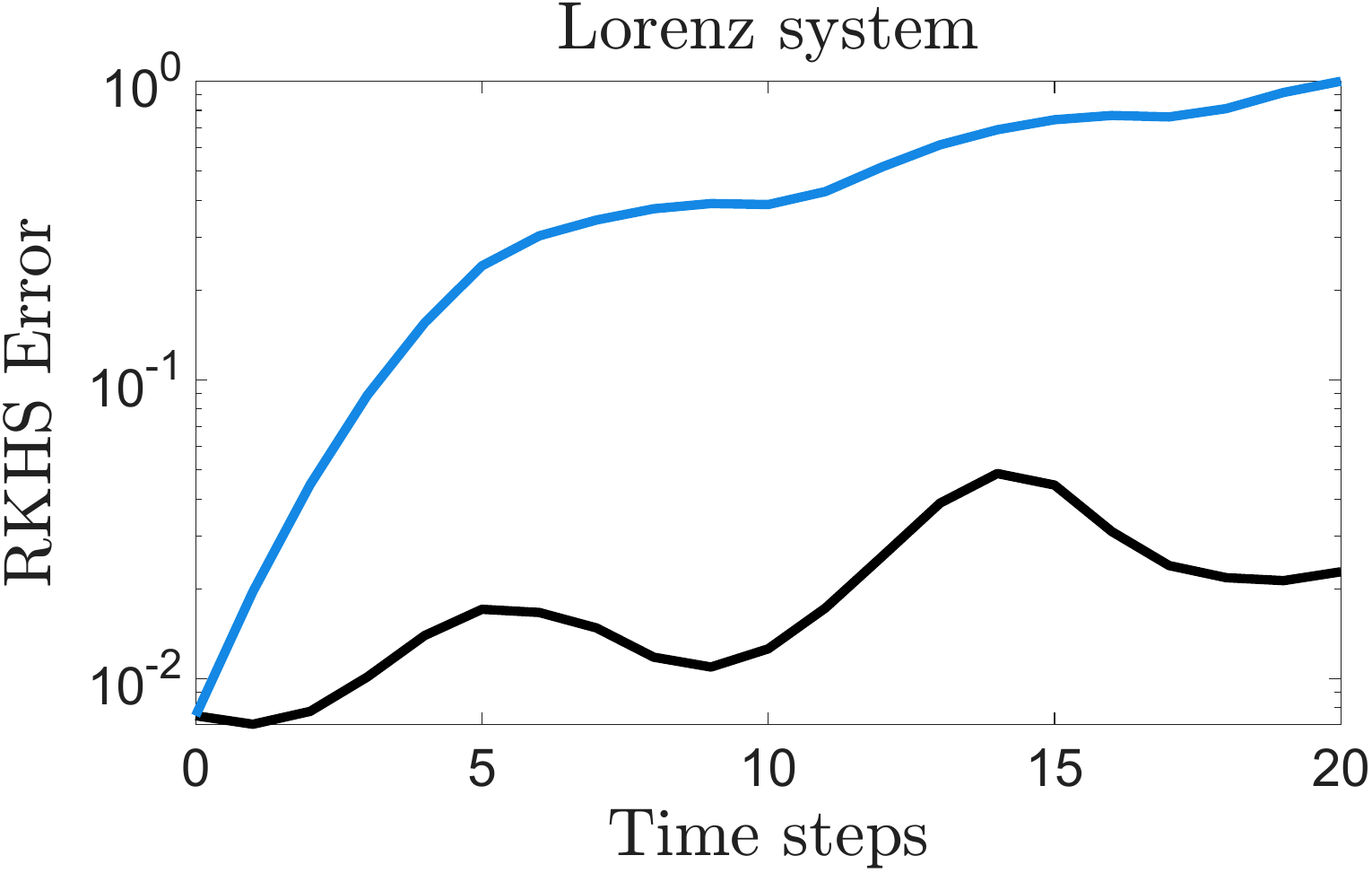}
        \caption{Exact RKHS errors and error bounds \cref{eqn:init_proj_error} for the Duffing oscillator (left) and Lorenz system (right) using a Mat\'ern kernel.}
    \label{fig:erb_svd_pad_l2}
\end{figure}

\begin{figure}[t]
    \centering
    \raisebox{-0.5\height}{\includegraphics[height=0.24\linewidth]{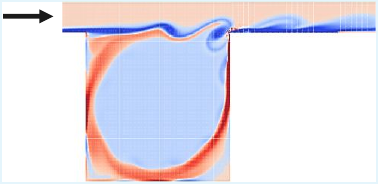}}\hfill
    \raisebox{-0.5\height}{\includegraphics[height=0.3\linewidth]{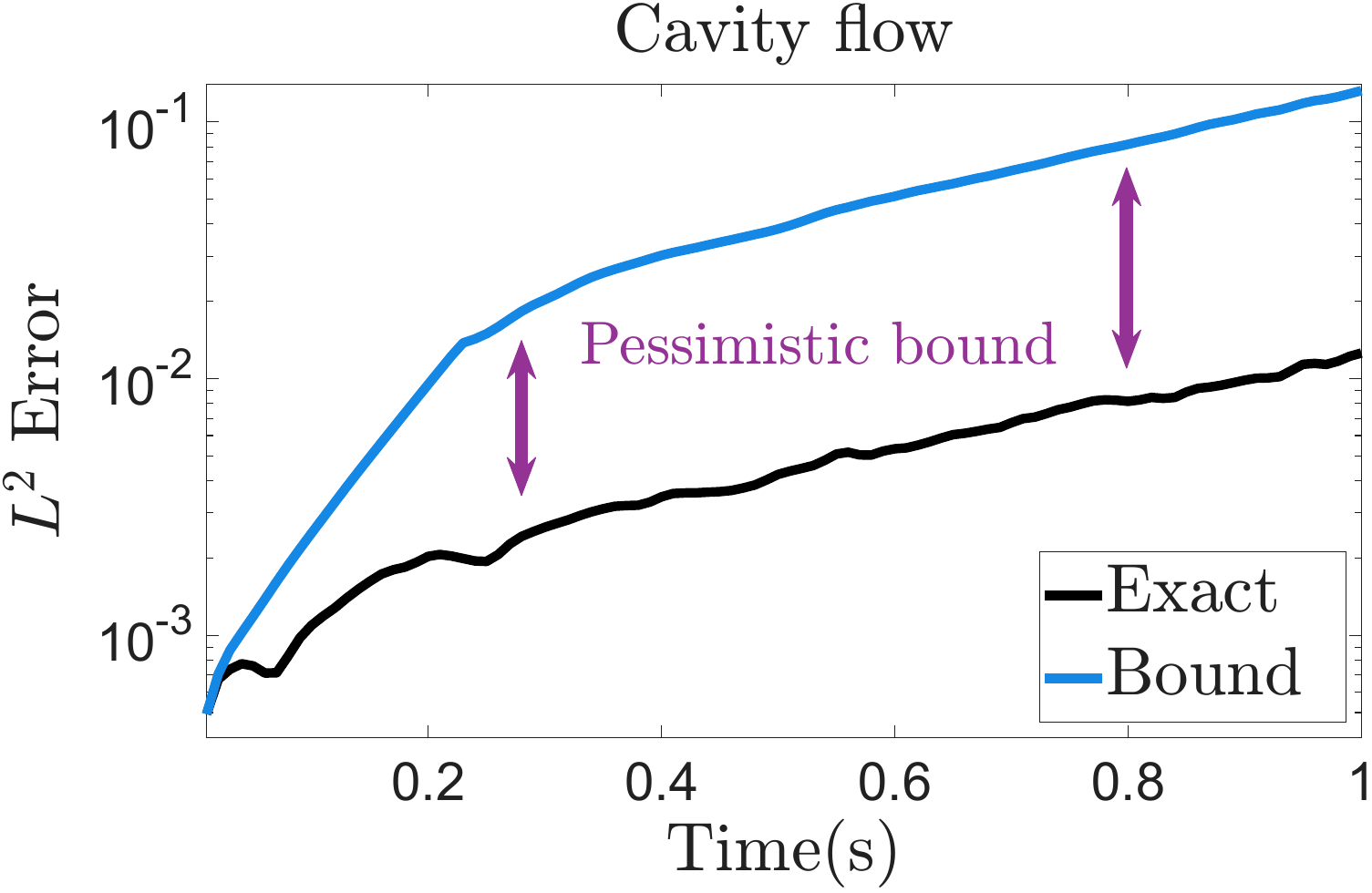}}
    \caption{Left: A snapshot of the cavity flow system. Right: Exact $L^2$-errors and error bounds for the cavity flow system. The error bounds overestimate the true error by a factor of over $10$.}
    \label{fig:erb_cavity}
\end{figure}

\subsubsection{Cavity flow---an example where the bounds are pessimistic} As a final example, we consider applying error bounds to a more complicated, high-dimensional system involving cavity flow data at Reynolds number 7500 with $144000$ degrees of freedom. The snapshot data are drawn from a single trajectory. We use the first $200$ pairs of snapshot data to construct a DMD-based basis, and the next $1000$ pairs to compute the ResDMD matrices. We apply \cref{alg:resDMD_KMD} to compute error bounds on the $L^2$-norm errors of the prediction of a component of the full-state observable, and compare to the true $L^2$-errors. The results are shown in \cref{fig:erb_cavity}. For roughly $t\leq 0.25$, the bound given by \cref{KMD_bound_first_order} is smaller than that given by \cref{KMD_bound_first_order_2}, while the opposite is true for $t\geq 0.25$. We see that while the error bounds do accurately bound the true errors, they are pessimistic. In the next section, we will explore how to overcome this challenge and provide a more accurate estimate of the error.

\section{Expected errors}
\label{sec:expected_errors}

The inequalities used to derive the error bounds above can introduce a substantial gap between the bound and the true error. Motivated by the cavity flow example in \cref{fig:erb_cavity}, we introduce randomized, Gaussian--based methods that yield expected error estimates, and provide a closer approximation to the true error.

\subsection{Gaussian averaging}
\label{sec:gaussian_process_intro}

We consider Gaussian random elements, the Hilbert-space analogue of multivariate Gaussian distributions. Let $\{\phi_j\}_{j=1}^{\infty}$ be an orthonormal basis of $\mathcal{H}$ and let $\{\lambda_j\}_{j=1}^{\infty}\subset(0,\infty)$ satisfy $\smash{\sum_{j=1}^{\infty}\lambda_j<\infty}$. If $\{a_j\}_{j=1}^{\infty}$ are independent standard normal variables, then $\zeta=\sum_{j=1}^{\infty}\sqrt{\lambda_j}a_j\phi_j$ converges almost surely in $\mathcal{H}$ and defines a centered Gaussian random element. This representation is the Karhunen--Lo\`eve expansion \cite[Thm.~6.19]{stuart2010inverse}. The associated covariance operator $\mathscr{C}$ has eigenvectors $\phi_j$ with corresponding eigenvalues $\lambda_j$, and we write $\zeta\sim\mathcal{GP}(0,\mathscr{C})$.

Using Gaussian random elements, we aim to replace the strict inequalities used in the construction of \cref{bound_lemma} and \cref{bound_lemma2} with expected quantities. In the derivations, we used submultiplicativity of the operator norm $\|Ax\|\leq \|A\|\|x\|$. This represents a worst-case scenario. To obtain tighter expected errors, we instead look for ``average-case'' scenarios. Given a bounded operator $A:\mathcal{H}\rightarrow\mathcal{H}$ we define an expected operator norm
\[\setlength\abovedisplayskip{0pt}\setlength\belowdisplayskip{2pt}
\mathbb{E}_{\mathscr{C}}[A] \coloneqq \mathbb{E}\left[{\|A\zeta\|}/{\|\zeta\|}\right],\qquad\zeta\sim \mathcal{GP}(0,\mathscr{C})
\]
and take $\|Ax\|\approx \mathbb{E}_{\mathscr{C}}[A]\|x\|$ in our error bounds. Similar expectations over a probability measure on a function space are used in the machine learning literature~\cite{lanthaler2022error}. 

\subsection{Computation of expected errors}
\label{sec:ee_gp_comp}

We first consider approximations of $\mathbb{E}_{\mathscr{C}}[\koop^j]$ for a Koopman operator $\koop$ acting on $\mathcal{H}$. In the RKHS setting, $\koop$ is replaced by $\koop^*$. Let $\zeta\sim\mathcal{GP}(0,\mathscr{C})$. Since $\|(\mathcal{P}_{\mathcal{V}_N}\koop\mathcal{P}_{\mathcal{V}_N}^*)^j\|\leq\|\koop\|^j<\infty$, by the dominated convergence theorem and convergence of (k)EDMD in the strong operator topology,
\begin{equation}\setlength\abovedisplayskip{6pt}\setlength\belowdisplayskip{6pt}\label{dct_expected_error}
\lim_{N\rightarrow\infty}\mathbb{E}\left[{\|(\mathcal{P}_{\mathcal{V}_N}\koop\mathcal{P}_{\mathcal{V}_N}^*)^j\mathcal{P}_{\mathcal{V}_N}\zeta\|}\big/{\|\mathcal{P}_{\mathcal{V}_N}\zeta\|}\right]=\mathbb{E}_{\mathscr{C}}[\koop^j].
\end{equation}
We can compute the term inside the expectation using only the (k)EDMD matrices $\mathbf{G}$ and $\mathbf{K}$ and pointwise samples of $\zeta$, which follow a multivariate Gaussian distribution. Writing $\mathcal{P}_{\mathcal{V}_N}\zeta=\sum_{i=1}^Nw_i\psi_i$, $\mathbf{w}=(w_1\,\cdots\,w_N)^\top$ is the solution to:
\[\setlength\abovedisplayskip{6pt}\setlength\belowdisplayskip{6pt}
\mathbf{w}=\argmin_{\mathbf{w}'\in\mathbb{C}^N}\Big\|\zeta-\sum_{i=1}^Nw_i'\psi_i\Big\|=\mathbf{G}^{-1}\mathbf{d},\quad \mathbf{d}=(\langle\zeta,\psi_1\rangle\,\cdots\,\langle\zeta,\psi_N\rangle
)^\top.
\]
In the $L^2$ case, these inner products can be computed via a quadrature approximation at the snapshot data points, while in the RKHS case 
$\mathbf{d}=(\zeta(x^{(1)})\,\cdots\,\zeta(x^{(N)}))^\top$. Hence in both cases, $\mathbf{d}$ and $\mathbf{w}$ can be constructed using only pointwise samples of $\zeta$.

The norm in \cref{dct_expected_error} is then given by
\begin{equation}\setlength\abovedisplayskip{6pt}\setlength\belowdisplayskip{6pt}
{\left\|(\mathcal{P}_{\mathcal{V}_N}\koop\mathcal{P}_{\mathcal{V}_N}^*)^j\mathcal{P}_{\mathcal{V}_N}\zeta\right\|}\big/{\|\mathcal{P}_{\mathcal{V}_N}\zeta\|}=\lim_{M\rightarrow\infty}{\|\mathbf{G}^{1/2}\mathbf{K}^j\mathbf{w}\|}\big/{\|\mathbf{G}^{1/2}\mathbf{w}\|},
\end{equation}
where no large-data limit is necessary in the RKHS case. To compute the expectation, we use a Monte Carlo approach. Letting $\{\zeta_k\}_{k=1}^P$ be samples of the chosen Gaussian process and $\mathbf{w}_k$ representations of $\mathcal{P}_{\mathcal{V}_N}\zeta_k$ in the chosen basis, we have that
\begin{equation}\label{eqn:exp_op_norm}\setlength\abovedisplayskip{6pt}\setlength\belowdisplayskip{6pt}
\mathbb{E}_{\mathscr{C}}[\koop^j]\approx\cfrac{1}{P}\sum_{k=1}^P{\|\mathbf{G}^{1/2}\mathbf{K}^j\mathbf{w}_k\|}\big/{\|\mathbf{G}^{1/2}\mathbf{w}_k\|}.
\end{equation}
In particular, almost surely by the strong law of large numbers,
\[\setlength\abovedisplayskip{6pt}\setlength\belowdisplayskip{6pt}
\mathbb{E}_{\mathscr{C}}[\koop^j]=\lim_{N\rightarrow\infty}\lim_{P\rightarrow\infty}\lim_{M\rightarrow\infty}\cfrac{1}{P}\sum_{k=1}^P{\|\mathbf{G}^{1/2}\mathbf{K}^j\mathbf{w}_k\|}\big/{\|\mathbf{G}^{1/2}\mathbf{w}_k\|}.
\]

In practice, we construct a data-driven covariance operator from the (k)EDMD basis. We obtain an orthonormal dictionary from a truncated SVD or PAD of the Gram matrix $\mathbf{G}$. This extracts the dominant dynamical modes (those contributing most to the expected operator norm) while also improving computational efficiency. Next, we choose eigenvalues $\lambda_1\geq \lambda_2\geq \cdots \geq 0$ with $\sum_{j=1}^\infty \lambda_j<\infty$. Together with the SVD/PAD basis, these define a covariance operator $\mathscr{C}$ and a Gaussian process $\zeta\sim\mathcal{GP}(0,\mathscr{C})$ as in \cref{sec:gaussian_process_intro}. Then with respect to the chosen orthonormal basis, $\mathcal{P}_{\mathcal{V}_N}\zeta$ is 
$
\mathbf{w}=(
    \sqrt{\lambda_1}w_1\,\cdots \, \sqrt{\lambda_N}w_N
)^\top,
$
where the coefficients $w_i$ are i.i.d.\ $\mathcal{N}(0,1)$. The expectation $\mathbb{E}_{\mathscr{C}}[\koop^j]$ is then computed exactly as before.

\subsection{Further expectations}

So far, we have addressed the gap between true errors and their bounds arising from inequalities involving the operator norm, such as $\|Ax\|\leq \|A\|\|x\|$. In \cref{bound_lemma} this was the only source of overestimation, allowing a straightforward alteration of \cref{bound_lemma} replacing $\|\mathbf{G}^{1/2}\mathbf{K}\mathbf{G}^{-1/2}\|$ by $\mathbb{E}_{\mathscr{C}}[\koop]$. However, the other error bound formulae suffer from other sources of overestimation. In deriving \cref{bound_lemma2} we also invoked the triangle inequality, while \cref{eqn:pfmd_pointwise_bound} relies on the Cauchy--Schwarz inequality. Although essential for obtaining general upper bounds (as we saw in \cref{sec:decomp_kmd}), these inequalities often lead to conservative estimates (see \cref{fig:erb_cavity}). We therefore extend the expectation-based arguments of the previous section to address these additional sources of overestimation, continuing to use the data-driven covariance operator constructed above.

\begin{algorithm}[t]
\textbf{Input:} Galerkin matrices $\mathbf{G}$, $\mathbf{A}$ and $\mathbf{L}$ for dictionary $\{\psi_n\}_{n=1}^{N}$, a sequence $\lambda_1\geq \lambda_2\geq \dots \geq 0$ with $\sum_{j=1}^{\infty}\lambda_j<\infty$, sample size $P\in\mathbb{N}$, observable $g\in\mathcal{V}_N$, number of timesteps $n\in\mathbb{N}$.\\
\begin{algorithmic}[1]
\STATE Using a truncated SVD or PAD (see \cref{alg:pad}), construct an orthonormal basis and compress $\mathbf{G}$, $\mathbf{A}$, $\mathbf{L}$ and $\mathbf{K}$ to act on that basis (so $\mathbf{G}=\mathbf{I}$, $\mathbf{A}=\mathbf{K}$).
\STATE For each $j,k$, let
$\mathbf{w}_{j,k}\,{=}\,(\sqrt{\lambda_1}w_{j,k,1}\,\cdots \, \sqrt{\lambda_N}w_{j,k,N})^\top$ for i.i.d. $w_{j,k,i}\sim\mathcal{N}(0,1)$.
\STATE Compute the approximations to $\mathbb{E}_{\mathscr{C}}[\koop^j]$ given by \cref{eqn:exp_op_norm} for $j=0,\dots,n-1$.
\STATE Define $\alpha_j=\mathbb{E}_{\mathscr{C}}[\koop^j]\sqrt{(\mathbf{K}^{n-j-1}\mathbf{g})^*(\mathbf{L}-\mathbf{K}^*\mathbf{K})(\mathbf{K}^{n-j-1}\mathbf{g})}$ for $j=0,\dots,n-1$.
\STATE Compute $E_n\approx \frac{1}{P}\sum_{k=1}^P\big[\big\|\sum_{j=0}^{n-1}\alpha_j\mathbf{w}_{j,k}/\|\mathbf{w}_{j,k}\|\big\|\big]$.
\end{algorithmic} \textbf{Output:} $E_n$, the expected $L^2$ error of $\|\koop^n g-{\Psi}\mathbf{K}^n\mathbf{g}\|$.
\caption{Computation of expected $L^2$ errors for the KMD.}\label{alg:expected_error_all_l2}
\end{algorithm}

\begin{algorithm}[t]
\textbf{Input:} Galerkin matrices $\mathbf{G}$, $\mathbf{A}$ and $\mathbf{R}$ for dictionary $\{\mathfrak{K}_{x^{(n)}}\}_{n=1}^{N}$, a sequence $\lambda_1\geq \lambda_2\geq \dots \geq 0$ with $\sum_{j=1}^{\infty}\lambda_j<\infty$, sample size $P\in\mathbb{N}$, starting point $x_0$, observable $g\in\mathcal{H}$, number of timesteps $n\in\mathbb{N}$.\\
\begin{algorithmic}[1]
\STATE Let $\mathbf{b}=(\mathfrak{K}(x_0,x^{(1)})\,\cdots \,\mathfrak{K}(x_0,x^{(N)})
)^\top$, $\mathbf{c}=\mathbf{G}^{-1}\mathbf{b}$ and  
$\delta=\sqrt{\mathfrak{K}(x_0,x_0)-\mathbf{c}^*\mathbf{b}}$.
\STATE Using a truncated SVD or PAD (see \cref{alg:pad}), construct an orthonormal basis and compress $\mathbf{G}$, $\mathbf{A}$, $\mathbf{R}$ and $\mathbf{K}$ to act on that basis (so $\mathbf{G}=\mathbf{I}$, $\mathbf{A}=\mathbf{K}$).
\STATE For each $j,k$, let
$\mathbf{w}_{j,k}\,{=}\,(\sqrt{\lambda_1}w_{j,k,1}\,\cdots \, \sqrt{\lambda_N}w_{j,k,N})^\top$ for i.i.d. $w_{j,k,i}\sim\mathcal{N}(0,1)$.
\STATE Compute the approximations to $\mathbb{E}_{\mathscr{C}}[(\koop^*)^j]$ and $\mathbb{E}_{\mathscr{C}}[g]$ given by \cref{eqn:exp_op_norm,monte_carlo_g}.
\STATE Define $\alpha_j=\mathbb{E}_{\mathscr{C}}[(\koop^*)^j]\sqrt{(\mathbf{K}^{n-j-1}\mathbf{c})^*(\mathbf{R}-\mathbf{K}^*\mathbf{K})(\mathbf{K}^{n-j-1}\mathbf{c})}$ for $j=0,\dots,n-1$ and $\alpha_n=\mathbb{E}_{\mathscr{C}}[(\koop^*)^n]\delta$.
\STATE Compute $E_n\approx \frac{1}{P}\sum_{k=1}^P\big[\big\|\sum_{j=0}^{n}\alpha_j\mathbf{w}_{j,k}/\|\mathbf{w}_{j,k}\|\big\|\big]$.
\end{algorithmic} \textbf{Output:} $E_n$, the expected RKHS-norm error of $\|(\koop^*)^n\mathfrak{K}_{x_0}-\mathcal{P}_{\mathcal{V}_N}^*(\mathcal{P}_{\mathcal{V}_N}\koop^*\mathcal{P}_{\mathcal{V}_N}^*)^nc\|$ and $\mathbb{E}_{\mathscr{C}}[g]E_n$, the expected pointwise error of $|g(x_n)-\Gamma(g,x_0,n)|$.
\caption{Computation of expected RKHS-norm and pointwise errors for the PFMD.}\label{alg:expected_error_all_kernel}
\end{algorithm}

In \cref{bound_lemma2}, we applied the triangle inequality to terms of the form
\[\setlength\abovedisplayskip{6pt}\setlength\belowdisplayskip{6pt}
h_j\coloneqq(\mathcal{P}_{\mathcal{V}_N}^*\mathcal{P}_{\mathcal{V}_N}\koop)^j\mathcal{Q}_{\mathcal{V}_N}^*\mathcal{Q}_{\mathcal{V}_N}\koop^{n-j}\mathcal{P}_{\mathcal{V}_N}^*g,\quad j=0,\dots,n-1.
\]
By using a first-order approximation as in \cref{sec:comp_erbs} and the expected operator norm, we have that
\[\setlength\abovedisplayskip{6pt}\setlength\belowdisplayskip{6pt}
\|h_j\|\approx \mathbb{E}_{\mathscr{C}}[\koop^j]\sqrt{(\mathbf{K}^{n-j-1}\mathbf{g})^*(\mathbf{L}-\mathbf{K}^*\mathbf{G}\mathbf{K})(\mathbf{K}^{n-j-1}\mathbf{g})}\coloneqq \alpha_j.
\]
To obtain an approximation of the error after applying the triangle inequality, we assume that the terms are given by Gaussian processes normalized by the $\alpha_j$, and compute the expectation. In particular, we take the resulting error to be
\[\setlength\abovedisplayskip{6pt}\setlength\belowdisplayskip{6pt}
E_n\approx \mathbb{E}\Big[\Big\|\sum_{j=0}^{n-1}\alpha_j\mathcal{P}_{\mathcal{V}_N}\zeta_j/\|\mathcal{P}_{\mathcal{V}_N}\zeta_j\|\Big\|\Big],\quad\zeta_j\sim\mathcal{GP}(0,\mathscr{C}).
\]
As in \cref{sec:ee_gp_comp}, we approximate this as 
\begin{equation}\setlength\abovedisplayskip{6pt}\setlength\belowdisplayskip{6pt}\label{eqn:exp_tri_eq}
E_n\approx \cfrac{1}{P}\sum_{k=1}^P\Big\|\sum_{j=0}^{n-1}\alpha_j\mathbf{G}^{1/2}\mathbf{w}_{j,k}/\|\mathbf{G}^{1/2}\mathbf{w}_{j,k}\|\Big\|.
\end{equation}
Combining \cref{KMD_bound_first_order_2,eqn:exp_tri_eq,eqn:exp_op_norm} leads to \cref{alg:expected_error_all_l2}.

\begin{figure}[t]
    \centering
    \includegraphics[height=0.3\linewidth]{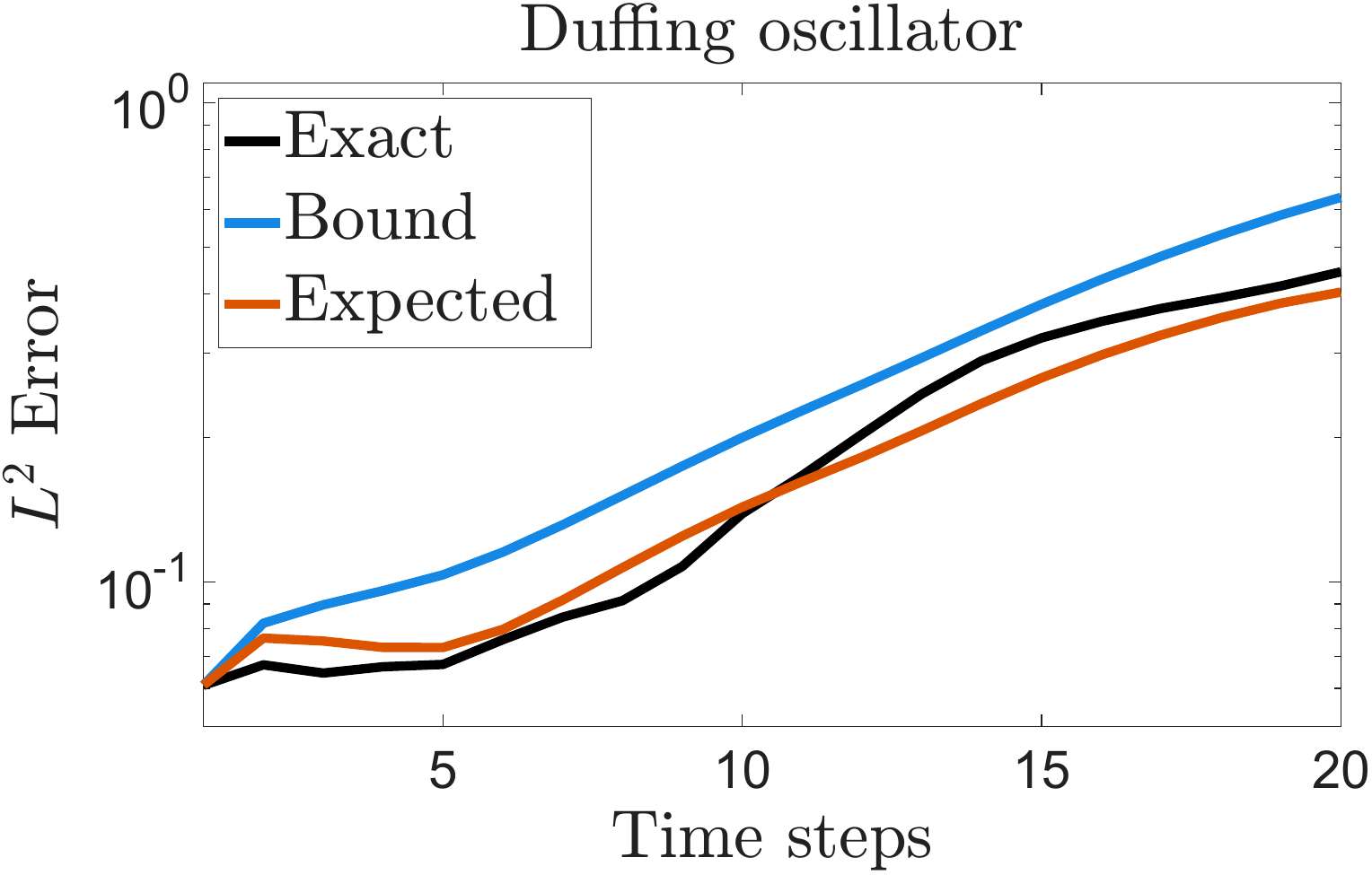}\hfill
    \includegraphics[height=0.3\linewidth]{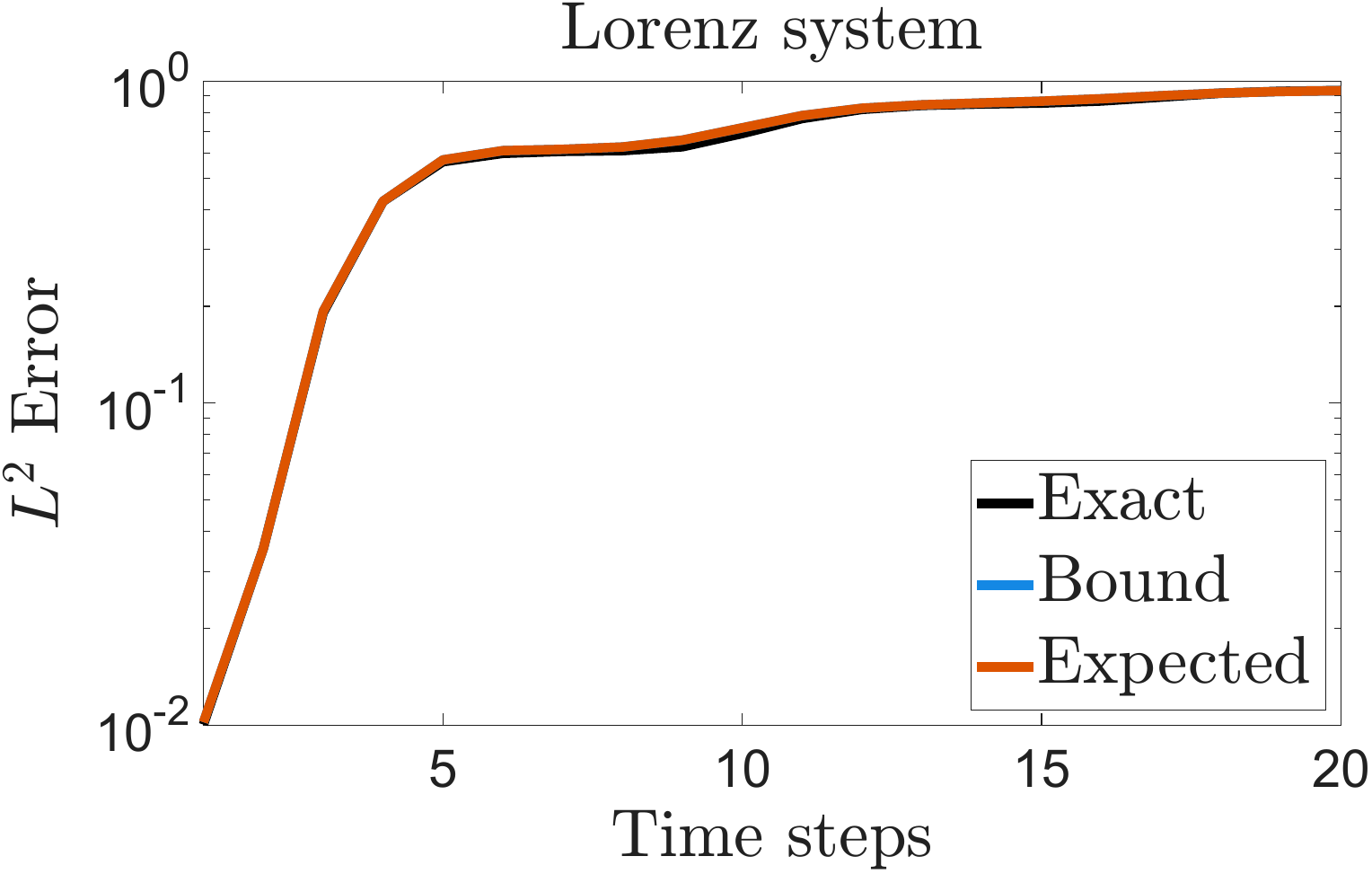}\\\vspace{0.3cm}
    \includegraphics[height=0.3\linewidth]{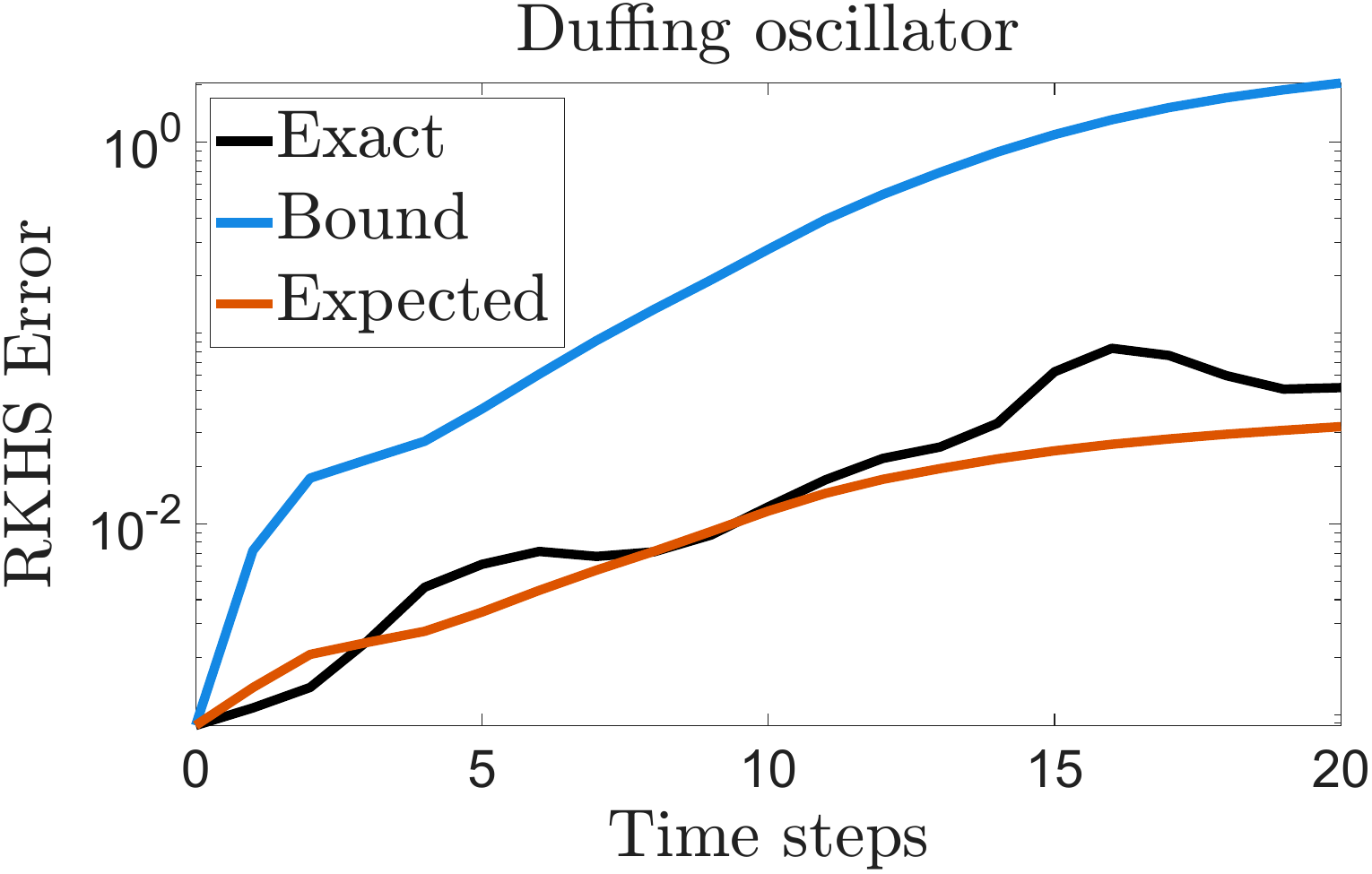}\hfill
    \includegraphics[height=0.3\linewidth]{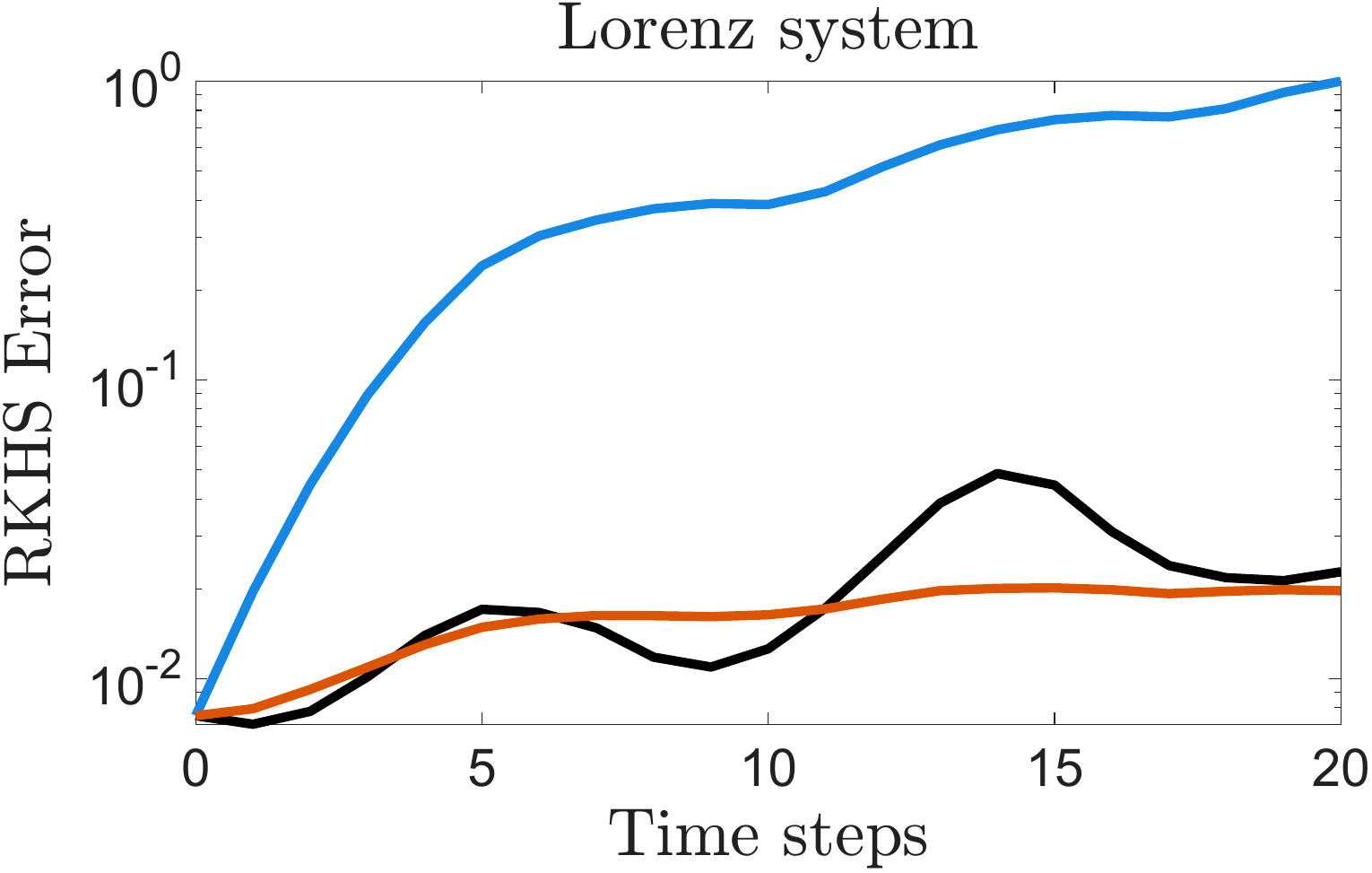}
    \caption{$L^2$ (top) and RKHS (bottom) norm forecast errors, error bounds and expected bounds for the Duffing oscillator (left) and Lorenz system (right), along with variance error bars including three standard deviations.}
    \label{fig:exp_er}
\end{figure}

Similarly, in the derivation of \cref{eqn:pfmd_pointwise_bound} for RKHS we needed to bound the expression $|\langle g,r\rangle|$ where
$
r=(\koop^*)^n\mathfrak{K}_{x_0}-\mathcal{P}_{\mathcal{V}_N}^*(\mathcal{P}_{\mathcal{V}_N}\koop^*\mathcal{P}_{\mathcal{V}_N}^*)^nc.
$
We applied the Cauchy--Schwarz inequality to obtain $|\langle g,r\rangle|\leq \|g\|\|r\|$ and then bounded $\|r\|$ in terms of the initialization and projection errors. This represents the worst-case scenario for the error. Instead, we utilize an expectation-based approach. Given $g\in\mathcal{H}$, we define
\[\setlength\abovedisplayskip{6pt}\setlength\belowdisplayskip{6pt}
\mathbb{E}_{\mathscr{C}}\left[g\right]=\mathbb{E}\left[{|\langle g,\zeta\rangle |}\big/{\|\zeta\|}\right],\quad \zeta\sim\mathcal{GP}(0,\mathscr{C}).
\]
Using the data-driven covariance operator, $\zeta$ is approximated by a linear combination of kernel functions $\mathcal{P}_{\mathcal{V}_N}^*\mathcal{P}_{\mathcal{V}_N}\zeta$, and so $\langle g,\mathcal{P}_{\mathcal{V}_N}^*\mathcal{P}_{\mathcal{V}_N}\zeta\rangle$ is a linear combination of pointwise evaluations of $g$. We approximate
\begin{equation}\setlength\abovedisplayskip{6pt}\setlength\belowdisplayskip{6pt}\label{monte_carlo_g}
\mathbb{E}_{\mathscr{C}}\left[g\right]\approx \cfrac{1}{P}\sum_{k=1}^P{|\langle g,\mathcal{P}_{\mathcal{V}_N}^*\mathcal{P}_{\mathcal{V}_N}{\zeta}_k\rangle|}\big/{\|\mathbf{G}^{1/2}\mathbf{w}_k\|}.
\end{equation}
We then take $|\langle g,r\rangle|\approx \mathbb{E}_{\mathscr{C}}[g]\|r\|$ in \cref{eqn:pfmd_pointwise_bound}. \cref{alg:expected_error_all_kernel} combines all of \cref{eqn:exp_op_norm,eqn:exp_tri_eq,monte_carlo_g} with \cref{rkhs_error_bound}.

\subsection{Numerical examples}

We return to the Duffing and Lorenz systems to illustrate how \cref{alg:expected_error_all_l2,alg:expected_error_all_kernel} improve upon the error bounds computed in \cref{sec:erb_numex}. The expected $L^2$-errors for the predicted evolution of the observable $g(x)=x_1$ are shown in the top panels of \cref{fig:exp_er}, alongside the strict error bounds from \cref{sec:erb_numex} and the exact $L^2$-errors; we use the SVD rather than PAD truncated systems, as the PAD bounds were essentially tight. We take $\lambda_j^2=e^{-j/10^3}$ and use an orthonormal basis obtained from the SVD of the Gram matrix. Since the Koopman operator for the Lorenz system is unitary, $\mathbb{E}_{\mathscr{C}}[\koop]=\|\koop\|=1$ and so the error bounds and expected errors are the same. For the Duffing oscillator, although the expected error is no longer a strict upper bound, it provides an improvement in tracking the true error compared to the original bounds.

The same behavior is observed in the RKHS setting, as shown in the bottom panels of \cref{fig:exp_er}. We plot the exact RKHS-norm errors, the corresponding error bounds, and the expected errors using \cref{alg:expected_error_all_kernel}. As in the $L^2$ case, the expected errors yield a substantial improvement. 

Finally, we return to the cavity flow example that motivated this section. In addition to the strict error bounds shown in \cref{fig:erb_cavity}, we compute expected errors using \cref{alg:expected_error_all_l2}. As seen in \cref{fig:cavity_flow_expected_errors}, the $L^2$ expected error remains close to the true error over the entire time horizon, representing a substantial improvement over the strict bounds. This is particularly notable given the high dimensionality and complex chaotic dynamics of the flow. The expected error formulation extends straightforwardly to computing variances of the relevant terms; as such, we plot variance error bars for the expected errors encompassing three standard deviations. We see that the exact trajectory lies inside these expected error bars for the majority of the time.

\begin{figure}[t]
    \centering
    \includegraphics[height=0.3\linewidth]{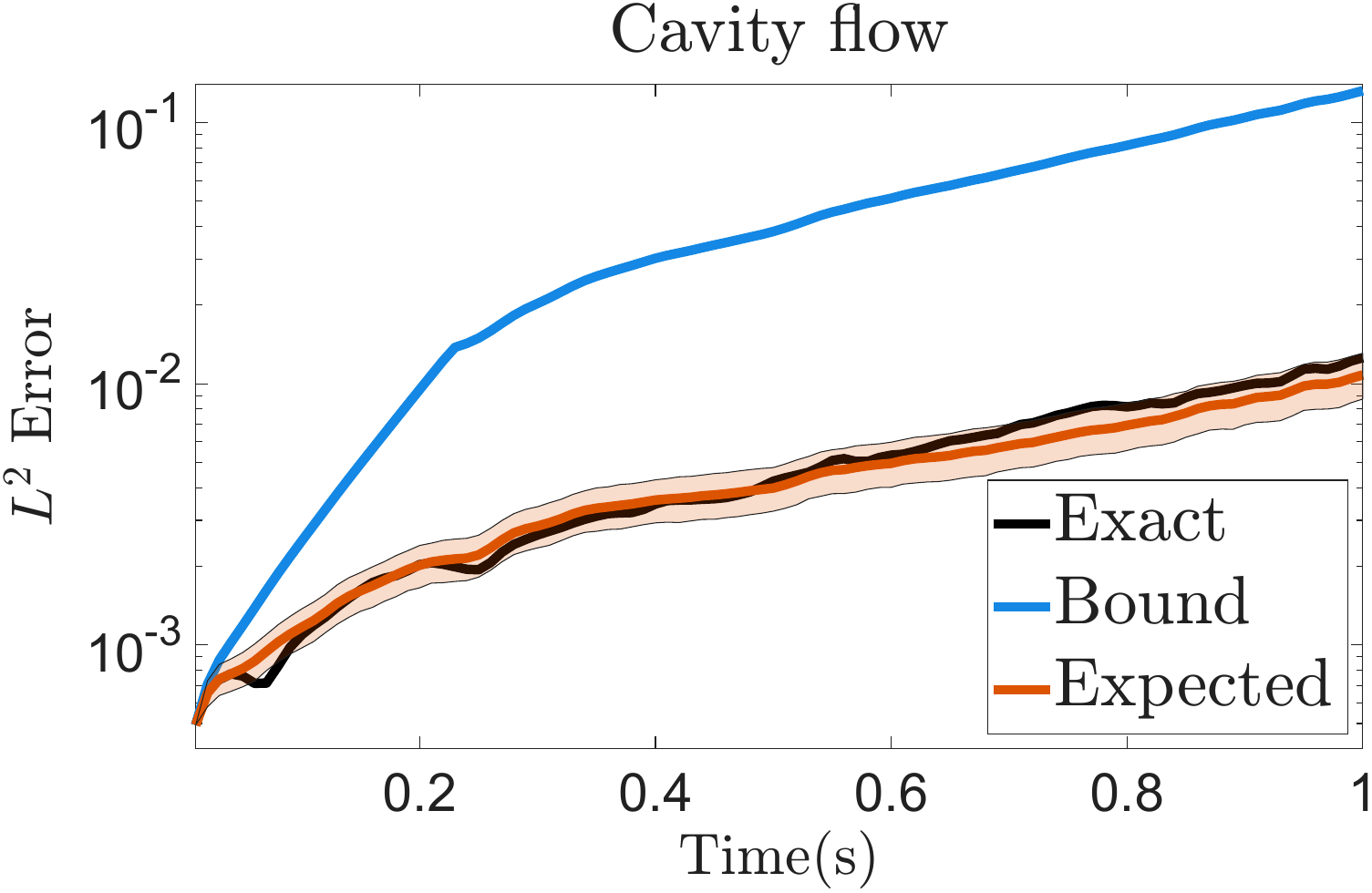}
    \caption{Exact error, error bounds, expected error and three standard deviation error bars for the cavity flow system. The expected error tracks the error much more closely than the strict bound.}
    \label{fig:cavity_flow_expected_errors}
\end{figure}

\section{Example: The Pluto--Charon system}\label{sec:num_examples}

As a more challenging example, we consider the Pluto--Charon system, which consists of Pluto and its moons Charon, Styx, Nix, Hydra, and Kerberos. Charon has mass an eighth of Pluto's, so they form a binary system, giving rise to complex dynamics \cite{showalter2015resonant}. The New Horizons spacecraft flew past the system in July 2015, collecting extensive observational data \cite{stern2018pluto}. To model these dynamics, we use data from \cite{giorgini2025website}, which fits numerically integrated trajectories of planets, satellites, and small bodies in the Solar System to observational measurements. The dataset comprises positions and velocities of the six bodies in the 
$x$, $y$ and $z$ directions relative to the Solar System barycentre, yielding a $36$-dimensional system. Our goal is to model simultaneously the dynamics of Pluto and its five moons.

\subsection{Computation of error bounds}

We collect 80 years of daily trajectory snapshots of the Pluto--Charon system to train the model. Each component of the data is normalized using its mean and standard deviation. Using kernelized EDMD, we construct an approximation of the Koopman operator with a kernel given by a product of scaled Mat\'ern kernels for each of the six bodies. The resulting operator advances the system in $10$-day steps. We test the model by extrapolating over a further $72$ unseen $10$-day intervals. For the predicted trajectories, we compute RKHS-norm and pointwise error bounds using \cref{alg:specrkhs_pfmd}, and expected errors using \cref{alg:expected_error_all_kernel}.

{\begin{figure}[t]
    \centering
        \raisebox{-0.5\height}{\includegraphics[height=0.25\linewidth]{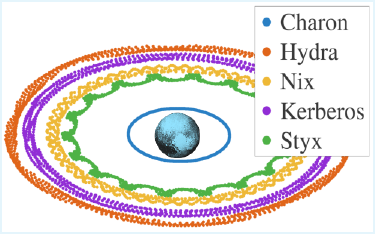}}\hfill
        \raisebox{-0.5\height}{\includegraphics[height=0.3\linewidth]{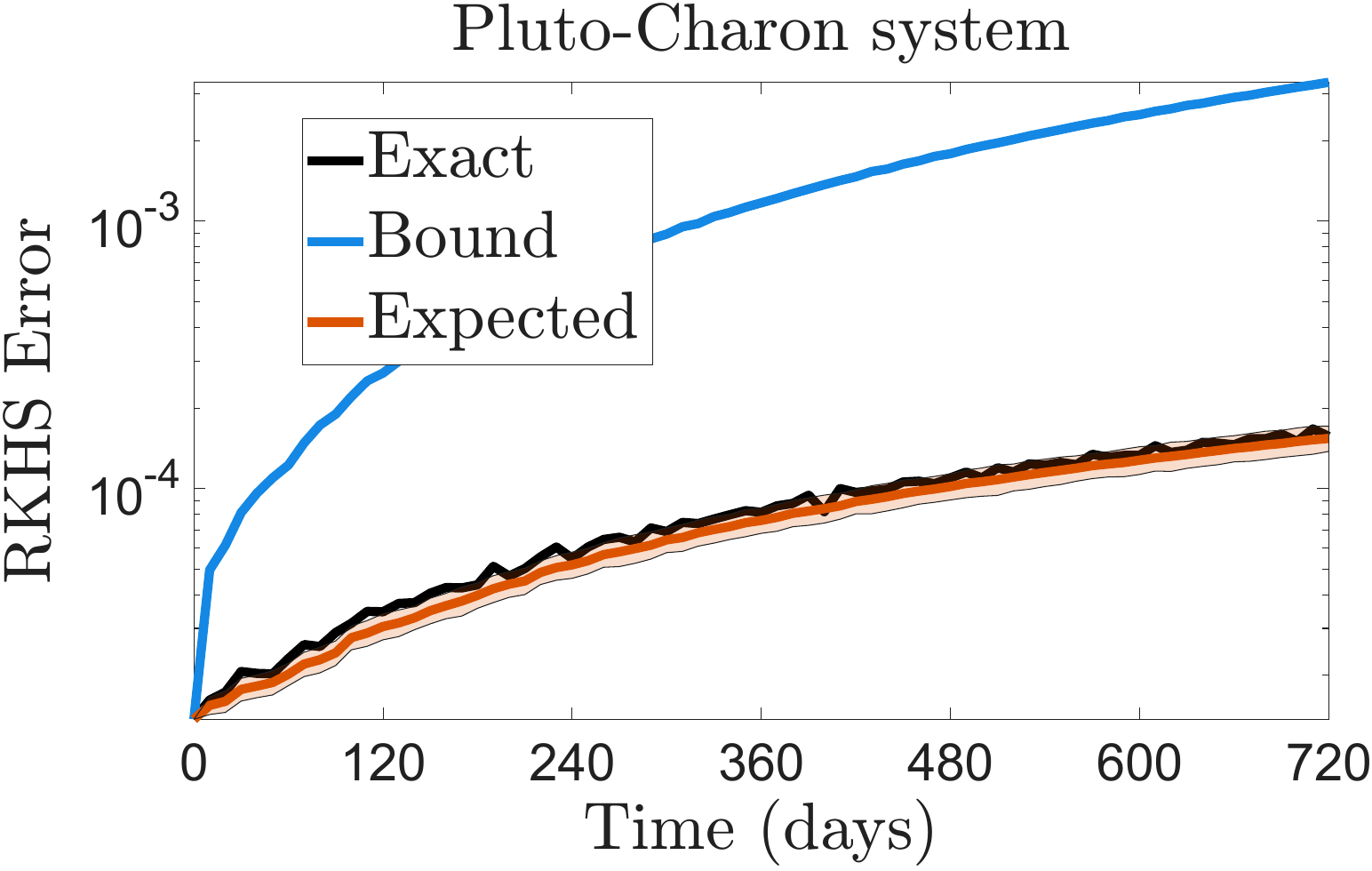}}
    \caption{Left: Orbits of moons in the Pluto--Charon system with Pluto at the origin. Right: RKHS-norm errors, strict error bounds, expected errors and three standard deviation error bars for the Pluto--Charon system. While the strict bounds overestimate the true error, the expected errors track it closely.}
    \label{fig:pluto_charon_system}
\end{figure}}

The results are shown in \cref{fig:pluto_charon_system,fig:pluto_charon_pred}. The right panel shows the RKHS errors for the evolution of $\mathfrak{K}_{x_0}$, where $x_0$ is the starting point of the extrapolated trajectory, together with the associated error bounds and expected errors. While the strict bound is an order of magnitude larger, the expected error closely tracks the true error. In \cref{fig:pluto_charon_pred}, we examine \emph{pointwise} errors for selected components of the dynamics, along with the corresponding pointwise error bounds and expected errors. The gap between exact (in dashed black) and expected pointwise errors (in dashed red) is larger than for the RKHS-norm errors, since the expected pointwise errors are obtained by scaling the expected norm errors by the constant $\mathbb{E}_{\mathscr{C}}[g]$, whereas the exact pointwise errors exhibit greater variability.

\begin{figure}[t]
    \centering
    \begin{minipage}{0.05\linewidth}
        \centering
        \rotatebox{90}{\textbf{Position}}
    \end{minipage}
    \begin{minipage}[c]{0.06\linewidth}
    \centering
        \vspace{0.3cm}
    \rotatebox{90}{\small{Predictions}}

    \vspace{0.6cm}

    \rotatebox{90}{\small{Rel. Ptwise Err.}}

    \end{minipage}
    \begin{minipage}{0.85\linewidth}
        \centering
        \includegraphics[width=0.32\linewidth]{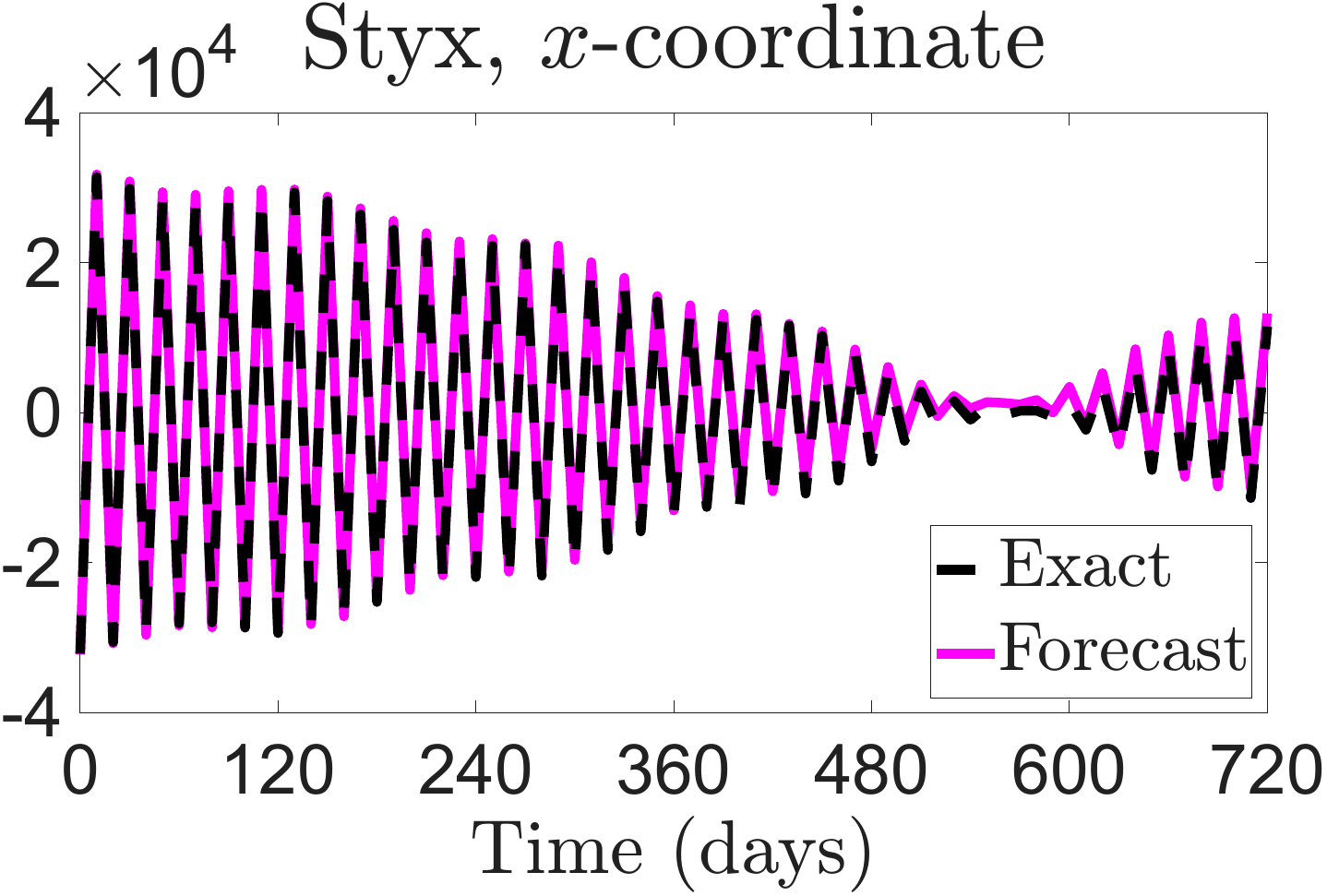}
        \includegraphics[width=0.32\linewidth]{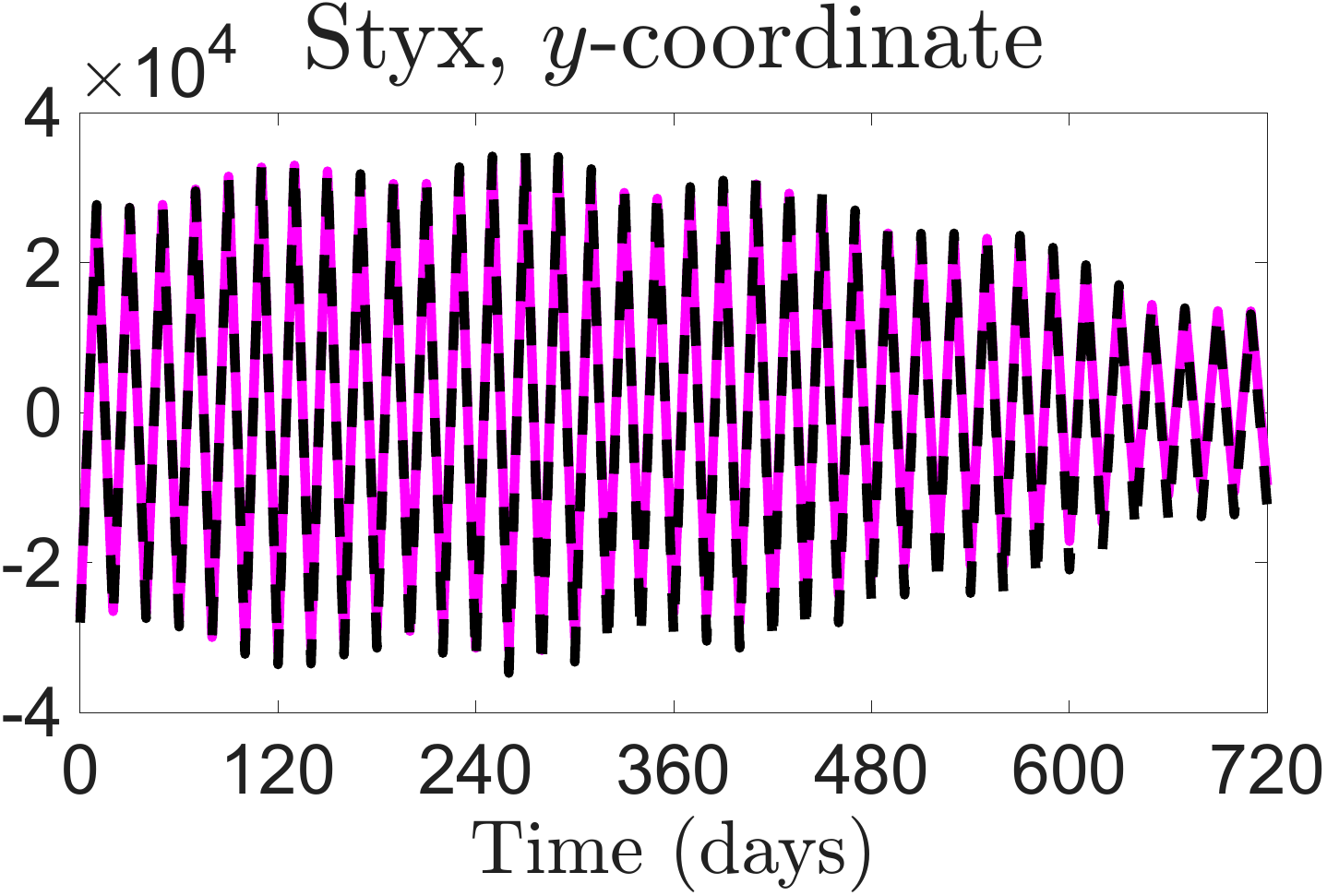}
        \includegraphics[width=0.32\linewidth]{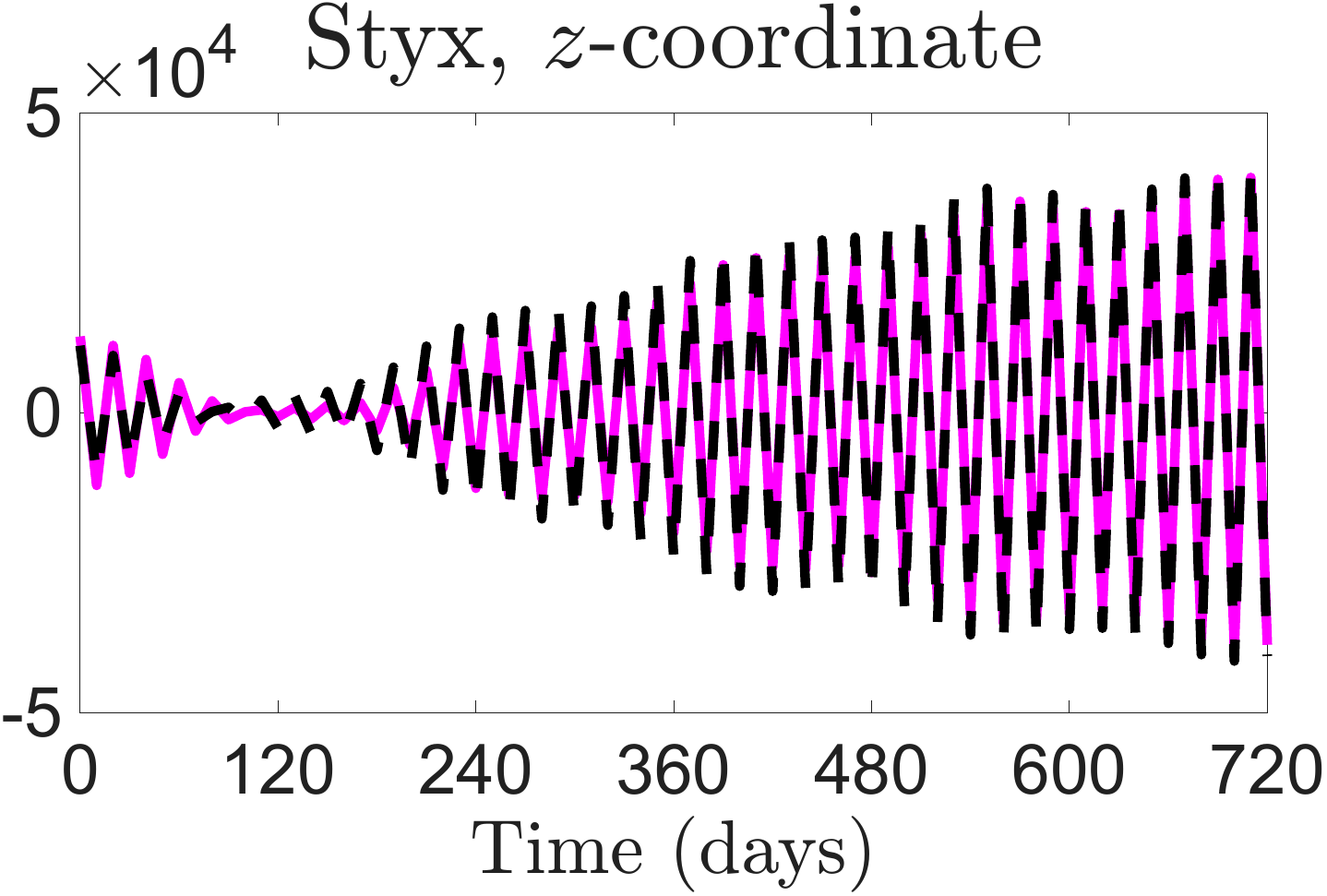}\\
        \vspace{0.2cm}
        \includegraphics[width=0.32\linewidth]{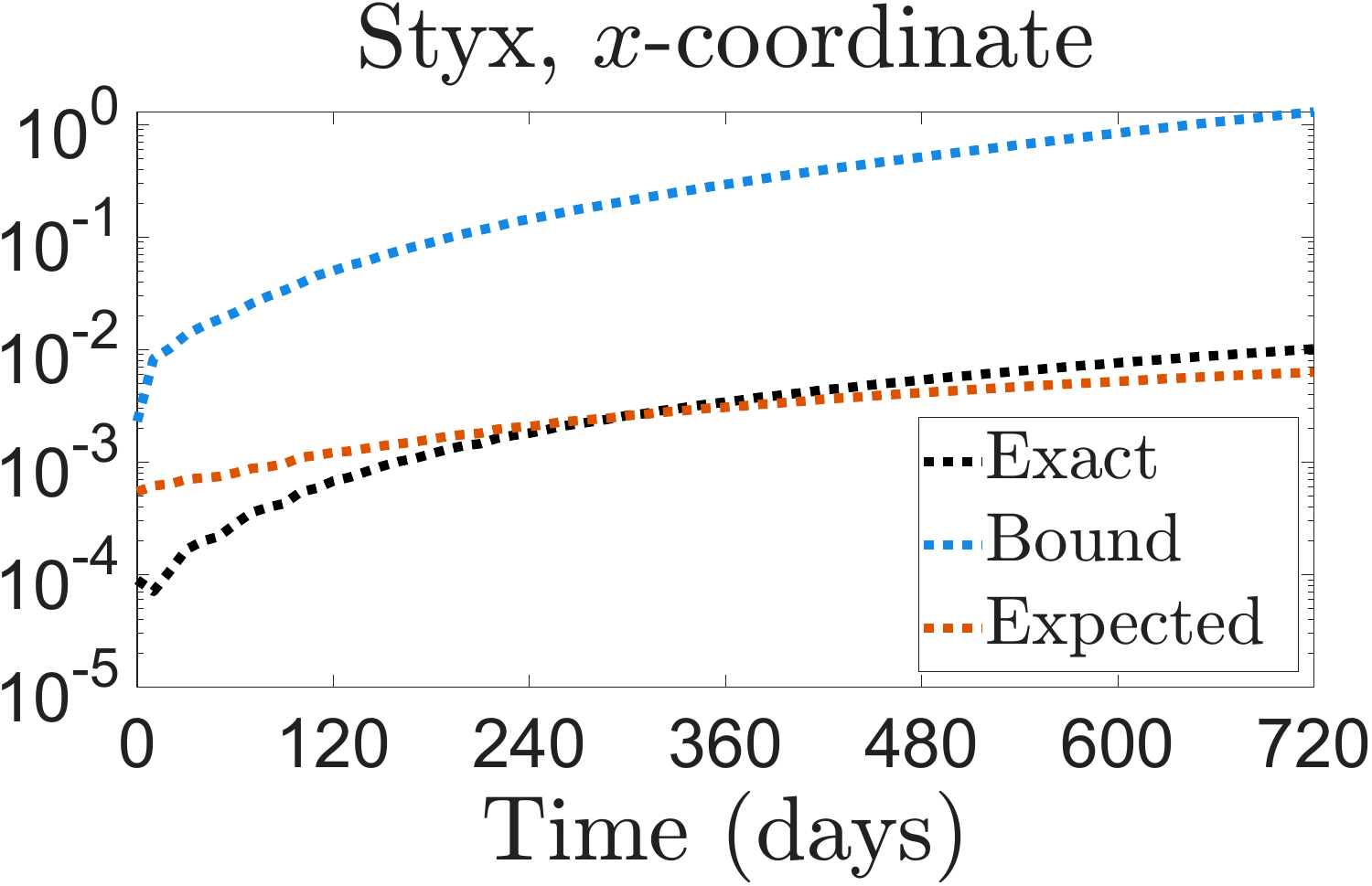}
        \includegraphics[width=0.32\linewidth]{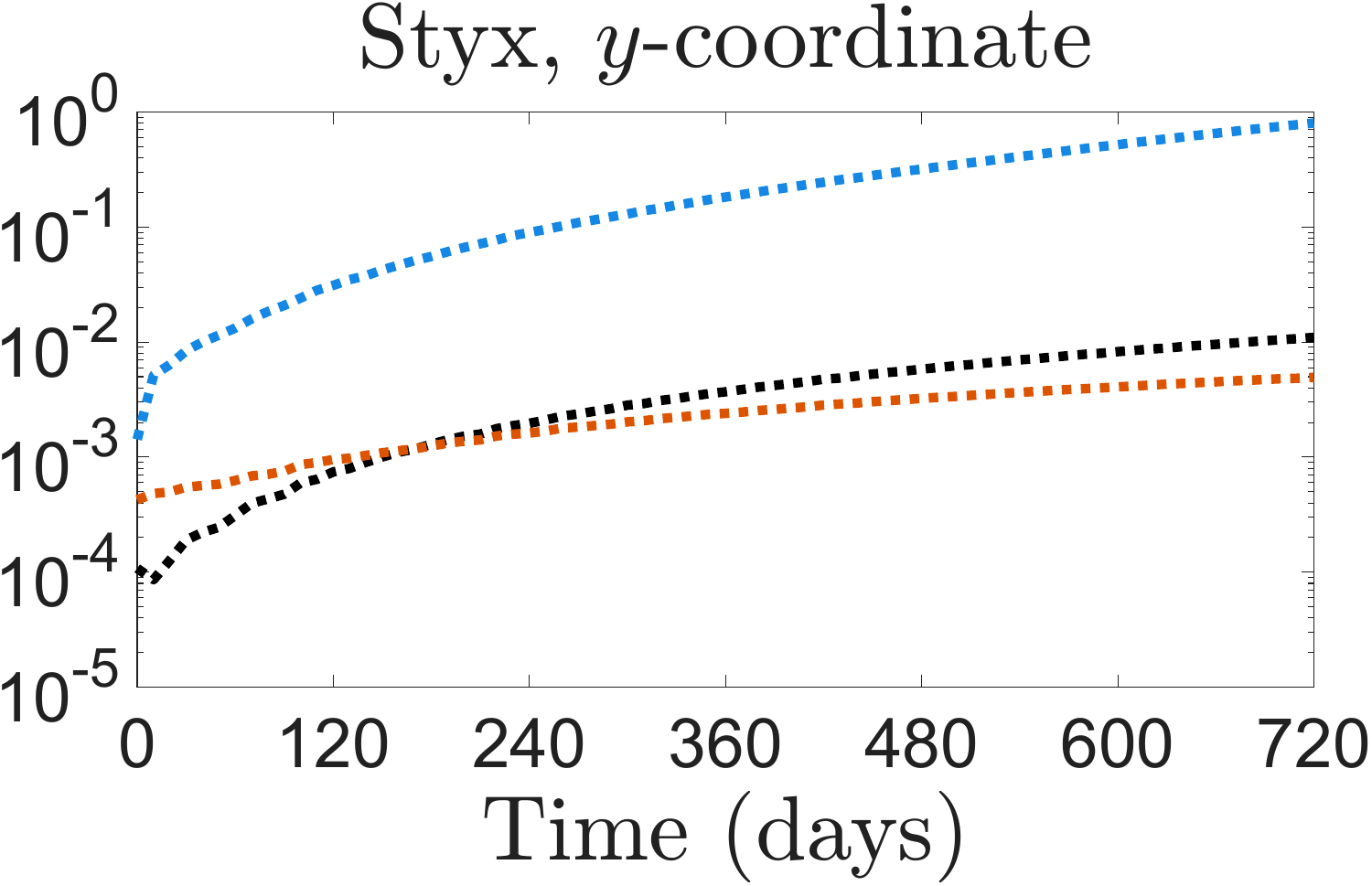}
        \includegraphics[width=0.32\linewidth]{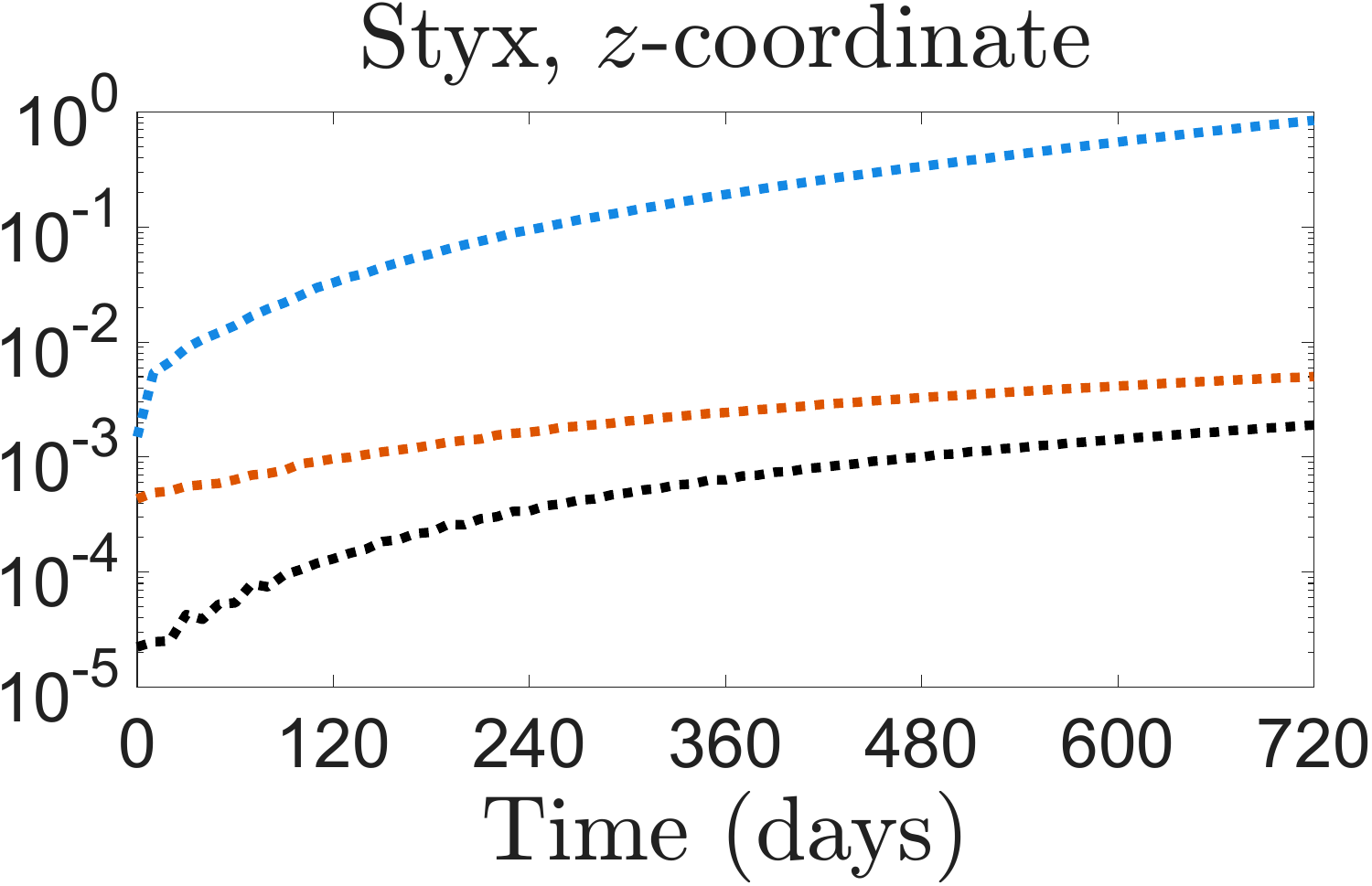}
    \end{minipage}

    \vspace{0.1cm}
\noindent\makebox[\linewidth]{%
  \leaders\hbox{\rule{2mm}{0.4pt}\hspace{1.5mm}}\hfill
}

    \vspace{0.1cm}
    \begin{minipage}{0.05\linewidth}
        \centering
        \rotatebox{90}{\textbf{Velocity}}
    \end{minipage}
    \begin{minipage}[c]{0.06\linewidth}
    \centering
    \vspace{0.3cm}
    \rotatebox{90}{\small{Predictions}}

    \vspace{0.6cm}

    \rotatebox{90}{\small{Rel. Ptwise Err.}}
    \end{minipage}
    \begin{minipage}{0.85\linewidth}
        \centering
        \includegraphics[width=0.32\linewidth]{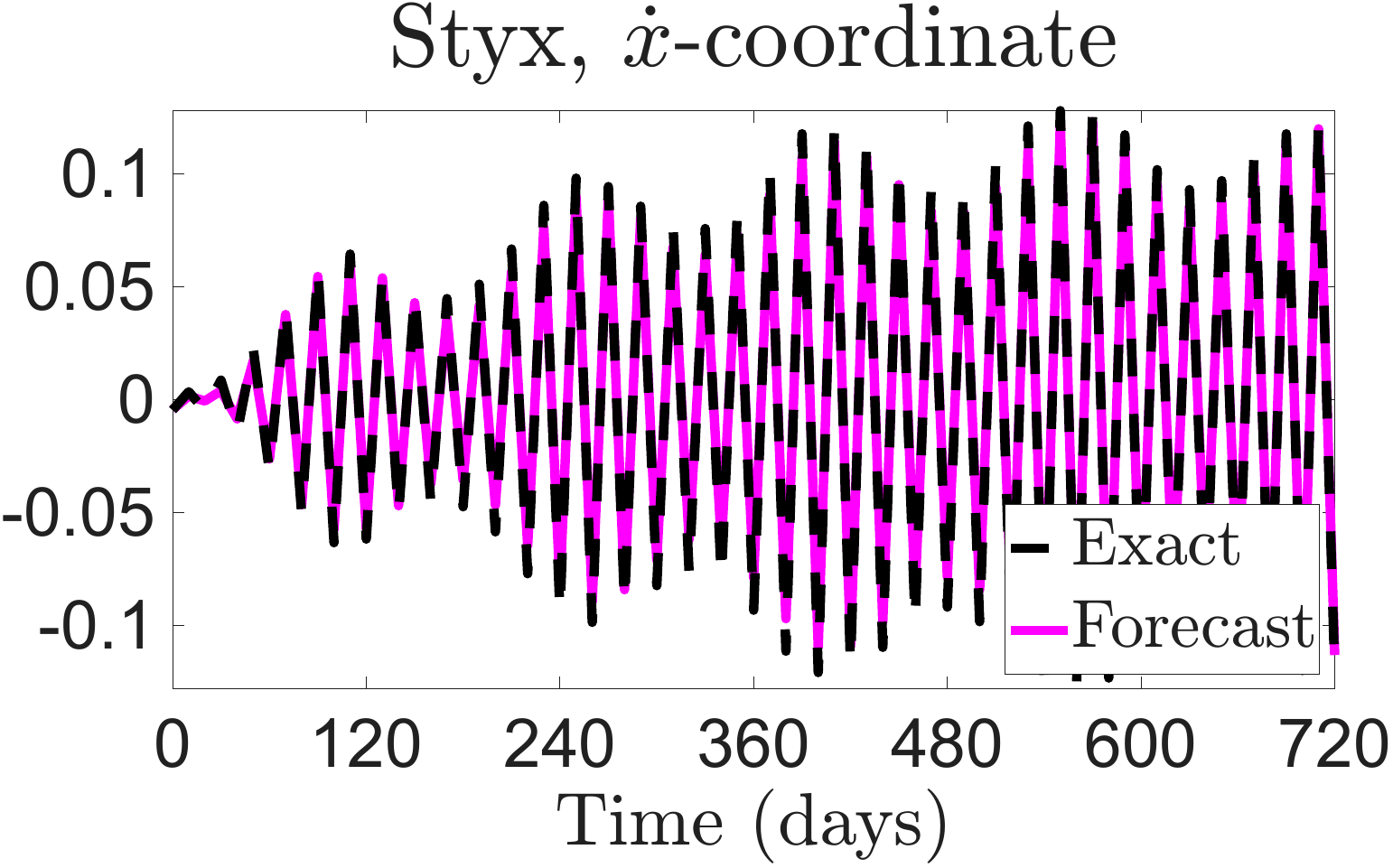}
        \includegraphics[width=0.32\linewidth]{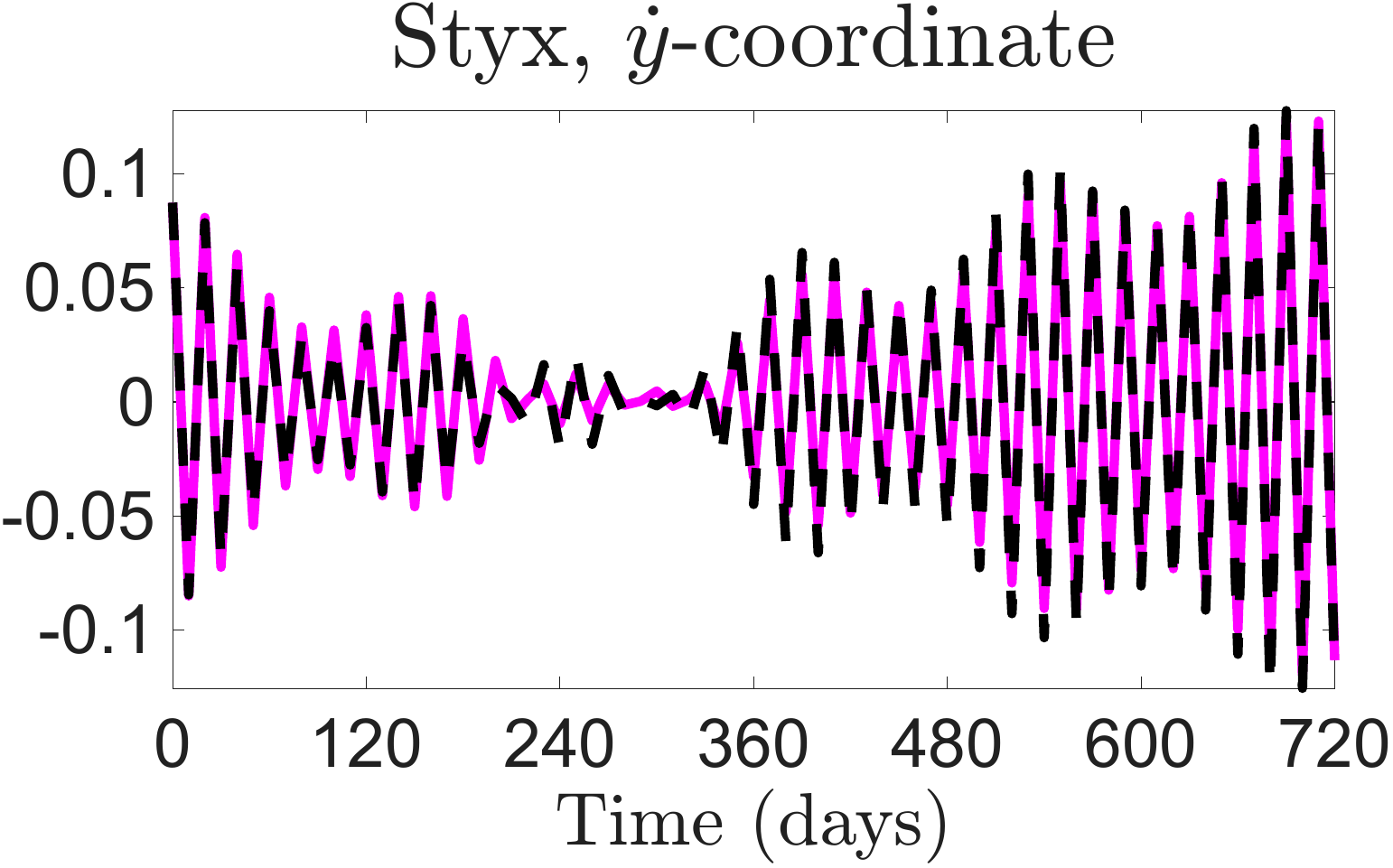}
        \includegraphics[width=0.32\linewidth]{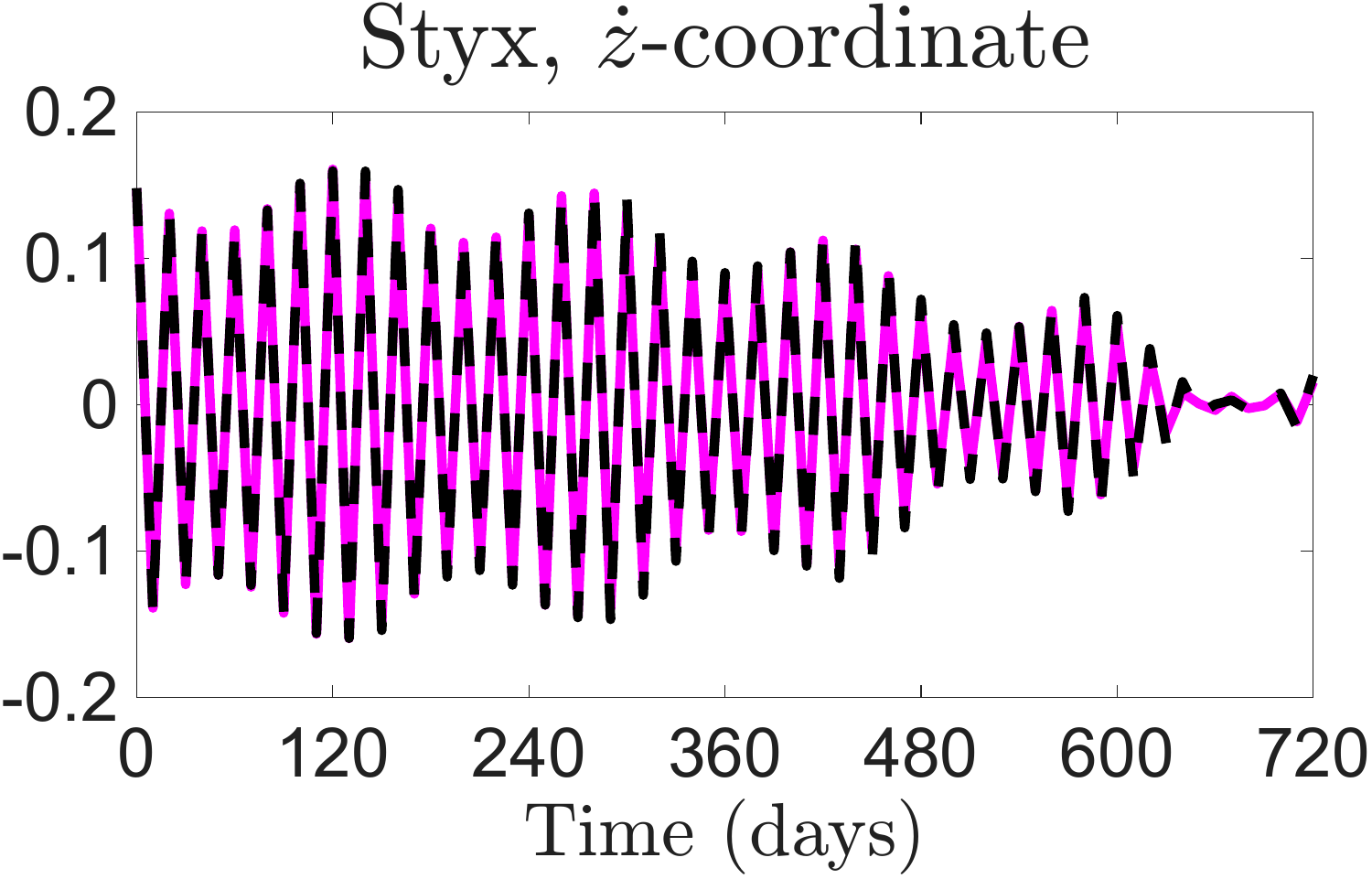}\\
        \vspace{0.2cm}
        \includegraphics[width=0.32\linewidth]{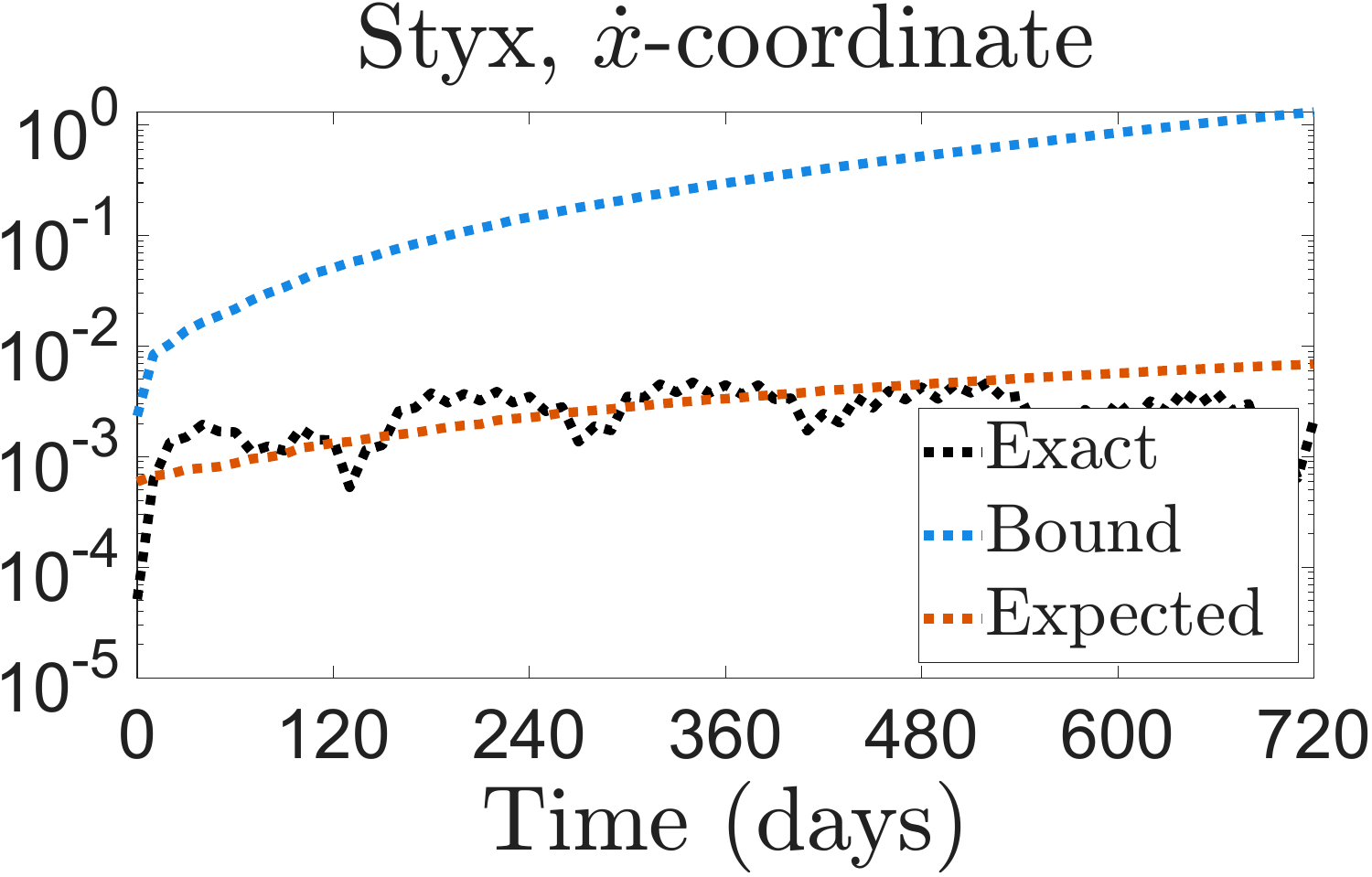}
        \includegraphics[width=0.32\linewidth]{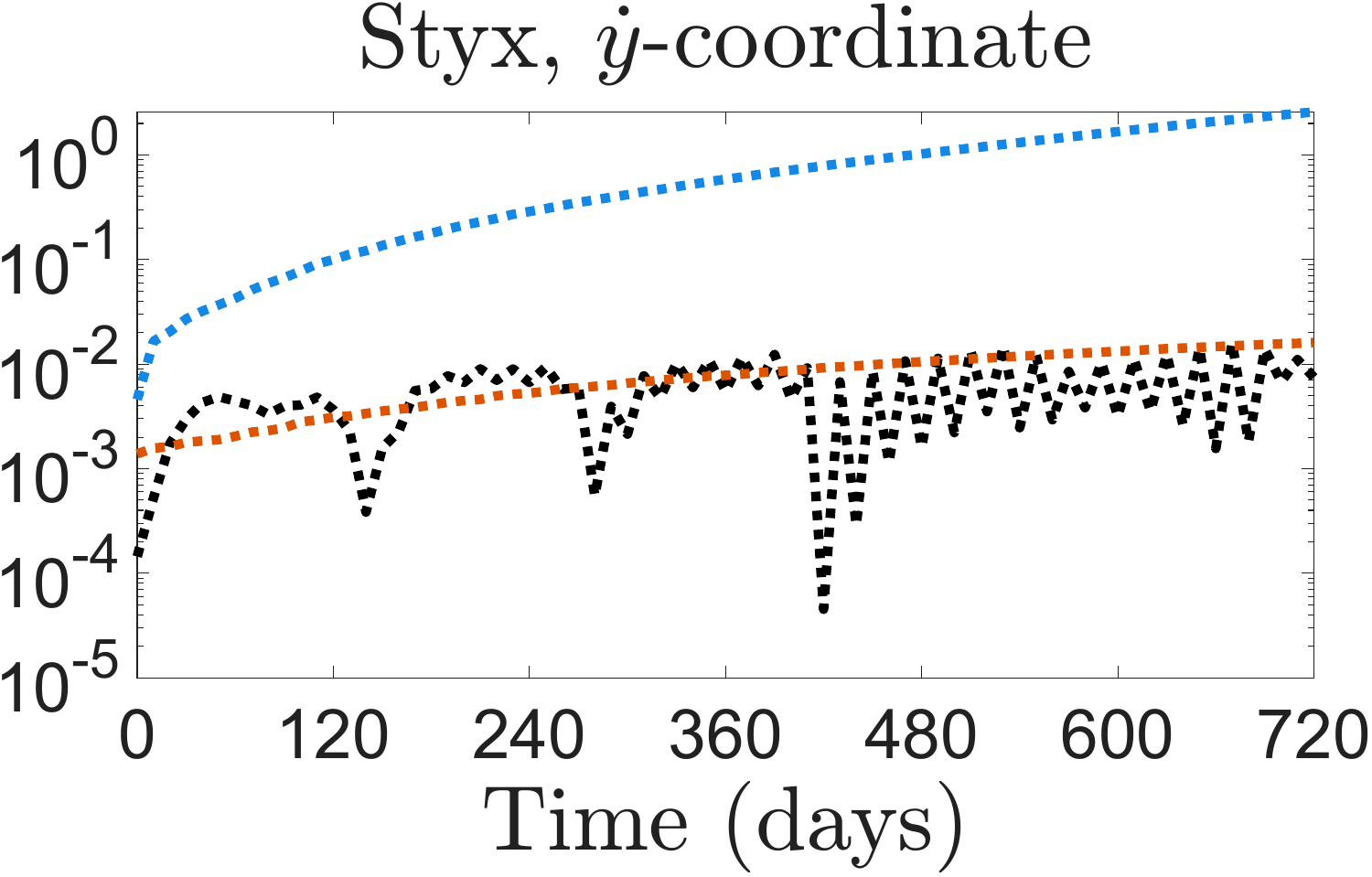}
        \includegraphics[width=0.32\linewidth]{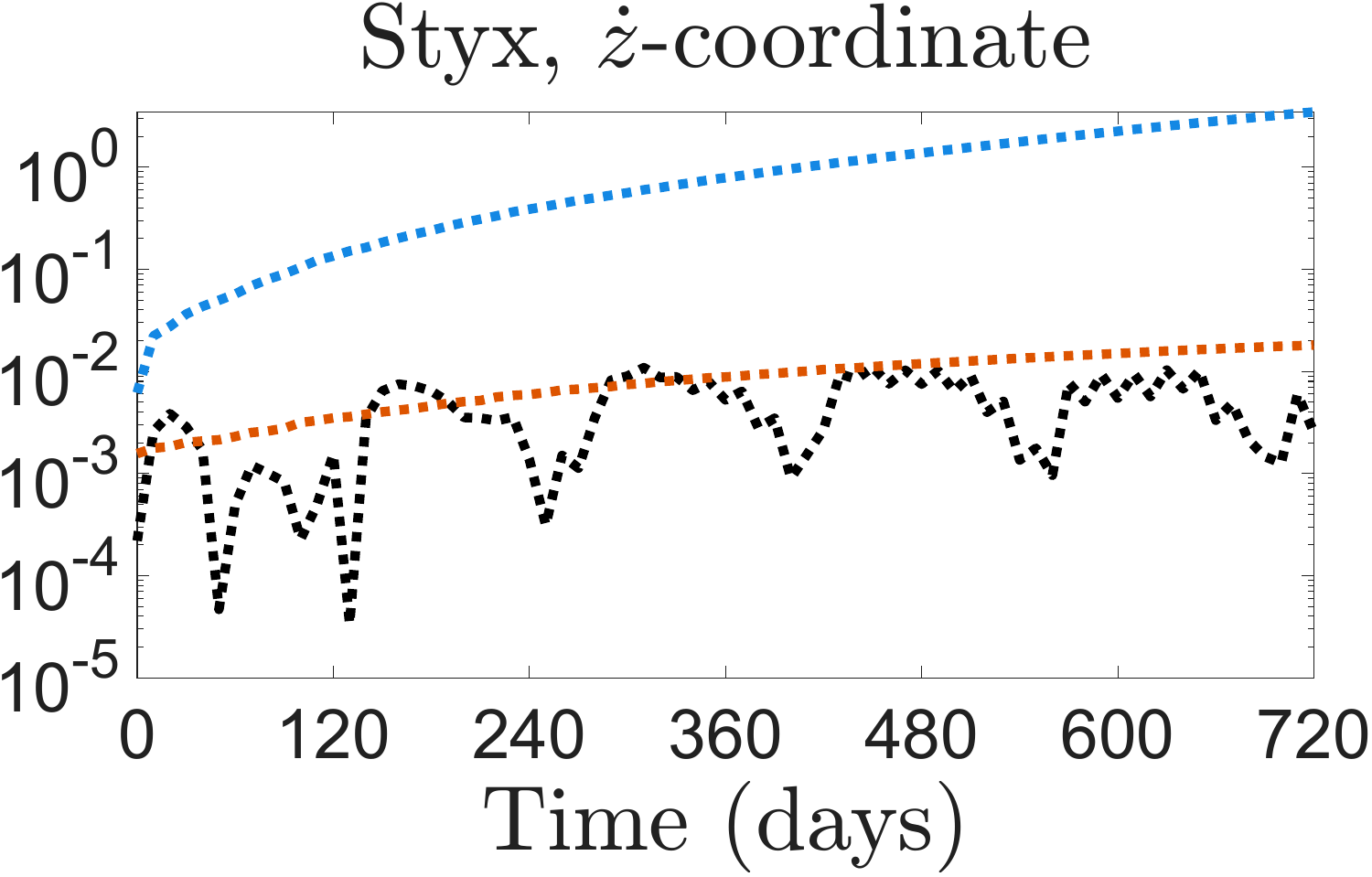}
    \end{minipage}

    \caption{First row: Exact and predicted trajectories with kEDMD for the position of Styx (smallest body in the Pluto--Charon system) relative to Pluto. Second row: Exact relative pointwise errors together with the corresponding error bounds and expected errors. Third and fourth rows: As in the first and second rows, but for the velocities of Styx relative to Pluto.}
    \label{fig:pluto_charon_pred}
\end{figure}

\subsection{Dictionary learning via error bounds}

A key challenge in Koopman--based methods is the choice of dictionary or kernel \cite{colbrook2023multiverse}. Many recent approaches address this using machine-learning techniques to learn dictionaries for dynamic mode decomposition \cite{lusch2018deep,yeung2019learning,li2017extended,Takeishi2017nips,Mardt2018natcomm,otto2019linearly}. Given a set of hyperparameters (here, the dictionary functions or kernel), a Koopman model is constructed from training data, and its performance on a validation set is used to refine those hyperparameters. After multiple tuning iterations, the unbiased performance and generalizability of the resulting model are assessed on a held-out test dataset.

The error bounds developed in this paper enable an improved workflow. Rather than requiring a separate validation dataset to assess model performance, hyperparameters can be optimized directly by minimizing the error bounds computed from the training data. The test dataset then retains its usual role for unbiased evaluation. This leads to more efficient use of data, particularly when data are limited, and provides a principled approach to dictionary selection.

We apply this methodology to optimize hyperparameters for EDMD applied to the Pluto--Charon dataset. The training dataset consists of $10$ years of data, with the subsequent $10$ years reserved for testing; no separate validation dataset is required. We apply EDMD using a dictionary of $200$ RBFs of the form
\[\setlength\abovedisplayskip{6pt}\setlength\belowdisplayskip{6pt}
\psi_i^{(s,\nu)}(x)=f^{(s,\nu)}(\|x-c_i\|),\quad\text{ for }\quad  f^{(s,\nu)}(r)=(sr)^\nu K_{\nu}(sr),
\]
where $K_\nu$ is the modified Bessel function of the second kind of order $\nu$, and $\{c_i\}_{i=1}^{200}$ are fixed centers sampled uniformly at random from the bounding box of the trajectory. We aim to learn optimal values of $s,\nu\in\mathbb{R}^+$. To improve quadrature convergence, before normalizing the data we shift it so that the barycentre of the Pluto--Charon system lies at the origin. This requires only the masses of the bodies and ensures that trajectories remain close to the centers of the chosen radial basis functions. 

We approximate the one-step Koopman operator and, for numerical stability, construct a reduced basis via a truncated SVD retaining $r=50$ modes. The hyperparameters $s,\nu\in\mathbb{R}^{+}$ are chosen to minimize the average error bounds produced by \cref{alg:resDMD_KMD} for the components of the full-state observable $g_i(x)=[x]_i$, $i=1,\dots,36$, averaged over $21$ time steps. Since the full-state observable is not contained in the dictionary, initialization errors are included as described in \cref{sec_kmd_eb}. Specifically, we minimize the loss
\[\setlength\abovedisplayskip{6pt}\setlength\belowdisplayskip{6pt}
L(s,\nu)=\frac{1}{21}\sum_{n=0}^{20}\left(\sum_{i=1}^{36}E_n(g_i,s,\nu)\right),
\]
where
\[\setlength\abovedisplayskip{6pt}\setlength\belowdisplayskip{6pt}
E_n(g_i,s,\nu)=\|\mathbf{K}^n\|\|g_i-\mathbf{\Psi}\mathbf{g}_i\|+\sum_{j=0}^{n-1} 
\|\mathbf{K}^j\|
\sqrt{(\mathbf{K}^{n-j-1}\mathbf{g}_i)^*(\mathbf{L}-\mathbf{K}^*\mathbf{K})(\mathbf{K}^{n-j-1}\mathbf{g}_i)}.
\]
Here, $\mathbf{K}$, $\mathbf{L}$ are the EDMD and ResDMD matrices, depending on $s$ and $\nu$, written with respect to the SVD basis (so that $\mathbf{G}=\mathbf{I}$ and $\mathbf{A}=\mathbf{K}$).

We solve the resulting optimization problem using the Adam stochastic optimization algorithm \cite{kingma2014adam}, implemented in PyTorch. The learning rate is set to $10^{-2}$, and we run $500$ Adam epochs starting from $s=1/5$ and $\nu=1$. The results are shown in \cref{fig:pluto_charon_hyperparameter}. We plot the training-set error bounds (which constitute the loss being minimized) together with the exact $L^2$ and pointwise errors of the predicted trajectories. To assess generalizability, the exact errors are computed on the test dataset. The trained dictionary has much smaller error bounds, showing clear improvement. Moreover, as the optimization reduces the error bounds, the exact $L^2$ errors (which remain bounded above by the error bounds) and the pointwise errors decrease in tandem, indicating that the bounds accurately capture the true system error. Hence, by training on just the a posteriori error bounds we have significantly improved performance of the model in terms of its exact prediction errors.

\begin{figure}[t]
    \centering
        \includegraphics[height=0.3\linewidth]{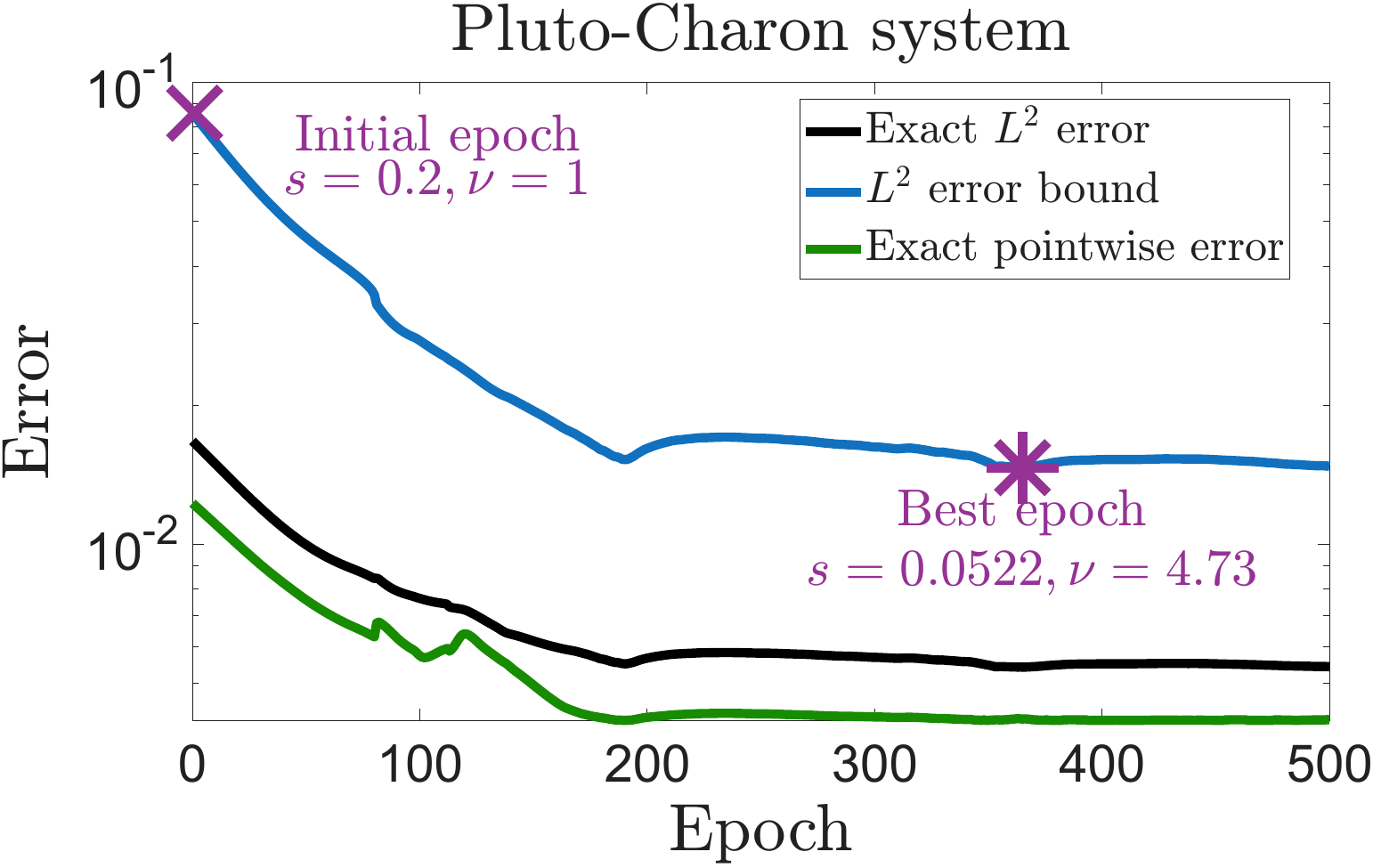}\hfill
    \includegraphics[height=0.3\linewidth]{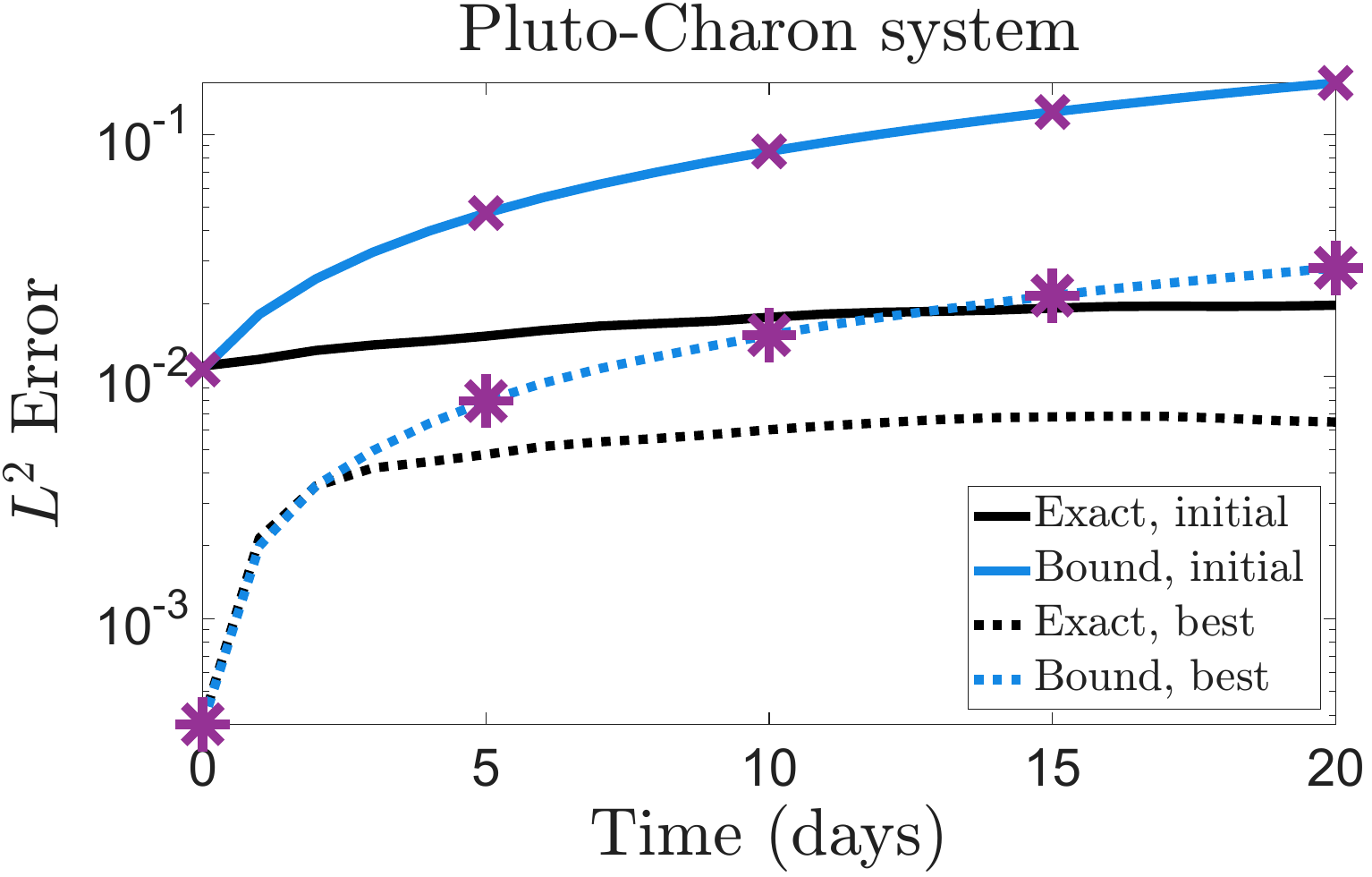}
    \caption{Left: Averaged error bounds, exact $L^2$ errors, and pointwise errors for the Pluto--Charon system over $500$ iterations of the Adam stochastic optimization algorithm. Right: Comparison of exact $L^2$ forecast errors and error bounds over $20$ time steps for the initial parameter choice (marked by a cross) and the best parameters found by Adam (marked by a star).}
    \label{fig:pluto_charon_hyperparameter}
\end{figure}

\section{Conclusion}

This work developed a suite of theoretically grounded algorithms that produce computable error bounds for Koopman--based approximations of dynamical systems and measure Koopman subspace invariance. These certificates lead naturally to the Principal Angle Decomposition (PAD), a dimensionality-reduction procedure that complements the SVD by preferentially retaining directions that are (approximately) invariant under the Koopman operator. Across high-dimensional and chaotic examples, the resulting upper bounds and expected error surrogates track the observed prediction errors for Koopman and Perron--Frobenius mode decompositions. Moreover, because the bounds are explicit functions of the chosen feature space, they can be used to guide dictionary design and refinement directly from the available data, providing a principled route to feature selection when data are limited. 

More broadly, the proposed methodology strengthens the reliability of data-driven modeling by attaching quantitative uncertainty information to Koopman outputs without requiring additional measurements beyond standard snapshot data. Rigorous bounds support decision-making when guarantees are needed, while expected error estimates offer a practical complement when conservative certification is unavoidable. A limitation of the expected error methodology described in this paper is the assumption that the relevant functions are distributed according to a Gaussian process. In general, this is not the case. Future work may involve methods to measure the distributional distance, or to use Gaussian process regression to optimise the kernel used for the expected errors. Gaussian process regression may also be integrated into dictionary refinement using the error bounds.

Finally, as Koopman--based methods continue to gain traction in control due to their global linearization viewpoint \cite{otto2021koopman,strasser2025overview,brunton2021modern,mauroy2020koopman}, these tools provide a foundation for verified Koopman control, including integration with established workflows such as model predictive control. The expectation-based surrogates may also be of independent interest in machine learning settings where calibrated error prediction and model selection are central.




\linespread{0.97}
\bibliographystyle{jabbrv_siam}
\bibliography{bib_file_koopman} 
\end{document}